\documentclass[11pt]{article}

\usepackage{amssymb}
\usepackage{amsmath}
\usepackage{amsfonts}
\usepackage{enumerate}
\usepackage{enumitem}
\usepackage{hyphenat}

\usepackage{parskip}

\usepackage{tikz}
\usepackage[makeroom]{cancel}
\usepackage{amsthm}
\usepackage{dsfont}

\usepackage{graphicx}

\usepackage{hyperref}

\usepackage{pgfplots}
\pgfplotsset{width=10cm,compat=1.17}

\newcommand{\bN}{ {\mathbb{N} } }
\newcommand{\bQ}{ {\mathbb{Q}} }
\newcommand{\bR}{ {\mathbb{R}} }
\newcommand{\bZ}{ {\mathbb{Z}} }

\newtheorem{theorem}{Theorem}
    \newtheorem{lemma}[theorem]{Lemma}
    
    \newtheorem{corollary}[theorem]{Corollary}

    \newtheorem{manualtheoreminner}{Theorem}
    \newenvironment{manualtheorem}[1]{%
        \renewcommand\themanualtheoreminner{#1}%
        \manualtheoreminner
    }{\endmanualtheoreminner}

\theoremstyle{definition} 
    \newtheorem*{definition}{Definition}

\theoremstyle{remark} 
    \newtheorem*{remark}{Remark}

\DeclareMathOperator*{\argmin}{arg\,min}

\usepackage{titling}
\usepackage{ytableau}

\pretitle{\begin{center}\Large\bfseries}
\posttitle{\par\end{center}\vskip 0.5em}
\preauthor{\begin{center}\Large\ttfamily}
\postauthor{\end{center}}

\title{Grid entropy in last passage percolation --- a superadditive critical exponent approach}
\author{Alexandru Gatea \\ \quad \\ Department of Mathematics, University of Toronto \\ \href{mailto:alex.gatea@mail.utoronto.ca}{alex.gatea@mail.utoronto.ca} }

\begin{document}

\maketitle
\thispagestyle{empty}

\begin{abstract}
Working in the setting of i.i.d. last-passage percolation on $\mathbb{R}^D$ with no assumptions on the underlying edge\hyp{}weight distribution, we arrive at the notion of grid entropy --- a Subadditive Ergodic Theorem limit of the entropies of paths with empirical measures weakly converging to a given target, or equivalently a deterministic critical exponent of canonical order statistics associated with the Levy-Prokhorov metric. This provides a fresh approach to an entropy first developed by Rassoul-Agha and Sepp{\"a}l{\"a}inen as a large deviation rate function of empirical measures along paths. In their 2014 paper \cite{rassoul2014quenched}, variational formulas are developed for the point-to-point/point-to-level Gibbs Free Energies as the convex conjugates of this entropy. We rework these formulas in our new framework and explicitly link our descriptions of grid entropy to theirs. We also improve on a known bound for this entropy by introducing a relative entropy term in the inequality.  Furthermore, we show that the set of measures with finite grid entropy coincides with the deterministic set of limit points of empirical measures studied in a recent paper \cite{bates} by Bates. In addition, we partially answer a directed polymer version of a question of Hoffman which was previously tackled in the zero temperature case by Bates. Our results cover both the point-to-point and point-to-level scenarios.
\end{abstract}

\tableofcontents

\section{Introduction}

The limiting behaviour of empirical measures is of paramount importance in percolation theory. The normalized passage time along a path, what we ultimately care about, is nothing but the identity function integrated against the empirical measure of that path. In First/Last Passage Percolation we study the minimizing/maximizing paths called geodesics.

A major open problem posed by Hoffman  in \cite{american} is whether empirical measures along geodesics in a fixed direction converge weakly. This is partially answered in \cite{bates} for FPP. Bates  proves that the  sets $\mathcal{R}^q, \mathcal{R}$ of limiting distributions in direction $q$, limiting distributions in the direction-free case respectively are deterministic and derives an explicit variational formula for the limit shape of the first passage time as the minimum value of a linear functional over $\mathcal{R}^q$. When the set of minimizers is a singleton, which Bates argues happens for a dense family of edge-weight distributions, it follows that Hoffman's question is answered in the affirmative. The same argument applied to the LPP model yields analogous conclusions.

Generally, there is no known way of computing the limiting distributions along geodesics. \cite{martinAllan} showcases some cutting edge developments for the solvable Exponential LPP on $\bZ^2$ model, including an explicit formula for  weak limits of empirical measures along geodesics. For other recent work on geodesics see 
 \cite{ahlberg}, \cite{janjigianRassoul} and \cite{janjigianShen}.

One way to extend the work in \cite{bates} is to study the more general question of what \emph{entropy} of  empirical measures in a fixed direction or direction-free converge in distribution to a certain target measure. Grid entropy gives an exact, deterministic answer.

In \cite{rassoul2014quenched}, Rassoul-Agha and Sepp{\"a}l{\"a}inen derive this entropy as a large deviation rate function of empirical measures along paths and they provide variational formulas realizing the point-to-point/point-to-level Gibbs Free Energies as the convex conjugates of this mysterious new entropy. This rate function approach suggests that one might be able to derive formulas for this entropy coming from subadditivity, and furthermore that it is related to a critical exponent of paths.

Our main aim is to prove these intuitions correct, working in an LPP setting on $\bZ^D$ which can easily be extended to more general frameworks. Rather than starting from the definitions given in \cite{rassoul2014quenched}, we  define grid entropy as a certain critical exponent of paths and show that it is equivalently described as a superadditive ergodic limit. Moreover, we arrive at variational formulas which are analogues to Rassoul-Agha and Sepp{\"a}l{\"a}inen's and which are in fact positive-temperature analogues of Bates' variational formula. Along the way, we  establish various properties of this entropy, some of which are already noted in or follow from \cite{rassoul2014quenched} and \cite{bates}, but also some of which are not, such as a partial answer to Hoffman's question in a directed polymer setting and such as an amazing equality between grid entropy, relative entropy and Shannon entropy. We give versions of these results for both direction-$q$ grid entropy $||(q,\nu)||$ and direction-free grid entropy $||\nu||$ of target measures $\nu$. By concavity and the lattice symmetries,  the direction-free case turns out to simply be the direction-fixed case in the unit direction $(\frac1D,\ldots, \frac1D)$ which maximizes the total number of paths. 

\bigskip

Let us now be more precise. We consider north-east nearest-neighbor paths on $\bZ^D$, and work on $\bR^D$ by taking coordinate-wise floors. We follow a similar approach to \cite{bates}, in that we couple our i.i.d. edge weights $\tau_e$ to i.i.d. edge labels Unif[0,1] random variables $U_e$ via a measurable function $\tau: [0,1] \rightarrow \bR$ satisfying $\tau_e = \tau(U_e)$. This lets us work on [0,1] at no additional cost, as we can lift everything back to $\bR$ via the pushforward of $\tau$. We then interest ourselves with \emph{how many} empirical measures $\frac1n \mu_{\pi} = \frac1n \sum \limits_{e\in \pi} \delta_{U_e}$ for paths $\vec{0} \rightarrow \lfloor nq \rfloor$ converge to some given target measure $\nu$ in $\mathcal{M}_+$, the set of finite non-negative Borel measures on [0,1]. We may keep the direction $q$ fixed, or we may vary $q$ over all points in $\bR^D$ with the same 1-norm $||q||_1$.

To perform the counting, we consider the order statistics of the distance of the paths' empirical measures $\frac1n \mu_{\pi}$ to the target $\nu$, where distance is measured via the Levy-Prokhorov metric $\rho$, which metrizes weak convergence of measures. That is, given a direction $q \in \bR^D$ and a target measure $\nu$, for every $n \in \bN$ we let
\[ \min_{\pi: \vec{0} \rightarrow \lfloor nq \rfloor}^1 \rho\bigg(\frac1n \mu_{\pi}, \nu \bigg) \leq \min_{\pi: \vec{0} \rightarrow \lfloor nq \rfloor}^2   \rho\bigg(\frac1n \mu_{\pi}, \nu \bigg) \leq \ldots \leq  \min_{\pi: \vec{0} \rightarrow \lfloor nq \rfloor}^{\# \pi: \vec{0} \rightarrow \lfloor nq \rfloor}  \rho\bigg(\frac1n \mu_{\pi}, \nu \bigg) \]
denote the order statistics value of $ \rho(\frac1n \mu_{\pi}, \nu)$. It is convenient to define 
\[\min \limits_{\pi: \vec{0} \rightarrow \lfloor nq \rfloor}^j \rho \bigg(\frac1n \mu_{\pi}, \nu \bigg) := +\infty \ \mbox{for} \ j > \# \pi: \vec{0} \rightarrow \lfloor nq \rfloor\]
In the direction-free case, where we count all paths  of a certain scaled length from $\vec{0}$, we let \\ $t := ||\nu||_{TV}$ and similarly define  $\min \limits_{\pi \ \mbox{s.t.} \ |\pi| = \lfloor nt \rfloor }^j \rho(\frac1n \mu_{\pi}, \nu)$ over paths of length $\lfloor nt \rfloor$ anchored at $\vec{0}$.
Of course, these order statistics and the paths corresponding to them are event-dependent.

We then define the grid entropy with respect to the target  $\nu$ and the direction $q$, denoted $||(q, \nu)||$, and the direction-free grid entropy, denoted $||\nu||$, to be the critical exponent where the corresponding order statistics change from converging to 0 a.s. to diverging a.s.:
\begin{equation} \label{firstDef}
\begin{split}
 ||(q, \nu)|| &:= \sup \bigg\{\alpha \geq 0 \ : \lim_{n \rightarrow \infty} \min_{\pi: \vec{0} \rightarrow \lfloor nq \rfloor}^{\lfloor e^{\alpha n} \rfloor}  \rho\bigg(\frac1{n} \mu_{\pi}, \nu \bigg) = 0  \ \mbox{a.s.} \bigg\} \\
 ||\nu|| &:= \sup \bigg\{\alpha \geq 0 \ : \lim_{n \rightarrow \infty} \min_{\pi \ \mbox{s.t.} \ |\pi| = \lfloor nt \rfloor}^{\lfloor e^{\alpha n} \rfloor}  \rho\bigg(\frac1{n} \mu_{\pi}, \nu \bigg) = 0  \ \mbox{a.s.} \bigg\}
 \end{split}
 \end{equation}
where these are defined to be $-\infty$ if the set of $\alpha$'s is empty.
It will turn out that we can replace the limits in this definition  with $\liminf$'s and  get the same quantity.  Observe that these grid entropies lie in $\{-\infty\} \cup [0, H(q)]$, $\{-\infty\} \cup [0, \log D]$ respectively where
\[ H(q): =\sum_{i =1}^D -q_i \log \frac{q_i}{||q||_1}  \]
is  the (Shannon) entropy of the total number of paths in direction $q$, in the sense that
\[(\# \mbox{paths} \ \vec{0} \rightarrow nq) = e^{H(q)n+o(n)}\]

We note that the description \eqref{firstDef}  of grid entropy as the critical exponent of these order statistics has been previously shown to hold only for Bernoulli edge labels (see \cite{carmona2010directed}).

Over the course of this paper we establish two other equivalent definitions of grid entropy which avoid the annoyance of dealing with these event-dependent orderings of the paths. One of these alternate descriptions is remarkably simple: grid entropy is the negative convex conjugate of Gibbs Free Energy; as a direct consequence, our notion of grid entropy is equivalent to that appearing in \cite{rassoul2014quenched} modulo some normalizations. The other, as the entropy of paths with empirical measure Levy-Prokhorov-close to the target, is completely new though not unexpected. 

We summarize these characterizations in the following theorem and link them to current literature in subsequent remarks.

\begin{manualtheorem}{A}
\begin{enumerate}[label=(\roman*)]
\item[]
\item Let $q \in \bR^D_{\geq 0}$ and let $\nu$ be a finite non-negative Borel measure on [0,1]. Then the direction$-q$ grid entropy $||(q,\nu)||$ as defined in \eqref{firstDef} is also given by
\begin{equation} \label{newLabel}
    ||(q, \nu)|| = \inf_{\epsilon > 0} \lim_{n \rightarrow \infty} \frac1n \log \sum_{\pi: \vec{0} \rightarrow \lfloor nq \rfloor} e^{-\frac{n}{\epsilon} \rho(\frac1n \mu_{\pi}, \nu)} \ \mbox{a.s.}
    \end{equation}
The expressions we take an infimum of are each directed metrics with negative sign on $\bR^D \times \mathcal{M}_+$. That is, they take value in $[-\infty, \infty)$, evaluate to 0 when $(q,\nu) = (\vec{0},0)$, and satisfy the triangle inequality with the sign reversed.

Moreover, direction-fixed grid entropy is the negative convex conjugate of the point-to-point $\beta$-Gibbs Free Energy in direction $q$ (as a function of the environment-coupling function $\tau$): For $\beta > 0$,
\[ ||(q, \nu)|| = -(G_q^{\beta})^{*}(\nu) = -\sup_{\tau} [ \beta \langle \tau, \nu \rangle - G_q^{\beta}(\tau)] \]
where the supremum is over bounded measurable $\tau: [0,1] \rightarrow \bR$, where $\langle \cdot, \cdot \rangle$ is the integration linear functional $\langle \tau, \nu \rangle = \int_0^1 \tau(u) d\nu$ and where the point-to-point $\beta$-Gibbs Free Energy is given by
\[ G_{q}^{\beta}(\tau)  = \lim_{n \rightarrow \infty} \frac1n \log \sum_{\pi: \vec{0} \rightarrow \lfloor nq \rfloor} e^{\beta T(\pi)} \]
\item Analogous results hold in the direction-free case. Let $\nu$ be a finite non-negative Borel measure on [0,1] and let $t := ||\nu||_{TV}$. Then the direction-free grid entropy $||\nu||$ as defined in \eqref{firstDef} is also given by
\begin{align*}
      ||\nu|| &= \inf_{\epsilon > 0} \lim_{n \rightarrow \infty} \frac1n \log \sum_{\pi \in \mathcal{P}_{\lfloor nt \rfloor}( \vec{0})} e^{-\frac{n}{\epsilon} \rho(\frac1n \mu_{\pi}, \nu)} \ \mbox{a.s.} \\
      &= \sup_{q \in \bR^D_{\geq 0}, ||q||_1 = t} ||(q,\nu)|| \\
      &= ||(t\ell, \nu)|| \ \mbox{where} \ \ell = \bigg(\frac1D, \ldots, \frac1D \bigg)
\end{align*}  
The expressions we take an infimum of are each directed metrics with negative sign on $\mathcal{M}_t$, the set of finite non-negative Borel measures with total mass $t$. Moreover, direction-free grid entropy is the negative convex conjugate of the point-to-level $\beta$-Gibbs Free Energy:
\[ || \nu|| = -(G^{\beta})^{*}(\nu) = -\sup_{\tau} [\beta \langle \tau, \nu \rangle - G^{\beta}(\tau)] \ \mbox{a.s.}\]
where the supremum is over bounded measurable $\tau: [0,1] \rightarrow \bR$ and where the point-to-level $\beta$-Gibbs Free Energy is given by
\[ G^{\beta}(\tau)  = \lim_{n \rightarrow \infty} \frac1n \log \sum_{\pi \ \mbox{s.t.} \ |\pi| = n} e^{\beta T(\pi)} \]
\end{enumerate}
\end{manualtheorem}

\begin{remark}
We can extend these variational formulas for point-to-point/point-to-level Gibbs Free Energy to general, possibly unbounded, measurable $\tau$ by truncating $\tau$ in $\beta \langle \tau, \nu \rangle$ at some constant $C > 0$ we take to $\infty$. For example,
\[ ||(q, \nu)|| =  - \sup_{C > 0} \sup_{\tau} [ \beta \langle \tau \land C, \nu \rangle - G_q^{\beta}(\tau)] \]
\end{remark}

\begin{remark}
This theorem is partially proved in \cite{carmona2010directed} [Corollary 2] for the case when the edge labels follow a Bernoulli($p$) distribution. In this setting, Carmona shows that the negative convex conjugate of Gibbs Free Energy of measures of the form $\nu_s := s \delta_1 + (1-s) \delta_0$ is given by
\begin{equation*} 
 -(G^{1})^{\ast} (\nu_s) =  \begin{cases} 
\lim\limits_{n \rightarrow \infty} \frac1n \log \#(\mbox{length $n$ paths from $\vec{0}$ with $\geq ns$ 1-labels}), & s \geq p \\ \lim \limits_{n \rightarrow \infty} \frac1n \log \#(\mbox{length $n$ paths from $\vec{0}$ with $\leq ns$ 1-labels}), & s < p \\ \end{cases} 
\end{equation*}
The inequalities $\geq, \leq$ can be replaced with equality since the number of paths with $ns$ 1-labels is exponentially decaying in $s$. Using the definition of the Levy-Prokhorov metric we can conclude that this formulation is equivalent to \eqref{firstDef}.
\end{remark}
\begin{remark}
As mentioned, grid entropy first appears in literature as the rate function for the Large Deviation Principle of the empirical measures in \cite{rassoul2014quenched}. The fact that we derive formula \eqref{newLabel}  using the Subadditive Ergodic Theorem should therefore not be surprising. Our definitions \eqref{firstDef} for direction-fixed/direction-free grid entropy align with those of Rassoul-Agha and Sepp{\"a}l{\"a}inen as follows. Consider the product space $\Omega \times \mathcal{G}$ of the environment \\ $\Omega := [0,1]^{\bZ^D \times  \mathcal{G}}$ consisting of the i.i.d. edge labels and of the $D$ unit NE steps \\ $\mathcal{G} = \{e_1, \ldots, e_D\} $, and let  $\phi: \Omega \times \mathcal{G} \rightarrow [0,1]$ map an (environment, unit direction) pair to the edge label of the corresponding unit direction anchored at the origin. Then 
\begin{equation} \label{translationEqn}
||(q,\nu)|| = \log D - \inf_{\mu: \phi_{\ast}(\mu) = \nu, E^{\mu}[Z_1] = q} H_1(\mu), \ ||\nu|| = \log D - \inf_{\mu:\phi_{\ast}(\mu) = \nu}  H_1(\mu) 
\end{equation}
where $H_1$ is the relative entropy defined in (5.2) of \cite{rassoul2014quenched} which can be traced back to Varadhan's paper \cite{varadhan2003large}, and where the infimums are over the measures $\mu$ on $\Omega \times \mathcal{G}$ whose $\phi$-pushforward is $\nu$ and, in the direction-fixed case, for which in addition the $\mu$-mean of the step coordinate $Z_1$ is $q$. See Section \ref{translationSection} for a more detailed derivation and discussion of \eqref{translationEqn}.
\end{remark}

\bigskip

Now grid entropy satisfies some rather interesting properties. The following theorem captures the highlights for direction$-q$ grid entropy; the direction-free analogues hold as well. All but properties (iii) and (v) follow easily from the framework presented previously in \cite{rassoul2014quenched}, yet we will showcase the power of our new approach to grid entropy by proving \emph{all} of these properties directly.

\begin{manualtheorem}{B}\label{part1_B}
Let $q \in \bR^D_{\geq 0}$. Then:
\begin{enumerate}[label=(\roman*)]
\item Grid entropy $||(q,\nu)||$ is a  directed norm with negative sign; it scales with positive-factors and it satisfies a reverse triangle inequality:
\[ ||(p, \xi)|| + ||(q, \nu)|| \leq ||(p+q, \xi+\nu)|| \]
\item Grid entropy $||(q,\nu)||$ is upper semicontinuous.

Let 
\[\mathcal{R}^q := \{\mbox{accumulation points of empirical measures along paths in direction $q$}\}\]
\item $\mathcal{R}^q$ is weakly closed, convex and deterministic and coincides almost surely with 
\[\{\nu \in \mathcal{M}_+: ||(q, \nu)|| > -\infty\}\]
\item $\mathcal{R}^q$ consists only of measures $\nu$ with total variation $||\nu||_{TV} = ||q||_1$ that are absolutely continuous with respect to the Lebesgue measure $\Lambda$ on [0,1].
\item Any $\nu \in \mathcal{R}^q$ satisfies the following upper bound on the sum of the grid entropy and the relative entropy with respect to Lebesgue measure on [0,1]:
$$D_{KL} (\nu||\Lambda) + ||(q, \nu)|| \leq  \sum_{i =1}^D -q_i \log \frac{q_i}{||q||_1} := H(q)$$
 where $D_{KL}$ denotes relative entropy (or Kullback-Leibler divergence), and where, \\ again, this upper bound is simply the (Shannon) entropy of the total number of paths in direction $q$. 
 \end{enumerate}
\end{manualtheorem}

Why do we care? The  deterministic set $\mathcal{R}^q$ is nothing more than the LPP analogue of the set Bates takes a infimum over in \cite{bates} in his variational formula for the FPP time constant. In our LPP setting, his formula becomes
\begin{align*} 
\mbox{LPP time constant} &:= \lim_{n \rightarrow \infty} \frac{\mbox{last passage time for paths $\vec{0} \rightarrow \lfloor nq \rfloor$}}n \\
& = \sup_{\nu \in \mathcal{R}^q} \langle \tau, \nu \rangle \ \mbox{a.s.}
\end{align*}
We thus link \cite{bates} and \cite{rassoul2014quenched} by providing a new, more enlightening description of these sets $\mathcal{R}^q$ in terms of these grid entropies. Furthermore, the amazing bound in (iii) is an improvement of a  version  without the grid entropy term proved by Bates and which also follows from Rassoul-Agha and Sepp{\"a}l{\"a}inen's work.

As in \cite{rassoul2014quenched}, the main application of grid entropy we present is a variational formula for the point-to-point/point-to-level Gibbs Free Energies in direction $q$/to the "level" $x_1 +\ldots +x_D = nt$ in the directed polymer model. Our variational formula is simply the zero temperature LPP analogue of Bates' variational formula for the FPP limit shape, and furthermore can be molded into Rassoul-Agha and Sepp{\"a}l{\"a}inen's variational formula via some normalizations, thus proving that what we call grid entropy really is the same object first developed in \cite{rassoul2014quenched}.

 \begin{manualtheorem}{C}\label{part1_C}
Fix a direction $q \in \bR^D_{\geq 0}$ and an inverse temperature $\beta > 0$. For bounded measurable $\tau: [0,1] \rightarrow \bR$, the point-to-point/point-to-level Gibbs Free Energies are given by
%s.t. the environment-coupling function $\tau$ satisfies
%\[E[e^{\beta \tau(U)}] < \infty  \ \mbox{for} \ U \sim Unif[0,1]\]
\[ G_{q}^{\beta}(\tau) = \sup_{\nu \in \mathcal{R}^q} [ \beta \langle \tau, \nu \rangle + ||(q, \nu)||],\]
\[G^{\beta}(\tau) = \sup_{\nu \in \mathcal{R}^1} [ \beta \langle \tau, \nu \rangle + ||\nu||] = G^{\beta}_{\ell} = \sup_{q \in \bR^D_{\geq 0}, ||q||_1=1} G^{\beta}_q \]
where $\ell = (\frac1D, \ldots, \frac1D)$ is the maximizing direction, and 
\[\mathcal{R}^1 = \bigcup \limits_{q \in \bR^D_{\geq 0}, ||q||_1=1} \mathcal{R}^q\]
 is the set of Borel probability measures $\nu$ that have finite direction-free grid entropy. 
Moreover, these supremums are achieved by some $\nu$ in $\mathcal{R}^q, \mathcal{R}^{\ell}$ respectively.
\end{manualtheorem}
\begin{remark}
We may replace the supremums in these formulas to be over all measurable $\tau$ by truncating $\tau$ at some $C>0$ and taking a supremum over $C$. 
\end{remark}
 
 As in \cite{bates}, it follows that the directed polymer analogue of Hoffman's question is answered in the affirmative when  our variational formula has a unique maximizer, which happens for a dense family of measurable functions $\tau$.
 
 \begin{manualtheorem}{D}\label{part1_D}
Fix an inverse temperature $\beta > 0$ and a bounded measurable $\tau: [0,1] \rightarrow \bR$. 
\begin{enumerate}[label=(\roman*)]
\item Fix $q \in \bR^D_{\geq 0}$ and suppose $\beta \langle \tau, \nu \rangle + ||(q,\nu)||$ has a unique maximizer $\nu \in \mathcal{R}^q$. For every $n$ pick a path $\pi_n: \vec{0} \rightarrow \lfloor nq \rfloor$ independently and at random according to the probabilities prescribed the corresponding point-to-point $\beta$-polymer measure 
\[ \rho_{n,q}^{\beta}(d\pi) = \frac{e^{\beta T(\pi)}}{ \sum \limits_{\pi: \vec{0} \rightarrow \lfloor nq \rfloor} e^{\beta T(\pi)} } \ \mbox{for paths} \ \pi:\vec{0} \rightarrow \lfloor nq \rfloor \]
Then the empirical measures $\frac1n \mu_{\pi_n}$ converge weakly to $\nu$ a.s.
\item Suppose $\beta \langle \tau, \nu \rangle + ||\nu||$ has a unique maximizer $\nu \in \mathcal{M}_1$. For every $n$ pick a length $n$ path $\pi_n$ from $\vec{0}$ independently and at random according to the probabilities prescribed the corresponding point-to-level $\beta$-polymer measure 
\[ \rho_{n}^{\beta}(d\pi) = \frac{e^{\beta T(\pi)}}{ \sum \limits_{\pi \ \mbox{s.t.} \ |\pi| =  n} e^{\beta T(\pi)} } \ \mbox{for length $n$ paths $\pi$ from} \ \vec{0}  \]
Then the empirical measures $\frac1n \mu_{\pi_n}$ converge weakly to $\nu$ a.s.
\end{enumerate}
\end{manualtheorem}
\begin{remark}
In fact, even if there is no unique maximizer, we can show that all accumulation points of random paths chosen according to the direction-fixed/direction-free $\beta$-polymer measure are among the maximizers of the corresponding variational formula. This partial answer to Hoffman's question in the positive temperature case is immediate from the development of grid entropy by Rassoul-Agha and Sepp{\"a}l{\"a}inen \cite{rassoul2014quenched} as the rate function of the Large Deviation Principle of empirical measures. However, to the best of our knowledge this is the first time it has been formulated explicitly.
\end{remark}
 
The plan is as follows. In Section \ref{section2}, we describe the model and setup, and outline various facts and notions we will need over the course of this paper. Section \ref{section3} focuses on developing the second definition of grid entropy \eqref{newLabel}  as a directed norm with negative sign and showing that it is equivalent to the original definition with the $\min \limits_{\pi}^j$.  Then, in Section \ref{section4}, we investigate what information we can extract from grid entropy and what properties it satisfies (Theorem \ref{part1_B}). We devote Section \ref{section5} to applying our results to establish our variational formula for point-to-point/point-to-level Gibbs Free Energies and study the consequences of this (namely Theorems \ref{part1_C},\ref{part1_D}) as well as the correspondence to the work in \cite{rassoul2014quenched}. Last but not least, we make some closing remarks about adapting our results to other models.

\newpage

\section{Preliminaries}\label{section2}

\subsection{Empirical Measures on the Lattice}
We begin by briefly describing the setup we use in this paper. 

We restrict ourselves to a directed Last Passage Percolation model. 
Consider north-east nearest-neighbour paths on the lattice $\bZ^D, D \geq 1$ with i.i.d. edge weights $\tau_e \sim \theta$ for some probability distribution  $\theta$ on $\bR$. By north-east, we of course mean that the coordinates of points on the path are nondecreasing.

For $p, q \in \bZ^D$ we denote by $\mathcal{P}(p,q)$ the set of all NE paths $\pi: p \rightarrow q$. Similarly, for $q \in \bZ^D$ and $t \in \bZ_{\geq 0}$ we denote by $\mathcal{P}_t(q)$ the set of all NE paths from $q$ of length $t$ (no restriction on the endpoint).

Observe that  either $q-p \notin \bZ_{\geq 0}^D$ and $\mathcal{P}(p,q) = \emptyset$ , or $q-p \in \bZ_{\geq 0}$ and  
\[|\mathcal{P}(p,q) | = \binom{||q-p||_1}{q_1-p_1,q_2-p_2,\cdots, q_D-p_D}\]
Here $||\cdot||_1$ is the 1-norm on $\bR^D$ defined by $||p||_1 = \sum \limits_{i=1}^D |p_i|$.

On the other hand, $|\mathcal{P}_t(q)| = D^t$ trivially for any $q \in \bZ^D$.

Note that unlike other recent work (such as \cite{martinAllan}) we do not restrict ourselves to a known solvable model. We do not impose any restrictions on the edge weight distribution $\theta$, so that our results hold with the greatest generality possible.

We scale the grid by $n$ and look at the behaviour in the limit. As is standard, we extend our initial inputs $p, q$ to lie in $\bR^D$ by taking coordinatewise floors of the scaled coordinates. That is, we consider  paths $\pi \in \mathcal{P}(\lfloor np \rfloor, \lfloor nq \rfloor)$ where
 \[\lfloor (x_1,\ldots, x_D) \rfloor := (\lfloor x_1 \rfloor, \ldots, \lfloor x_D \rfloor)\]
  Our normalized directed metrics will converge almost surely to a translation-invariant limit so it will suffice to consider the case when $p = \vec{0}$. But for now we let $\vec{p}$ be arbitrary.

Various inequalities we derive involve the asymptotics of the number of length $n$ NE paths from the origin in a fixed or unfixed direction. The following lemma, which is easily proved using Sterling's approximation, gives us what we want.

\begin{lemma}\label{lemma1}
Let $k, n \in \bN, a_i \in \bR_{\geq 0}, \sum a_i = a$. Then
\[\binom{\lfloor na \rfloor}{\lfloor na_1 \rfloor, \lfloor na_2 \rfloor, \ldots, \lfloor na_k \rfloor}  = \bigg(\frac{a^{a}}{ \prod \limits_{1 \leq i \leq k} a_i^{a_i}} + o(1) \bigg)^n  \]
where we use the convention $0^0 = 1$.
\end{lemma}

\begin{remark}
Recall that for $q \in \bR^D_{\geq 0}$ we denote by $\mathcal{P}(\vec{0}, \lfloor nq \rfloor)$ the set of paths $\pi: \vec{0} \rightarrow \lfloor nq \rfloor$. Thus
\[  \lim_{n \rightarrow \infty} \frac1n \log |\mathcal{P}(\vec{0}, \lfloor nq \rfloor)|   = \sum_{i=1}^D -q_i \log\frac{q_i}{||q||_1} := H(q) \]
$H(q)$ is the (Shannon) entropy of the number of paths in direction $q$. Note that $H(q)$ scales with positive scalars and  $H(q)$ is maximized among $q \in \bR^D_{\geq 0}$ with the same 1-norm by
\[q = ||q||_1 \bigg(\frac1D, \ldots, \frac1D \bigg) := ||q||_1 \ell\]
in which case $H(q) = ||q||_1 \log D$.
\end{remark}

On the other hand, for $t \geq 0$ recall that we denote by $\mathcal{P}_{\lfloor nt \rfloor}(\vec{0})$ the set of paths from $\vec{0}$ of length $\lfloor nt \rfloor$. Thus
\[ \lim_{n \rightarrow \infty} \frac1n \log |\mathcal{P}_{\lfloor nt \rfloor}(\vec{0})| = t \log D = H(t\ell) \]

 We study the distribution of weights that we observe along paths $\# \pi: \lfloor np \rfloor \rightarrow \lfloor nq \rfloor$. For a path $\pi$, let the unnormalized empirical measure along $\pi$ be
\[\sigma_{\pi} =  \sum_{e\in \pi} \delta_{\tau_e}\]
Note that we normalize by $\frac1n$ rather than $\frac1{|\pi|}$. This is simply for convenience in our proofs, as it gives us a certain superadditivity we do not get when we normalize by $\frac1{|\pi|}$.

 The Glivenko-Cantelli Theorem \cite[Thm.~2.4.7]{durrett} tells us that for any fixed infinite NE path in the grid, empirical measures along the path converge weakly to $\theta$.

\begin{theorem}[Glivenko-Cantelli Theorem]\label{thm1}
Let $F_{\theta}$ be the cumulative distribution function of $\theta$, let $X_i \sim \theta$ be i.i.d.  random variables and let 
\[F_{n}(x) = \frac1{n} \sum_{i=1}^n \mathbf{1}_{\{X_i \leq x\}} \]
be the cumulative distribution functions of the empirical measures. Then 
\[\sup_x |F_n(x) - F_{\theta}(x)| \rightarrow 0 \ \mbox{a.s. as} \ n \rightarrow \infty \]

\end{theorem}

However, we are interested in the limiting behavior of the empirical measure of not one path, but of all paths from $\vec{0}$ or all paths with a given direction, as we scale the length of the paths. This allows us to observe more than just the original measure $\theta$.

\subsection{Metrics on Measures}
To gauge the distance between two measures we use the Levy-Prokhorov metric. We briefly introduce this metric as well as the total variation metric and we  outline the relevant properties.

Consider a metric space $(X,d)$ with Borel $\sigma$-algebra $\mathcal{B}$. We denote by $\mathcal{M}$ the set of finite  Borel measures on $(X, \mathcal{B})$, by $\mathcal{M}_+$ the set of non-negative finite Borel measures, and by $\mathcal{M}_t$ the set of Borel non-negative finite  measures with total mass $t$ for any $t \geq 0$. In this notation, $\mathcal{M}_1$ is the set of Borel probability measures.

\begin{definition}
The total variation norm on $\mathcal{M}$ is defined by 
\[||\mu||_{TV} = \sup_{A \in \mathcal{B}} |\mu(A)|\]
This of course gives rise to a total variation metric, given by \[d_{TV}(\mu, \nu) = ||\mu-\nu||_{TV}\]
\end{definition}

For example, the total variation of any measure $\mu \in \mathcal{M}_+$ is its total mass $\mu(X)$.

\begin{definition}
For any $A \in \mathcal{B}$ and $\epsilon > 0$, the $\epsilon$-neighborhood of $A$ is defined to be
$$A^{\epsilon} := \{x \in X: d(x, a) < \epsilon \ \mbox{for some} \ a \in A \}$$
\end{definition}

\begin{definition}The Levy-Prokhorov metric on $\mathcal{M}_+$ is defined by 
\[\rho(\mu, \nu) = \inf \{\epsilon > 0: \mu(A) \leq \nu(A^{\epsilon}) + \epsilon \ \mbox{and} \ \nu(A) \leq \mu(A^{\epsilon}) + \epsilon \ \forall A \in \mathcal{B}\}\]
\end{definition}

 It is a standard result that $\rho$ metrizes the weak convergence of measures and total variation metrizes the strong convergence of measures in $\mathcal{M}_+$. For details, see \cite[Sect.~2.3]{huber}.

 We now derive two useful inequalities involving the Levy-Prokhorov metric.

\begin{lemma}\label{lemma3}
For $\mu, \nu \in \mathcal{M}_+$,
\[\rho(\mu, \nu) \leq ||\mu - \nu||_{TV}\]
\end{lemma}

\begin{remark}
This lemma establishes that the Levy-Prokhorov metric is weaker than the total variation metric.
\end{remark}

\begin{remark}
It is trivial to see that $\rho(\mu, 0) = ||\mu||_{TV}$.
\end{remark}

\begin{proof}
Let $\epsilon := ||\mu-\nu||_{TV}$. If $\epsilon = 0$ then $\mu = \nu$ so $\rho(\mu,\nu) = 0$. If $\epsilon > 0$, for any $A \in \mathcal{B}$, 
$$A \subseteq A^{\epsilon} \Rightarrow \mu(A) - \nu(A^{\epsilon}) \leq \mu(A) -  \nu(A)  \leq \sup_{A' \in \mathcal{B}} |\mu(A') - \nu(A')| = \epsilon$$
and similarly $\nu(A) - \mu(A^{\epsilon}) \leq \epsilon$ so $\rho(\mu, \nu) \leq \epsilon$.
\end{proof}

We next show that $\rho$ satisfies a kind of subadditivity.

\begin{lemma}\label{lemma4}
Let $\mu_1, \mu_2, \nu_1, \nu_2 \in \mathcal{M}_+$. Then
\[\rho(\mu_1 + \mu_2,  \nu_1 + \nu_2) \leq \rho(\mu_1, \nu_1) + \rho(\mu_2, \nu_2)\]
\end{lemma}

\begin{remark}
Note that in the case $\mu_2 = \nu_2$  the inequality becomes
\[\rho(\mu_1 + \mu_2,  \nu_1 + \mu_2) \leq \rho(\mu_1, \nu_1) \]
\end{remark}

\begin{proof}
For any $\epsilon_1 > \rho(\mu_1, \nu_1), \epsilon_2 > \rho(\mu_2, \nu_2)$ we have for any $A \in \mathcal{B}$,
\[\mu_1(A) \leq \nu_1(A^{\epsilon_1}) + \epsilon_1 \leq \nu_1(A^{\epsilon_1 +\epsilon_2}) + \epsilon_1 \ \mbox{and} \ \mu_2(A) \leq \nu_2(A^{\epsilon_2}) + \epsilon_2 \leq \nu_2(A^{\epsilon_1 +\epsilon_2}) + \epsilon_2\]
\[\Rightarrow (\mu_1 + \mu_2)(A) \leq (\nu_1 + \nu_2) (A^{\epsilon_1 +\epsilon_2}) + \epsilon_1 + \epsilon_2\]
By symmetry, the same inequality holds with $\mu_i, \nu_i$ swapped. Thus 
\[\rho(\mu_1 + \mu_2,  \nu_1 + \nu_2) \leq \rho(\mu_1, \nu_1) + \rho(\mu_2, \nu_2)\]
\end{proof}

\subsection{A Convenient Coupling of the Edge Weights}
We follow \cite[Sect.~2.1]{bates} in coupling the environment to uniform random variables in order to work in a compact space of measures and to connect our results with his.

The idea is to write our i.i.d. edge weights $\tau_e \sim \theta$ as
\[ \tau_e = \tau(U_e)\]
for some measurable function $\tau: [0,1] \rightarrow \bR$ and i.i.d. Unif[0,1]-valued random variables $(U_e)_{e \in E(\bZ^D)}$ on the same probability space as $(\tau_e)_{e \in E(\bZ^D)}$. For instance, we could take the quantile function
\[ \tau(x) = F_{\theta}^-(x) := \inf \{t \in \bR: F_{\theta}(t) \geq x\}\]
But our results (in particular, our definitions of grid entropy) are independent of the $\tau$ chosen so we allow $\tau$ to be arbitrary (with the conditions stated above). This comes into play later in Section \ref{section5}, when we study the Gibbs Free Energy as a function of  $\tau$.

To distinguish between the $\tau_e$ and the $U_e$, we call the former edge \emph{weights} and the latter edge \emph{labels}.

Let $\Lambda$ denote Lebesgue measure on $[0,1]$. We tweak the definition of empirical measures in this new setup: for any NE path $\pi: \lfloor np \rfloor \rightarrow \lfloor nq \rfloor$ in $\bZ^D$, define
\[ \mu_{\pi} := \sum_{e \in \pi} \delta_{U_e} \]
Then we can relate $\Lambda$ and the $\mu_{\pi}$ to $\theta$ and the $\sigma_{\pi}$ respectively via the pushforward:
\[ \theta = \tau_{*}(\Lambda), \sigma_{\pi} = \tau_{*}(\mu_{\pi}) \ \mbox{where} \ \tau_{*}(\xi)(B) = \xi(\tau^{-1} (B)) \ \forall B \in \mathcal{B}(\bR^D) \ \mbox{and $\forall$ measures $\xi$ on $[0,1]$}\]
One advantage is of course that the set of probability measures on $[0,1]$ is weakly compact, so for any sequence of paths $\pi_n: \lfloor np \rfloor \rightarrow \lfloor nq \rfloor$ in the grid we get a subsequence for which $\frac1{n_k} \mu_{\pi_{n_k}}$ converges weakly to some measure. 

In the case of a continuous cumulative distribution function $F_{\theta}$, we can use  a lemma proved in \cite{bates} to get a nice duality.

\begin{lemma}\label{lemma5Bates}
Given a measure $\theta$ on $\bR$ with continuous cdf $F_{\theta}$, if we let \\ $\tau = F_{\theta}^{-}: [0,1] \rightarrow \bR$ be its quantile function then there is a probability 1 event on which $\frac1{n_k} \mu_{\pi_{n_k}} \Rightarrow \nu$ for some subsequence $n_k$ and paths $\pi_{n_k}: \vec{0} \rightarrow \lfloor n_k q \rfloor$ if and only if $ \tau_{*}(\frac1{n_k}\mu_{\pi_{n_k}}) \Rightarrow \tau_{*}(\nu)$.
\end{lemma}

\begin{proof}
 \cite[Lemma 6.15]{bates}  establishes that there is a probability 1 event on which \\ $\frac1{n_k}\mu_{\pi_{n_k}} \Rightarrow \nu$ implies $ \tau_{*}(\frac1{n_k}\mu_{\pi_{n_k}}) \Rightarrow \tau_{*}(\nu)$. But then on the same event, given a subsequence for which $ \tau_{*}(\frac1{n_k}\mu_{\pi_{n_k}}) \Rightarrow \tau_{*}(\nu)$, compactness gives us a convergent subsubsequence $\frac1{n_{k_j}}\mu_{\pi_{n_{k_j}}} \Rightarrow \xi$ hence 
 \[ \tau_{*}\bigg(\frac1{n_{k_j}}\mu_{\pi_{n_{k_j}}} \bigg) \Rightarrow \tau_{*}(\xi)\]
 so $\tau_{*}(\xi) = \tau_{*}(\nu)$. The fact that $F_{\theta}$ is continuous and the quantile function $\tau$ satisfies
 \[ \tau^{-1}((-\infty, x]) = [0,F_{\theta}(x)] \ \forall x \in \bR\]
 implies $\xi$ and $\nu$ agree on all sets $[0,F_{\theta}]$ hence $\xi = \nu$.
\end{proof}

Thus, in the case when $\theta$ has continuous cdf, we lose no generality by doing this coupling and working with measures on [0,1].

However, even in the most general case where $\theta$ may not have continuous cdf or   bounded support, our work in developing grid entropy still holds because we only use the compactness of the space of measures in later sections devoted to our variational formula for the Gibbs Free Energy.

In short, we lose nothing by restricting ourselves to the compact space of measures on $[0,1]$. 

A benefit of this coupling is the following amazing result of Bates \cite[Lemma 6.3 and  Thm 6.4]{bates}. Here we denote by $\mathcal{M}_+, \mathcal{M}_t$ the sets of finite non-negative Borel measures on [0,1] and finite non-negative Borel measures on [0,1] with total mass $t\geq 0$.

\begin{theorem}\label{thm5Bates}
\begin{enumerate}[label=(\roman*)]
\item[]
\item Fix $q \in \bR^D_{\geq 0}$. Define $\mathcal{R}_{\infty}^{q}$ to be the (event-dependent) set of measures $\nu \in \mathcal{M}_+$ for which there is a subsequence $\pi_{n_k}$ of paths $\vec{0} \rightarrow \lfloor n_k q \rfloor$ with $\mu_{\pi_{n_k}} \Rightarrow \nu$. Then there exists a deterministic, weakly closed set $\mathcal{R}^{q} \subseteq \mathcal{M}_{||q||_1} $ independent of $\tau$ s.t.
\[ P(\mathcal{R}_{\infty}^{q} = \mathcal{R}^{q}) = 1\]
\item Fix $t \geq 0$. Define $\mathcal{R}^t_{\infty}$ to be the (event-dependent) set of measures $\nu \in \mathcal{M}_+$ for which there is a subsequence $\pi_{n_k}$ of length $\lfloor nt \rfloor$ paths from $\vec{0}$ with $\mu_{\pi_{n_k}} \Rightarrow \nu$. Then there exists a deterministic, weakly closed set $\mathcal{R}^t \subseteq \mathcal{M}_t$ independent of $\tau$ s.t.
\[ P(\mathcal{R}^t_{\infty} = \mathcal{R}^t) = 1\]
Moreover, $\mathcal{R}^t = \bigcup \limits_{q \in \bR^D_{\geq 0}, ||q||_1 = t} \mathcal{R}^q$.
\end{enumerate}
\end{theorem}

\begin{remark}
Bates proves this theorem in the setup of First Passage Percolation, but notes that it holds analogously in the Last Passage model.
\end{remark}

Instead of looking at \textit{all} empirical measures that have a weakly convergent subsequence we may look at only certain empirical measures with this property and the same result will hold. The proof is almost identical to Bates's original proof except for this change, so we omit it.

\begin{corollary}\label{corollary6Bates}
\begin{enumerate}[label=(\roman*)]
\item[]
\item Fix $q \in \bR^D$ and $0 \leq \alpha \leq R(q)$. Define $\mathcal{R}_{\infty}^{q, \alpha}$ to be the (event-dependent) set of measures $\nu \in \mathcal{M}_+$ for which there is a subsequence $\pi_{n_k}$ of the paths $\pi_n: \vec{0} \rightarrow \lfloor nq \rfloor$ with the $\lfloor e^{n\alpha}\rfloor$th smallest values of $\rho(\frac1n \mu_{\pi_n}, \nu)$ satisfying 
\[\frac1n \mu_{\pi_{n_k}} \Rightarrow \nu \ \mbox{i.e.} \ \liminf_{n \rightarrow \infty} \min_{\pi: \vec{0} \rightarrow \lfloor nq \rfloor}^{\lfloor e^{n\alpha} \rfloor} \rho\bigg(\frac1n \mu_{\pi}, \nu\bigg) = 0 \]
Then there exists a deterministic weakly closed set $\mathcal{R}^{q,\alpha} \subseteq \mathcal{M}_{||q||_1}$ s.t.
\[ P(\mathcal{R}_{\infty}^{q,\alpha} = \mathcal{R}^{q, \alpha}) = 1\]
\item Fix $t \geq 0$ and $0 \leq \alpha \leq t \log D$. Define $\mathcal{R}_{\infty}^{t, \alpha}$ to be the (event-dependent) set of  measures $\nu \in \mathcal{M}_+$ for which there is a subsequence $\pi_{n_k}$ of the length $\lfloor tn \rfloor$ paths $\pi_n$ from $\vec{0}$  with the $\lfloor e^{n\alpha}\rfloor$th smallest values of $\rho(\frac1n \mu_{\pi_n}, \nu)$ satisfying 
\[\frac1n \mu_{\pi_{n_k}} \Rightarrow \nu \ \mbox{i.e.} \ \liminf_{n \rightarrow \infty} \min_{\pi: |\pi| = \lfloor nt \rfloor}^{\lfloor e^{n\alpha} \rfloor} \rho\bigg(\frac1n \mu_{\pi}, \nu\bigg) = 0 \]
Then there exists a deterministic weakly closed set $\mathcal{R}^{t,\alpha} \subseteq \mathcal{M}_t$ s.t.
\[ P(\mathcal{R}_{\infty}^{t,\alpha} = \mathcal{R}^{t,\alpha}) = 1\]
\end{enumerate}
\end{corollary}

\begin{remark}
Since the $\min \limits_{\pi: \vec{0} \rightarrow \lfloor nq \rfloor}^{\lfloor e^{n\alpha} \rfloor}$ are increasing in $\alpha$ then the sets $\mathcal{R}^{q, \alpha}$ are decreasing in $\alpha$. The same holds for $\mathcal{R}^{t,\alpha}$.
\end{remark}

Once we develop the concept of grid entropy, we will easily relate these sets $\mathcal{R}^q, \mathcal{R}^{q,\alpha}, \mathcal{R}^t, \mathcal{R}^{t,\alpha}$ to the sets of measures with finite grid entropy in direction $q$, grid entropy at least $\alpha$ in direction $q$, finite direction-free length $t$ grid entropy, direction-free length $t$ grid entropy at least $\alpha$ respectively.

\subsection{Directed Metric Spaces}
Grid entropy will turn out to be a directed metric with negative sign. We recall what that entails.

\begin{definition}
A directed metric space with positive sign is a triple $(M, d, +)$ where $M$ is a vector space, $d: M^2 \rightarrow (-\infty, +\infty]$ is a distance function satisfying $d(x,x) = 0$ and the usual triangle inequality $d(x,y) + d(y,z) \geq d(x,z)$. A directed metric space with negative sign is a triple $(M, d, -)$ such that $(M, -d, +)$ is a directed metric space with positive sign.
\end{definition}
\begin{remark} Standard metric spaces are clearly examples of directed metric spaces with positive sign. However, directed metric spaces with positive sign might not be metric spaces: the distance $d$ might not be symmetric and might not be positive and finite for non-equal arguments. The "directed" in the name indicates the possibility for asymmetry.
\end{remark}
 
 Certain directed metrics give rise to directed norms in the same way certain metrics give rise to norms.
 \begin{definition}
 If $(M,d, \sigma)$ is a directed metric with positive/negative sign such that it is translation-invariant and homogeneous with respect to positive factors, then it induces a directed norm with positive/negative sign given by
 \[||x|| := d(\vec{0}, x)\]
 \end{definition}

 Of particular interest to us are directed norms defined in terms of the  empirical measures we observe along paths between points. In the fixed direction case, our directed norms will be defined on the space of  tuples consisting of a point in $\bR^D$ (the "direction" we are observing) and a finite Borel measure on $\bR$ (the target measure we want the empirical measures to be near). In the direction-free case, our directed norms will just be defined on the space $\mathcal{M}_+$ of  finite Borel measure on $\bR$.

\subsection{The Subadditive Ergodic Theorem }\label{section2.5}
The key theorem we use to prove the existence of the scaling limit of these directed metrics is Liggett's improved version of Kingman's Subadditive Ergodic Theorem. Before stating this theorem, we recall the definitions of stationary sequences and ergodicity, as presented in  \cite[Sect. 7]{durrett}.

\begin{definition}
A sequence $(Y_n)_{n \geq 1}$ of random variables is called stationary if the joint distributions of the shifted sequences $\{Y_{k+n}: n \geq 1\}$ is not dependent on $k \geq 0$.
\end{definition}

 As it turns out, the  sequence of random variables we are interested in is a sequence of i.i.d. $(Y_n)_{n \geq 1}$, which  clearly is stationary.

\begin{definition}
Let $(\Omega, \mathcal{F}, P)$ be a probability space and let $T: \Omega \rightarrow \Omega$ be a map. $T$ is said to be measure-preserving if $P(T^{-1}(A)) = P(A) \ \forall A \in \mathcal{F}$. $T$ is said to be ergodic if it is measure-preserving and if all $T$-invariant measurable sets are trivial, i.e. $P(A) \in \{0,1\}$ whenever $A \in \mathcal{F}$ and $T^{-1}(A) = A$.
\end{definition}

 In the context of sequences, we look at the space $\Omega = \bR^{\infty}$ of infinite sequences of real numbers with the $\sigma$-algebra $\mathcal{B}_{\infty}$ generated by 
 \[\{\{(y_1, y_2, \ldots) \in \Omega: y_n \in B\}: n \geq 1, B \in \mathcal{B}(\bR)\}\]
 (where $\mathcal{B}(\bR)$ is the Borel $\sigma$-algebra on $\bR$) and with the product probability measure $\mu_{\infty}$ on $\mathcal{B}_{\infty}$ determined by
$$\mu_{\infty}(B_1 \times B_2 \times \ldots \times B_n \times \bR \times \cdots) = \prod_{i=1}^n \mu(B_i) $$
where $B_i \in \mathcal{B}(\bR)$ and $\mu$ is a Borel probability measure on $\bR$. We consider the shift operator $T: \Omega \rightarrow \Omega$ given by $T(y_1, y_2, y_3, \ldots) = T(y_2, y_3, \ldots)$. $T$ is easily seen to be measure-preserving with respect to $\mu_{\infty}$ since $\mu_{\infty}(T^{-1}(A)) = \mu_{\infty}(A)$ for the generating sets $A$ of $\mathcal{B}_{\infty}$. When we refer to the ergodicity of a sequence of random variables, we mean the ergodicity of this shift operator.

 In our case, where we have a sequence of i.i.d. $(Y_n)_{n \geq 1}$, the corresponding shift operator is ergodic. Indeed, if $T^{-1}(A) = A$ then 
\[(Y_1, Y_2, \ldots) \in A \Leftrightarrow T^{n-1}(Y_1, Y_2, \ldots) = (Y_{n}, Y_{n+1}, \ldots) \in A \ \forall n \geq 1\]
so $A$ is in the tail $\sigma$-field $\bigcap \limits_{n \geq 1} \sigma(Y_n, Y_{n+1}, \ldots)$ and thus $\mu_{\infty}(A) \in \{0,1\}$ by Kolmogorov's 0-1 Law.

We are now ready to state the Subadditive Ergodic Theorem in the form we need.

\begin{theorem}[Kingman's Subadditive Ergodic Theorem, \cite{liggett}]  \label{thmKingman}
 Suppose $(Y_{m,n})_{0 \leq m < n}$ are random variables satisfying 
 \begin{enumerate}[label=(\roman*)]
\item $\exists$ constant $C$ s.t. $E|Y_{0,n}| < \infty$ and $EY_{0,n} \geq Cn$ for all $n$
\item $\forall k \geq 1$, $\{Y_{nk, (n+1)k}: n \geq 1\}$ is a stationary process
\item The joint distributions of $\{Y_{m, m+k}: k \geq 1\}$ are not dependent on $m$
\item $Y_{0,m+n} \leq Y_{0,m} + Y_{m, m+n} \ \forall m, n > 0$
\end{enumerate}
Then 
 \begin{enumerate}[label=(\alph*)]
\item $\lim \limits_{n \rightarrow \infty} \frac{EY_{0,n}}n = \inf \limits_{m \geq 0} \frac{EY_{0,m}}m := \gamma$
\item $Y:= \lim \limits_{n \rightarrow \infty} \frac{Y_{0,n}}n$ exists a.s. and in $L^1$, and $EY = \gamma$
\item If the stationary sequences in (ii) are ergodic, then $Y = \gamma$ a.s.
\end{enumerate}
\end{theorem}

\begin{remark}
We may replace $Y_{m,n}$ with $-Y_{m,n}$ in the statement of the theorem to obtain a version for superadditive sequences.
\end{remark}

\noindent This theorem is the basis for the construction of grid entropy in our paper.

\subsection{Relative Entropy and Sanov's Theorem} \label{section2.6}
In the next preliminary section, we recall the basics of the Kullback-Leibler divergence (introduced in \cite{kullback}) and Sanov's Theorem for large deviations. We later use this theorem to establish a relationship between our grid entropy and this notion of relative entropy.

\begin{definition}
Let $P, Q$ be distributions on our inherent metric space $X$.  The Kullback-Leibler divergence or relative entropy of $Q$ from $P$ is defined to be
\[D_{KL} (P || Q) = \begin{cases} \int_X  \log f \ dP = \int_X   f \log f \ dQ, & P \ll Q \\ +\infty, & \mbox{otherwise} \end{cases}\]
where $f := \frac{dP}{dQ}$ is the Radon-Nikodym derivative and $\log$ is the natural logarithm. 
\end{definition}

\begin{remark}
 \cite{kullback} also derive several basic properties such as $D_{KL}$ being a pre-metric. \cite{posner} shows that $D_{KL}$ is lower semicontinuous, in the sense that, given probability distributions $P_n \Rightarrow P$ and $Q_n \Rightarrow Q$, we have
 \[ D_{KL}(P||Q) \leq \liminf_{n \rightarrow \infty} D_{KL}(P_n || Q_n)\]
 
\end{remark}

 Our main interest in relative entropy is that it is the rate function for large deviations of empirical measures. This is captured by Sanov's Theorem. 

\begin{theorem}[Sanov's Theorem, \cite{deuschel}]\label{thmSanov}
Consider a sequence of i.i.d. random variables $X_i \sim \theta$ taking values in a set $X$. Let $\mu_n =  \sum_{i=1}^n \delta_{X_i}$ be their  unnormalized empirical measures. Then for any weakly closed set $F \subset \mathcal{M}_1$ we have
\[\limsup_{n \rightarrow \infty} \frac1n \log P \bigg(\frac1n \mu_n \in F \bigg) \leq -\inf_{\xi \in F} D_{KL}(\xi||\theta)\]
and for any weakly open set $G \subset \mathcal{M}_1$ we have
\[\liminf_{n \rightarrow \infty} \frac1n \log P\bigg(\frac1n \mu_n \in U \bigg) \geq -\inf_{\xi \in G} D_{KL}(\xi||\theta)\]
\end{theorem}
\begin{remark}
Since $F$ is closed then the infimum is achieved by some $\xi \in F$. Furthermore, if $\theta \in F$ then the right-hand side of the inequality is 0, which gives us no information; however, if $\theta \notin F$, then the theorem gives an exponential bound on  large deviations.
\end{remark}

\section{Grid Entropy as a Directed Norm}\label{section3}
\subsection{The Plan for Deriving Direction-Fixed Grid Entropy }
For the purposes of Section \ref{section3}, we temporarily forget our original definition \eqref{firstDef} of grid entropy and rederive it as a limit of scaled directed metrics. 

 To summarize the setting we described in section \ref{section2}, we consider empirical measures $\frac1n \mu_{\pi}$ along NE-paths on the lattice $\bZ^D$, where the edges have weights $\tau_e = \tau(U_e)$ for some measurable $\tau: [0,1] \rightarrow \bR$ and $U_e$ are i.i.d. Unif[0,1] random variables. $\mathcal{M}, \mathcal{M}_+, \mathcal{M}_t$ denote the spaces of finite, finite non-negative, and finite non-negative with total mass $t \geq 0$ Borel measures respectively.

We begin with direction-fixed grid entropy. We wish to count the number of paths with empirical measure very close to the target $\nu$. Let us try to define a distance on $\mathbb{\bR}^D \times \mathcal{M}_+$   by
\[d((p, \xi), (q, \nu)) = \log \# \{
\mbox{paths $\pi: \lfloor p \rfloor \rightarrow \lfloor q \rfloor$ with $ \mu_{\pi}= \nu-\xi$}\}\]
Note that this is $-\infty$ if there are no such paths and it is 0 if $p=q, \xi = \nu$ (since there is exactly one path $\pi: \lfloor p \rfloor \rightarrow \lfloor p \rfloor$, and it has empirical measure 0). 

 One glaring issue is that we are only counting paths that have exactly the target empirical measure.  Since the Lebesgue measure on [0,1] is continuous, then almost surely the $U_e$ have different values, and thus the unnormalized empirical measures uniquely determine the paths $\pi$. It follows that almost surely $d((p, \xi), (q, \nu))$ is always either $-\infty$ or 0. We need to change our definition of $d$ to count paths with an empirical measure "close" to $\nu-\xi$ instead. 

 Another problem is that we wish to apply the Subadditive Ergodic Theorem to learn about the  behavior of
 \[\frac{d((np,n\xi), (nq, n\nu))}n\]
as $n \rightarrow \infty$. Thus we must also change our definition of $d$ so that it is integrable (and in particular finite a.s.) when $q-p \in \bR_{\geq 0}^D$. 

 The trick is to replace the counting of the paths exhibiting  the exact target empirical measures with a "cost function" that attributes an exponential cost to each path based on how far its empirical measure is from the target. We also introduce a parameter $\epsilon$ that, as it decreases to 0, takes the cost function towards the counting of the paths with the target empirical measure we tried initially.
 
\begin{definition}
Fix $\epsilon > 0$. Define a distance on $\mathbb{R}^D \times \mathcal{M}_+$  by
\[d^{\epsilon}((p, \xi), (q, \nu)) = \log \sum_{\pi \in \mathcal{P}( \lfloor p \rfloor, \lfloor q \rfloor)}  e^{-\frac1{\epsilon} \rho(\mu_{\pi}, \nu-\xi)}\]
where  $\rho$ is the Levy-Prokhorov metric and where we sum over all NE paths $\pi: \lfloor p \rfloor \rightarrow \lfloor q \rfloor$.
\end{definition}

\begin{remark}
This distance is 0 if $p=q, \xi = \nu$ and $-\infty $ if and only if $q-p \notin \bR_{\geq 0}^D$. 
\end{remark}
\begin{remark}
For any path $\pi: \lfloor p \rfloor \rightarrow \lfloor q \rfloor$, the corresponding  empirical measure $\mu_{\pi}$ observed must necessarily be of the form $\mu_{\pi} = \sum \limits_{i=1}^{|\pi|} \delta_{a_i}$ where $|\pi|= ||\lfloor q \rfloor - \lfloor p \rfloor||_1$.
\end{remark}
\begin{remark}
As $\epsilon \rightarrow 0$ the costs $e^{-\frac1{\epsilon} \rho(\mu_{\pi}, \nu - \xi) }$ converge to the indicators
\[\mathbf{1}_{\nu-\xi}( \mu_{\pi}) = \mathbf{1}_{\frac{\nu-\xi}{||\lfloor q \rfloor - \lfloor p \rfloor||_1} } \bigg(\frac1{|\pi|} \mu_{\pi} \bigg)\]
Hence the sum of the costs approaches the number of paths with empirical measures $\frac1{|\pi|} \mu_{\pi}$ precisely equal to $\frac{\nu-\xi}{||\lfloor q \rfloor - \lfloor p \rfloor||_1}$.
\end{remark}

With this definition in mind, let us discuss the plan of attack. 

First, we prove the existence of $\lim \limits_{n \rightarrow \infty} \frac{d^{\epsilon}((np,n\xi), (nq, n\nu))}n$ using the Subadditive Ergodic Theorem and some estimates we derive for the error terms when $p, q$ do not have integer coordinates. Then we  take the infimum over $\epsilon > 0$ of these limits and we define the resulting norm to be our grid entropy.

\begin{theorem}\label{thm7}
Fix $\epsilon > 0, \nu, \xi \in\mathcal{M}_+$ and $p, q \in \bR^D$. Then
\[\frac{d^{\epsilon}((np, n\xi), (nq, n\nu))}n\]
converges in probability to a constant. When $p = \vec{0}$, the convergence is pointwise a.s.
\end{theorem}

\begin{remark}
The theorem holds trivially, with the limit being $-\infty$,  if $q-p \notin \bR^D_{\geq 0}$ or if $q = p$ and $\nu \neq \xi$. It also holds trivially, with the limit being 0, if $q=p$ and $\nu = \xi$.
\end{remark}

 The limit given by this theorem is a directed metric with negative sign on $\bR^D \times \mathcal{M}_+$. When we  take an infimum over $\epsilon > 0$ we still get a directed metric with negative sign. 

\begin{theorem}\label{thm8}
For $\epsilon > 0$, $\nu, \xi \in \mathcal{M}_+$ and $p, q \in \bR^D$ define 
\[\widetilde{d}^{\epsilon}((p, \xi), (q, \nu)) := \lim \limits_{n \rightarrow \infty} \frac{d^{\epsilon}((np,n\xi), (nq, n\nu))}n \ \mbox{and} \ \widetilde{d}((p, \xi), (q, \nu)) := \inf_{\epsilon > 0} \widetilde{d}^{\epsilon}((p, \xi), (q, \nu))\]
Then each $\widetilde{d}^{\epsilon}$ as well as $\widetilde{d}$ are directed metrics with negative sign on $\bR^D \times \mathcal{M}_+$.
\end{theorem}

 We show that this metric $\widetilde{d}$  gives rise to a norm on $\bR^D_{\geq 0} \times \mathcal{M}_+$. This will finish our discussion of the direction-fixed grid entropy and we will move on to the direction-free case.

\subsection{The Limit Shape of \texorpdfstring{$d^{\epsilon}$}{de} Starting at \texorpdfstring{$(\vec{0}, 0)$}{Lg}}
In this and the following subsection, we focus on the direction-fixed grid entropy.

To prove Theorem \ref{thm7}, we first prove a simplified version, which we later generalize easily.

\begin{theorem}\label{thm9}
Fix $\epsilon > 0, \nu \in \mathcal{M}_+$ and $q \in \bR^D_{\geq 0} \setminus \{\vec{0}\}$. Then
\[\frac{X_n^{\epsilon,q, \nu}}n:= \frac{d^{\epsilon}((\vec{0}, 0), (nq, n\nu))}n \rightarrow X^{\epsilon,q, \nu} := \sup \limits_n \frac{EX_n^{\epsilon,q, \nu}}n = \lim \limits_{n \rightarrow \infty} \frac{EX_n^{\epsilon,q, \nu}}n \ \mbox{a.s.}\]
\end{theorem}
\begin{remark}
As noted before, Theorem \ref{thm7} holds trivially when $q \notin \bR^D_{\geq 0} \setminus \{\vec{0}\}$, so we need not bother with this case.
\end{remark}

 We prove this theorem in stages, starting with the case when $q$ has integer coordinates. But first, we show a useful bound on our random variables $X_n^{\epsilon, q, \nu}$.

\begin{lemma}\label{lemma10}
Let $\epsilon > 0, \nu, \xi \in \mathcal{M}_+$ and $p, q \in \bZ^D$ with $q-p \in \bZ^D_{\geq 0} \setminus \{\vec{0}\}$. Then 
\[ d^{\epsilon}((p, \xi), (q, \nu))  \in  \bigg[ - \frac{1}{\epsilon} (||q-p||_1 + ||\nu-\xi||_{TV}), ||q-p||_1 \log D \bigg]\]
\end{lemma}

\begin{proof}
\noindent Recall that any path $\pi:  p  \rightarrow  q $ has $|| q-p ||_1 $ edges, so $||\mu_{\pi}||_{TV} = ||q-p||_1$, and that the total number of paths $\pi: p \rightarrow q$ is 
\[\binom{||q-p||_1}{q_1-p_1, \cdots, q_D - p_D} \in [1, D^{ ||q-p||_1}]\]
Also, for any such $\pi$ we have  by Lemma \ref{lemma3}
\[\rho(\mu_{\pi}, \nu-\xi) \in [0, ||\mu_{\pi} - (\nu-\xi)||_{TV}] \subseteq [0,||q-p||_1 + ||\nu-\xi||_{TV}]\]
\[\Rightarrow e^{-\frac1{\epsilon} \rho(\mu_{\pi}, \nu-\xi) } \in [e^{-\frac{1}{\epsilon} (||q-p||_1 + ||\nu-\xi||_{TV})}, 1]\]
Thus
\[ d^{\epsilon}((p, \xi), (q, \nu)) = \log \sum_{\pi \in \mathcal{P}(p,q)}  e^{-\frac1{\epsilon} \rho(\mu_{\pi}, \nu-\xi) }  \in  \bigg[ - \frac{1}{\epsilon} (||q-p||_1 + ||\nu-\xi||_{TV}), ||q-p||_1 \log D \bigg]\]
\end{proof}

\begin{lemma}\label{lemma11}
Theorem \ref{thm9} holds with $q \in \bZ^D_{\geq 0} \setminus \{\vec{0}\}$.
\end{lemma}

\begin{proof}
We wish to use Kingman's Subadditive Ergodic Theorem (Theorem \ref{thmKingman}) with 
\[Y_{m,n} := -d^{\epsilon}((mq, m\nu), (nq, n\nu)) \ \forall m \leq n\]
Let us now check the conditions (i)-(iv).

By Lemma \ref{lemma10}, 
\[Y_{0,n} =-d^{\epsilon}((\vec{0}, 0), (nq, n\nu))  \in  \bigg[-n||q||_1 \log D,  \frac{1}{\epsilon} (n||q||_1 + n||\nu||_{TV}) \bigg]\]
hence (i) holds.

Next, for every $k \geq 1$, 
the sequence 
\[Y_{nk, (n+1)k} = -d^{\epsilon}((nkq,  nk\nu), ((n+1)kq, (n+1)k\nu))\]
is i.i.d. because  the distribution of the unnormalized empirical measures of paths $\pi: nkq \rightarrow (n+1)kq$ is not dependent on $n$ (since the edge labels $U_e$ are i.i.d.) so the distribution of the cost functions $e^{-\frac1{\epsilon} \rho(\mu_{\pi}, k\nu)}$ for $\pi: nkq \rightarrow (n+1)kq$ is not dependent on $n$. Thus (ii) holds. Furthermore, as discussed in section \ref{section2.5}, $Y_{nk, (n+1)k}$ being i.i.d. implies the sequence is ergodic.

Similarly, the joint distributions of 
$$\{Y_{m, m+k}: k \geq 1\} = \{-d^{\epsilon}((mq,  m\nu), ((m+k)q, (m+k)\nu)): k \geq 1\}$$
are not dependent on $m$ since the edge labels are i.i.d. and the distribution of empirical measures of paths in a rectangle on the lattice with the difference between the top right and bottom left corners being $kq$ is independent of the location of the rectangle. Thus we have (iii).

It remains to show (iv), namely to show that given $m, n > 0,$
\[d^{\epsilon}((\vec{0}, 0), ((m+n)q, (m+n)\nu)) \geq d^{\epsilon}((\vec{0}, 0), (mq, m\nu)) + d^{\epsilon}((mq, m\nu), ((m+n)q, (m+n)\nu)) \]
 For any paths $\pi: \vec{0} \rightarrow mq$ and $\pi': mq \rightarrow (m+n)q$, we get a unique concatenation $\pi\cdot \pi': \vec{0} \rightarrow mq \rightarrow (m+n)q$. Its  empirical measure satisfies $\mu_{\pi \cdot \pi'} = \mu_{\pi} + \mu_{\pi'}$. But  the Levy-Prokhorov metric satisfies  subadditivity by Lemma \ref{lemma4}. Thus
 \[\rho(\mu_{\pi \cdot \pi'}, (m+n) \nu) \leq \rho(\mu_{\pi}, m\nu) + \rho(\mu_{\pi'}, n\nu)\]
But then
\begin{align*}
-Y_{0,m} - Y_{m,m+n}  & = \log \sum_{\pi : \vec{0} \rightarrow mq} e^{-\frac1{\epsilon} \rho(\mu_{\pi}, m\nu) }  + \log \sum_{\pi' : mq \rightarrow (m+n)q} e^{-\frac1{\epsilon} \rho(\mu_{\pi'}, n\nu) }\\
 & = \log \sum_{\substack{\pi : \vec{0} \rightarrow mq\\ \pi' : mq \rightarrow (m+n)q}}  e^{-\frac{1}{\epsilon} \rho( \mu_{\pi}, m\nu) -\frac{1}{\epsilon} \rho( \mu_{\pi'}, n\nu)} \\
  & \leq \log \sum_{\substack{\pi : \vec{0} \rightarrow mq\\ \pi' : mp \rightarrow (m+n)q}} e^{-\frac1{\epsilon} \rho( \mu_{\pi \cdot \pi'}, (m+n)\nu)} 
\end{align*}
Not all paths $\pi''': \vec{0} \rightarrow (m+n)q$ pass through $mq$ so we can upper bound the expression above by removing this condition:
\begin{align*}
-Y_{0,m} - Y_{m,m+n}  & \leq \log \sum_{\pi''': \vec{0} \rightarrow (m+n)q} e^{-\frac1{\epsilon} \rho( \mu_{\pi'''}, (m+n)\nu)} \\
 & = -Y_{0, m+n}
\end{align*}

 Thus we can apply the Subadditive Ergodic Theorem (Theorem \ref{thmKingman}) to get that 
 \[\frac{-Y_{0,n}}n = \frac{X_n^{\epsilon, q, \nu}}n\]
 converges a.s. to the  constant 
\[X^{\epsilon, q, \nu} := \sup \limits_n \frac{EX_n^{\epsilon, q, \nu}}n = \lim \limits_{n \rightarrow \infty} \frac{EX_n^{\epsilon,q, \nu}}n \in \bigg[ - \frac1{\epsilon} \max(||q||_1, ||\nu||), ||q||_1 \log D \bigg]\]
\end{proof}

The next order of business is proving the theorem for $q$ with rational coordinates. We will find useful the following error estimate on how $X_n^{\epsilon,q,\nu}$ changes when $q$ is perturbed in the SE direction and $\nu$ is perturbed arbitrarily.

\begin{lemma}\label{lemma12}
Fix $q \in \bR^D_{\geq 0}, \epsilon > 0$ and $\nu, \xi \in \mathcal{M}_+$. Then for any $p \in \bR^D_{\geq 0}$ with $q-p \in \bR^D_{\geq 0}$, 
\begin{align*} 
&X_n^{ \epsilon, p, \xi} - \frac1{\epsilon} (n||q-p||_1 + n||\nu-\xi||_{TV} +D) \\
&\leq X_n^{ \epsilon, p, \xi} - \frac1{\epsilon} (||\lfloor nq \rfloor - \lfloor np \rfloor||_1 + n\rho(\nu,\xi)) \\
& \leq X_n^{\epsilon,  q, \nu} 
\end{align*}
\end{lemma}

\begin{proof}
Fix any such $p$. The inequality holds trivially if $p = q$ so we may assume $||q-p||_1 > 0$. Then there is at least one path $\pi': \lfloor np \rfloor \rightarrow \lfloor nq \rfloor$, and we fix it.

For any path $\pi: \vec{0} \rightarrow \lfloor np \rfloor$, we concatenate it with  $\pi'$ to get a unique path $\pi \cdot \pi': \vec{0} \rightarrow \lfloor nq\rfloor$. Note that $\pi'$ consists of $||\lfloor nq \rfloor - \lfloor np \rfloor ||_1 \leq n||q-p||_1+D$ edges so its 
empirical measure $\mu_{\pi'}$ has Total Variation at most $n||q-p||_1+D$. Thus $\pi \cdot \pi'$ satisfies
\begin{align*}
\rho(\mu_{\pi \cdot \pi'}, n\nu) &\leq  \rho(\mu_{\pi}+\mu_{\pi'}, \mu_{\pi}) + \rho( \mu_{\pi}, n\xi) + \rho(n\xi, n\nu)\\
&\leq   ||\mu_{\pi'}||_{TV} +\rho(\mu_{\pi}, n\xi) + n\rho(\nu,\xi)  \\
&\leq \rho(\mu_{\pi}, n\xi)+  (||\lfloor nq \rfloor - \lfloor np \rfloor||_1+ n \rho(\nu,\xi))
\end{align*}
It follows that
\begin{align*} 
&X_n^{ \epsilon, p, \xi} - \frac1{\epsilon} (n||q-p||_1 + n||\nu-\xi||_{TV} +D) \\
&\leq X_n^{ \epsilon, p, \xi} - \frac1{\epsilon} (||\lfloor nq \rfloor - \lfloor np \rfloor||_1 + n\rho(\nu,\xi)) \\
& \leq X_n^{\epsilon,  q, \nu} 
\end{align*}
\end{proof}

Using this lemma to approximate $X_n^{\epsilon, q, \nu}$ for $q \in \bQ^D_{\geq 0}$ in terms of $X_n^{\epsilon, p, \nu}$ where $p \in \bZ^D_{\geq 0}$, we prove our limit theorem for $q$ with rational coordinates.

\begin{lemma}\label{lemma13}
Theorem \ref{thm9} holds for $q \in \bQ^D_{\geq 0} \setminus \{\vec{0}\}$.
\end{lemma}

\begin{proof}
Let $q = \frac{(s_1, \ldots, s_D)}t$ for $s_i, t \in \bZ_{\geq 0}$. First,   compare $X_{tr}^{\epsilon, q, \nu},  X_{tr+1}^{\epsilon, q, \nu},  \ldots,  X_{tr+ (t-1)}^{\epsilon, q, \nu}, X_{t(r+1)}^{\epsilon, q, \nu}$ for arbitrary $r$ using the previous lemma. The idea is that $trq$ has integer coordinates so $\frac{X_{tr}^{\epsilon, q, \nu}}{tr}, \frac{X_{t(r+1)}^{\epsilon, q, \nu}}{tr}$ have the desired limits by Lemma \ref{lemma11}. The rest of the $X_{tr+j}^{\epsilon, q, \nu}$ are bounded above/below by $X_{t(r+1)}^{\epsilon, q, \nu}/X_{tr}^{\epsilon, q, \nu}$ respectively plus some small error terms which go to 0 as $r \rightarrow \infty$.

 Consider any $1 \leq j \leq t$. Note that 
 \[X_{tr+j-1}^{\epsilon, q, \nu} = d^{\epsilon}((\vec{0}, 0), ((tr+j-1)q, (tr+j-1)\nu))  = X_{tr+j}^{\epsilon, \frac{tr+j-1}{tr+j} q, \frac{tr+j-1}{tr+j}\nu} \]
and
 \[ \bigg|\bigg|q-\frac{tr+j-1}{tr+j} q \bigg|\bigg|_1 + \bigg|\bigg|\nu-\frac{tr+j-1}{tr+j} \nu \bigg|\bigg|_{TV} = \frac{||q||_1+||\nu||_{TV}}{tr+j}\]
By Lemma \ref{lemma12},
\[X_{tr+j-1}^{\epsilon, q, \nu}= X_{tr+j}^{\epsilon, \frac{tr+j-1}{tr+j} q, \frac{tr+j-1}{tr+j}\nu} \leq  X_{tr+j}^{\epsilon, q, \nu} + \frac1{\epsilon} (||q||_1+ ||\nu||_{TV} + D)\]
Thus
\begin{equation} \label{1}
\begin{split}
\frac{X_{tr}^{\epsilon,q,\nu}}{r} &\leq  \frac{X_{tr+1}^{\epsilon,q, \nu}}{r} + \frac{||q||_1+||\nu||_{TV}+D}{\epsilon r} \\
& \leq \ldots \leq \frac{X_{t(r+1)}^{\epsilon,q, \nu}}{r} + \frac{t(||q||_1+||\nu||_{TV}+D)}{\epsilon r}
\end{split}
\end{equation}
But $tq \in \bZ_{\geq 0}^D$ so by Lemma \ref{lemma10}, our limit theorem holds for $X_{tr}^{\epsilon, q, \nu} = X_r^{\epsilon, tq, t\nu}$. That is, 
\begin{equation} \label{2}
\frac{X_{tr}^{\epsilon, q,\nu}}{r} \rightarrow  \sup_r \frac{EX_{tr}^{\epsilon, q, \nu}}{r} = \lim_{r \rightarrow \infty} \frac{EX_{tr}^{\epsilon, q, \nu}}{r}  \ \mbox{a.s.} 
\end{equation}
Taking expectations in \eqref{1}, then taking the supremum/limit  as $r \rightarrow \infty$ and using \eqref{2} we get
\begin{equation} \label{3}
\lim_{r \rightarrow \infty}  \frac{EX_{tr+j}^{\epsilon,q,\nu}}r = \sup_r \frac{EX_{tr+j}^{\epsilon,q, \nu}}r =  \lim_{r \rightarrow \infty} \frac{EX_{tr}^{\epsilon, q, \nu}}{r} = \sup_r \frac{EX_{tr}^{\epsilon, q, \nu}}{r} \ \forall 0 \leq j \leq t-1
\end{equation}
since the error terms in \eqref{1} go to 0.
Similarly, taking the limit as $r \rightarrow \infty$ in \eqref{2} and using \eqref{1}, \eqref{3} we get 
$$\lim_{r \rightarrow \infty}  \frac{X_{tr+j}^{\epsilon,q, \nu}}r = \sup_r \frac{EX_{tr}^{\epsilon, q, \nu}}{r} = \lim_{r \rightarrow \infty}  \frac{EX_{tr+j}^{\epsilon,q, \nu}}r = \sup_r \frac{EX_{tr+j}^{\epsilon,q, \nu}}r \ \mbox{a.s.} \ \forall 0 \leq j \leq t-1$$
Multiplying everything by $\frac1t$ and using the fact that any $n$ can be written as $tr+j$, we get
$$\frac{X^{\epsilon,q, \nu}_{n}}{n} \rightarrow  \lim_{n \rightarrow \infty} \frac{EX_n^{\epsilon,q, \nu}}n  = \sup_n \frac{EX_n^{\epsilon,q, \nu}}n \ \mbox{a.s.} $$
\end{proof}

 Our next objective is to prove the full version of Theorem \ref{thm9}. We will need two short lemmas giving sufficient conditions for a.s. convergence of bounded random variables to the supremum of their expectations. The first of these lemmas looks at the case where we are given a particular lower bound for the liminf of all subsequences of our sequence.

\begin{lemma}\label{lemma14}
Let $Y_n$ be uniformly bounded random variables. If a.s. 
\[\sup_n EY_{n} \leq \liminf_{i \rightarrow \infty} Y_{n_i} \ \mbox{for all subsequences $(n_i)$}\]
then $\lim \limits_{n \rightarrow \infty} Y_n = \sup \limits_n EY_n = \lim \limits_{n \rightarrow \infty} EY_n$ a.s.
\end{lemma}

\begin{proof}
First, taking expectations in 
\[ \sup_n EY_{n} \leq \liminf_{i \rightarrow \infty} Y_{n_i} \ \mbox{a.s.}\]
and using Fatou's Lemma, we get
\begin{align*}
&\sup_n EY_{n} \leq E\liminf_{i \rightarrow \infty} Y_{n_i} \leq \liminf_{i \rightarrow \infty} EY_{n_i} \\
&\leq \limsup_{i \rightarrow \infty} EY_{n_i} \leq \sup_i EY_{n_i} \leq \sup_n EY_n
\end{align*}
so all the inequalities are equalities and
\[    \sup_n EY_{n} = \lim \limits_{i \rightarrow \infty} EY_{n_i} = \liminf_{i \rightarrow \infty} Y_{n_i} \ \mbox{a.s.}\]
This holds for all subsequences $(n_i)$, including the full sequence. Thus a.s.
\[L:= \sup_n EY_{n} = \lim \limits_{n \rightarrow \infty} EY_{n} = \liminf_{i \rightarrow \infty} Y_{n_i} \ \mbox{ for all subsequences $(n_i)$}\]
hence $Y_n \rightarrow L$ a.s..
\end{proof}

 The next lemma is about approximating a sequence of random variables from below by sequences that each converge to the supremum of their expectations.

\begin{lemma}\label{lemma15}
Let  $Y_n$ be uniformly bounded random variables s.t. 
\begin{enumerate}[label=(\roman*)]
\item For every fixed $k$, $Y_{n,k} \leq Y_n$ for large enough $n$
\item For every fixed $k$,
\[\lim_{n \rightarrow \infty} Y_{n,k} = \sup_n EY_{n,k} = \lim_{n \rightarrow \infty} EY_{n,k} \ \mbox{a.s.} \]
\item $\sup \limits_n EY_n = \sup \limits_n \sup \limits_k EY_{n,k}$
Then $\lim \limits_{n \rightarrow \infty} Y_n = \sup \limits_n EY_n = \lim \limits_{n \rightarrow \infty} EY_n$ a.s.
\end{enumerate}
\end{lemma}

\begin{proof}
We wish to prove the hypothesis of Lemma \ref{lemma14} for $Y_n$. Let 
\[\mathcal{F} = \bigg\{\lim_{n\rightarrow \infty} Y_{n,k} \neq \sup_n EY_{n,k} \ \mbox{for some $k$}\bigg\}\]
which has measure 0 by (ii). We claim that in the event $\mathcal{F}^C$ we have
\[\sup_n EY_{n} \leq \liminf_{i \rightarrow \infty} Y_{n_i} \ \mbox{for all subsequences $(n_i)$}\]
Consider any subsequence $(n_i)$ and any $k$. By (i), $Y_{n_i, k} \leq Y_{n_i} \ \mbox{for large enough $i$}$. Taking the $\liminf$ over $i$, and using the fact that we are in the event $\mathcal{F}^C$, we get
\[ \sup_n EY_{n, k} = \lim_{n \rightarrow \infty} Y_{n, k} = \lim_{i \rightarrow \infty} Y_{n_i, k} \leq \liminf_{i \rightarrow \infty} Y_{n_i} \]
This holds for all $k$. Taking the supremum over $k$ and using (iii), we get
\[ \sup_n EY_n = \sup_n \sup \limits_{k} EY_{n,k} = \sup \limits_{k} \sup_n EY_{n,k}   \leq \liminf_{i \rightarrow \infty} Y_{n_i}  \]
as desired.
 Applying Lemma \ref{lemma14}, we get
\[  \lim_{n \rightarrow \infty}  EY_n = \sup_n  EY_n = \lim_{n \rightarrow \infty} Y_n  \ \mbox{a.s.}\]
\end{proof}

We can finally prove Theorem \ref{thm9}.

\begin{proof}[Proof of Theorem \ref{thm9}]
$ $\newline
The case $q \in \bQ^D_{\geq 0} \setminus \{\vec{0}\}$ is handled by Lemma \ref{lemma13}. So we may assume $q \notin \bQ^D_{\geq 0}$.

We construct a sequence $p_k \in \bQ_{\geq 0}^D$ as follows. For every $1 \leq j \leq D$, either $q_j \in \bQ$ and we pick $(p_k)_j := q_j$, or $q_j \notin \bQ$ and we pick $(p_k)_j \in \bQ_{\geq 0}$ s.t. $(p_k)_j \uparrow q_j$. It follows that $p_k \in \bQ_{\geq 0}^D$ with $p_{k+1} - p_k, q - p_k \in \bQ^D_{\geq 0} \ \forall k$. That is, $p_1, p_2, \ldots, q$ forms a "staircase" in $\bR^D$.

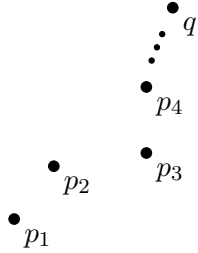
\begin{center}
\begin{tikzpicture}
\node (p1) at (0pt,0pt) {};
\node (p5) at (15pt,25pt) {};
\node (p2) at (50pt,25pt) {};
\node (p2) at (50pt,50pt) {};
\node (p3) at (60pt,80pt) {};
\filldraw (0pt,0pt)circle(2pt) (50pt,25pt)circle(2pt) (15pt,20pt)circle(2pt) (50pt,50pt)circle(2pt)  (52pt,60pt)circle(1pt) (54pt,65pt)circle(1pt) (56pt,70pt)circle(1pt)   (60pt,80pt)circle(2pt);
\draw (0pt,0pt) node [ below right] {$p_1$};
\draw (15pt,20pt) node [ below right] {$p_2$};
\draw (50pt,25pt) node [ below right] {$p_3$};
\draw (50pt,50pt) node [ below right] {$p_4$};
\draw (60pt,80pt) node [ below right] {$q$};
\end{tikzpicture}
\end{center}

The intuition is that $X_n^{\epsilon, p_k, \nu}$ approximates $X_n^{\epsilon, q, \nu}$ from below. Each $\frac{X_n^{\epsilon, p_k, \nu}}n$ converges in $n$ to the right limit a.s. by Lemma \ref{lemma13}, and we use Lemma \ref{lemma15} to prove that $\frac{X_n^{\epsilon,q,\nu}}n$ converges a.s. to the right limit.

Let us now prove the hypothesis of Lemma \ref{lemma15} with 
\[Y_{n,k} := \frac{X_n^{\epsilon, p_k,\nu}}n - \frac{||\lfloor nq\rfloor - \lfloor np_k \rfloor||_1}{\epsilon}, Y_n := \frac{X_n^{\epsilon, q, \nu}}n\]
Note that these random variables are bounded by Lemma \ref{lemma10}:
\[ Y_{n,k}, Y_n \in \bigg[-\frac1{\epsilon} (||q||_1 + ||\nu||_{TV}) - \frac{||q||_1}{\epsilon} , ||q||_1 \log D \bigg] \]
 First, fix $n$ and consider any $k$. By Lemma \ref{lemma12},
\[X_n^{\epsilon, p_k, \nu} - \frac{||\lfloor nq\rfloor - \lfloor np_k \rfloor||_1}{\epsilon} \leq X_n^{\epsilon,q, \nu} = Y_n \]
so we have (i). Next, (ii) follows from Lemma \ref{lemma13}:
\begin{align*}
&\lim_{n \rightarrow \infty} Y_{n,k} = \lim_{n \rightarrow \infty} \frac{X_n^{\epsilon, p_k,\nu} - \frac{||\lfloor nq\rfloor - \lfloor np_k \rfloor||_1}{\epsilon}}n  = \lim_{n \rightarrow \infty} \frac{X_n^{\epsilon, p_k, \nu}}n - \frac{||q-p_k||_1}{\epsilon}  \\
&= \sup_n \frac{EX_n^{\epsilon, p_k, \nu}}n - \frac{||q-p_k||_1}{\epsilon} = \lim_{n \rightarrow \infty} \frac{EX_n^{\epsilon, p_k, \nu}}n - \frac{||q-p_k||_1}{\epsilon} \\
&=  \lim_{n \rightarrow \infty} EY_{n,k} = \sup_n EY_{n,k}  
\end{align*}

It remains to show (iii). Consider any fixed $n$. For each $1 \leq j \leq D$ either $q_j \in \bQ$ and $q_j = (p_k)_j \ \forall k$, or $nq_j \notin \bZ$ and $(p_k)_j \uparrow q_j$ hence $\lfloor nq_j \rfloor = \lfloor n(p_k)_j \rfloor$ for large enough $k$. Thus $X_n^{\epsilon, p_k, \nu} = X_n^{\epsilon, q, \nu}$ for large enough $k$. It follows that
\[Y_{n,k} = X_n^{\epsilon, p_k, \nu} -  \frac{||\lfloor nq\rfloor - \lfloor np_k \rfloor||_1}{\epsilon} \uparrow X_n^{\epsilon, q, \nu} = Y_n \ \mbox{as} \ k \rightarrow \infty\]
By the Bounded Convergence Theorem, $EY_{n, k} \uparrow EY_n$. Thus
\[\lim_{k \rightarrow \infty} EY_{n,k} = \sup_k EY_{n,k}= EY_n \ \forall n \]
so taking the supremum over $n$ we get (iii):
\[\sup_n \sup_k EY_{n,k} = \sup_n EY_n\]

 Applying Lemma \ref{lemma15}, we get
\[\lim_{n \rightarrow \infty} Y_n = \lim_{n \rightarrow \infty} EY_n = \sup_n EY_n \ \mbox{a.s.}\]
i.e. 
\[ \lim_{n \rightarrow \infty} \frac{X_n^{\epsilon, q, \nu}}n = \lim_{n \rightarrow \infty}  \frac{EX_n^{\epsilon, q,\nu}}n = \sup_n  \frac{EX_n^{\epsilon, q, \nu}}n  \ \mbox{a.s.}\]
\end{proof}

This completes the proof of the existence of the limit shape of $d^{\epsilon}$ when measuring the distance from $(\vec{0}, 0) \in \bR^D \times \mathcal{M}_+$.

\subsection{The Limit Shape of \texorpdfstring{$d^{\epsilon}$}{Lg}—The General Case}
We use Lemma \ref{lemma15} to generalize Theorem \ref{thm9} to Theorem \ref{thm7}, where $d^{\epsilon}$ measures distance between any two elements of $\bR^D \times \mathcal{M}_+$.

\begin{manualtheorem}{\ref{thm7}}
Fix $\epsilon > 0, \nu, \xi \in\mathcal{M}_+$ and $p, q \in \bR^D$. Then
\[\frac{d^{\epsilon}((np, n\xi), (nq, n\nu))}n\]
converges in probability to a constant. When $p = \vec{0}$, the convergence is pointwise a.s.
\end{manualtheorem}

\begin{remark}
During the course of the proof, we also show that the limit shape of \\ $\frac{d^{\epsilon}((np, n\xi), (nq, n\nu))}n$ is translation-invariant. This will help us later.
\end{remark}

\begin{proof}
As noted before, the theorem holds trivially if $q - p \notin \bR^D_{\geq 0} \setminus \{\vec{0}\}$. Thus we may assume $q - p \in \bR^D_{\geq 0} \setminus \{\vec{0}\}$. 

 Observe that
\begin{equation}\label{EQ3} 
\begin{split}
d^{\epsilon}((np, n\alpha), (nq, n\nu)) &= d^{\epsilon}((np, 0), (nq, n\nu - n\xi)) \\
&=^d d^{\epsilon}((\vec{0}, 0), (\lfloor nq \rfloor - \lfloor np \rfloor, n\nu-n\xi)) 
\end{split}
\end{equation}
so it suffices to assume $\xi = 0$ and prove 
\[\frac{X_n^{\epsilon, \frac{\lfloor nq \rfloor - \lfloor np \rfloor}n, \nu}}n = \frac{d^{\epsilon}((\vec{0}, 0), (\lfloor nq \rfloor - \lfloor np \rfloor, n\nu))}n\]
converge a.s. to a constant.

Our argument mirrors the one used in the proof of Theorem \ref{thm9}: we approximate these $Y_n$ from below by some $Y_{n,k}$ converging in $k$ and we apply Lemma \ref{lemma15}.

First we prove some inequalities. Fix $n$. Observe that for $1 \leq j \leq D$,
\[n(q_j-p_j) -1  < \lfloor n (q_j-p_j) \rfloor \leq n(q_j-p_j) \]
\[n(q_j-p_j) - 1 < \lfloor nq_j \rfloor - \lfloor np_j \rfloor < n(q_j-p_j) + 1\]
\begin{equation}\label{4}
 \Rightarrow (\lfloor nq_j \rfloor - \lfloor np_j \rfloor) - \lfloor n (q_j-p_j) \rfloor \in \{0,1\}
 \end{equation}
 and thus
\begin{equation}\label{5}
(\lfloor nq \rfloor - \lfloor np \rfloor) - \lfloor n (q-p) \rfloor \in \bR^D_{\geq 0} \ \mbox{with} \    || (\lfloor nq \rfloor - \lfloor np \rfloor) - \lfloor n (q-p) \rfloor ||_1 \leq D
\end{equation}
On the other hand, for $1 \leq j \leq D$ s.t. $p_j \neq q_j$,
\[ \lfloor nq_j \rfloor - \lfloor np_j \rfloor \leq n(q_j - p_j) + 1 \leq  (n+c) (q_j-p_j) \]
for $c:= \bigg\lceil \max \limits_{1 \leq j \leq D \ \mbox{s.t.} \ p_j \neq q_j} \bigg(\frac1{q_j-p_j} \bigg) \bigg\rceil$. Thus
\[  0 \leq \lfloor (n+c) (q_j-p_j) \rfloor - (\lfloor nq_j \rfloor - \lfloor np_j \rfloor) \leq c(q_j-p_j) \ \mbox{by} \ \eqref{4}\]
This equation also holds trivially for $j$ s.t. $p_j = q_j$. Thus
\begin{equation}\label{6}
\lfloor (n+c) (q-p) \rfloor - (\lfloor nq \rfloor - \lfloor np \rfloor)  \in \bR^D_{\geq 0} \ \mbox{with} \  
||\lfloor (n+c) (q-p) \rfloor - (\lfloor nq \rfloor - \lfloor np \rfloor)  ||_1 \leq c||q-p||_1
\end{equation}

 We prove the hypothesis of Lemma \ref{lemma15} with 
 \[Y_{n,k} := \frac{X_n^{\epsilon, q-p, \nu} - \frac{n}{\epsilon k}}n, Y_n := \frac{X_n^{\epsilon, \frac{\lfloor nq \rfloor - \lfloor np \rfloor}n, \nu}}n\]
 First, by Lemma \ref{lemma12} and \eqref{5}, for every fixed $k$, for large enough $n$ we have
\begin{align*} 
\frac{X_n^{\epsilon, q-p, \nu} - \frac{n}{\epsilon k}}n &\leq \frac{X_n^{\epsilon, q-p, \nu} - \frac{D}{\epsilon}}n \\
&\leq \frac{X_n^{\epsilon, q-p, \nu} - \frac{1}{\epsilon}|| (\lfloor nq \rfloor - \lfloor np \rfloor) - \lfloor n (q-p) \rfloor ||_1}n \\
&\leq \frac{X_n^{\epsilon, \frac{\lfloor nq \rfloor - \lfloor np \rfloor}n, \nu}}n
\end{align*}
i.e. $Y_{n,k} \leq Y_n$ which is precisely (i). Also, Theorem \ref{thm9} gives (ii): for every fixed $k$, a.s. 
\[\lim_{n\rightarrow \infty} Y_{n,k} = \lim_{n\rightarrow \infty} \frac{X_n^{\epsilon, q-p, \nu} - \frac{n}{\epsilon k}}n = \sup_n \frac{EX_n^{\epsilon, q-p, \nu}}n - \frac1{\epsilon k}= \lim_{n \rightarrow \infty} \frac{EX_n^{\epsilon, q-p, \nu}}n - \frac1{\epsilon k} \]
\[= \sup_n EY_{n,k} = \lim_{n \rightarrow \infty} EY_{n,k} \]
Note that this implies
\begin{equation}\label{7}
\sup_n \frac{EX_n^{\epsilon, q-p, \nu}}n = \sup_n \sup_k EY_{n,k} = \sup_k \sup_n EY_{n,k} =  \sup_k \lim_{n \rightarrow \infty} EY_{n,k} = \lim_{n \rightarrow \infty} \frac{EX_n^{\epsilon, q-p, \nu}}n
\end{equation}
It remains to show (iii). Note that taking expectations in (i) immediately gives
\begin{equation}\label{8}
\sup_n \sup_k EY_{n,k} = \sup_k \sup_n EY_{n,k} \leq \sup EY_n
\end{equation}
We show this is an equality. By Lemma \ref{lemma12} and \eqref{6}, for every fixed $k$, for large enough $n$ we have
\begin{align*}
 X_n^{\epsilon,\frac{\lfloor nq \rfloor - \lfloor np \rfloor}{n}, \nu} & \leq X_n^{\epsilon, \frac{n+c}n (q-p), \frac{n+c}n\nu} + \frac1{\epsilon} \bigg(||\lfloor (n+c) (q-p) \rfloor - (\lfloor nq \rfloor - \lfloor np \rfloor)  ||_1+ c||\nu ||_{TV} \bigg) \\
&\leq X_{n+c}^{\epsilon, (q-p),  \nu} + \frac{c||q-p||_1+c||\nu||_{TV}}{\epsilon}
\end{align*}
Dividing by $n$, taking expectations and taking the supremum over $n$ we get by \eqref{7}
\begin{equation}\label{eqn9}
\sup_n EY_n \leq \sup_n \frac{EX_{n+c}^{\epsilon, q-p, \nu}}n  = \lim_{n \rightarrow \infty} \frac{EX_n^{\epsilon, q-p, \nu}}n = \sup_n \sup_k EY_{n,k} 
\end{equation}
which when combined with \eqref{8} gives (iii). 

 Applying Lemma \ref{lemma15}, we get 
\[ \lim_{n \rightarrow \infty} Y_n = \lim_{n \rightarrow \infty} EY_n = \sup_n EY_n \ \mbox{a.s.}\]
i.e.
\[\lim \limits_{n \rightarrow \infty} \frac{X_n^{\epsilon, \frac{\lfloor nq \rfloor - \lfloor np \rfloor}n}}n = \sup \limits_{n } \frac{EX_n^{\epsilon, \frac{\lfloor nq \rfloor - \lfloor np \rfloor}n}} n = \lim_{n \rightarrow \infty} \frac{EX_n^{\epsilon, \frac{\lfloor nq \rfloor - \lfloor np \rfloor}n}} n \ \mbox{a.s.}\]
 Furthermore, by \eqref{7}, \eqref{8} and \eqref{eqn9}, we have
\[\sup_n EY_n = \sup_n \sup_k EY_{n,k} = \sup_n \frac{EX_n^{\epsilon, q-p, \nu}}n = \lim_{n \rightarrow \infty} \frac{EX_n^{\epsilon, q-p, \nu}}n \]
Combining this with \eqref{EQ3}, we get that $\frac{d^{\epsilon}((np, n\xi), (nq, n\nu))}n$ converges in probability to the a.s. limit of $ \frac{d^{\epsilon}((\vec{0}, 0), (nq-np, n\nu - n\xi))}n$.

This completes the proof of Theorem \ref{thm7} and of translation-invariance of the limit shape.

\end{proof}

\subsection{Grid Entropy as a Directed Norm}
In the previous sections we showed the existence of the limit shape of $d^{\epsilon}$. We now take the infimum as $\epsilon \downarrow 0$ and we show that the result is a directed metric with negative sign that gives rise to a norm, which we call grid entropy.

\begin{manualtheorem}{\ref{thm9}}
For $\epsilon > 0$, $\nu, \xi \in \mathcal{M}_+$ and $p, q \in \bR^D$ define 
$$\widetilde{d}^{\epsilon}((p, \xi), (q, \nu)) := \lim \limits_{n \rightarrow \infty} \frac{d^{\epsilon}((np,n\xi), (nq, n\nu))}n, \ \mbox{and} $$
$$\widetilde{d}((p, \xi), (q, \nu)) := \inf_{\epsilon > 0} \widetilde{d}^{\epsilon}((p, \xi), (q, \nu)) \in [-\infty, \infty)$$
Then each $\widetilde{d}^{\epsilon}$ as well as $\widetilde{d}$ are  directed metrics with negative sign on $\bR^D \times \mathcal{M}_+$. 
\end{manualtheorem}

\begin{remark}
For any $p, q \in \bR^D, \epsilon > 0$ and $\nu, \xi \in \mathcal{M}_+$, $d^{\epsilon}((np, n\xi), (nq, n\nu))$ is monotone decreasing as $\epsilon \downarrow 0$ so 
\[\widetilde{d}((p, \xi), (q, \nu)) = \inf_{\epsilon > 0} \widetilde{d}^{\epsilon}((p, \xi), (q, \nu)) = \lim_{\epsilon \downarrow 0} \widetilde{d}^{\epsilon}((p, \xi), (q, \nu))\]
By Lemma \ref{lemma10}, for every $n$ and $\epsilon > 0$, 
\[d^{\epsilon}((np, n\xi), (nq, n\nu)) \in [-\infty, ||\lfloor nq \rfloor - \lfloor np \rfloor||_1 \log D]\]
so it follows that
\[ \widetilde{d}((p, \xi), (q, \nu)) \in [-\infty, ||q-p||_1 \log D]\]
Once we prove that our two definitions of grid entropy are equivalent, this bound will be improved.
\end{remark}

\begin{remark}
As was the case with Theorem \ref{thm9}, the limit $\widetilde{d}((p, \xi), (q, \nu))$ is trivially $-\infty$ if $q-p \notin \bR^D_{\geq 0}$ or if $q=p$ and $\nu \neq \xi$, and it is trivially 0 if $q=p$ and $\nu = \xi$.
\end{remark}

\begin{proof}
As noted above,
 \[\widetilde{d}^{\epsilon}((p, \xi), (p, \xi)) = \widetilde{d}((p, \xi), (p, \xi)) = 0 \ \forall \epsilon > 0\]

It remains to prove the reverse triangle inequality. Let $p,q,r \in \bR^D$ and $\nu,\xi, \eta \in \mathcal{M}_+$ and consider any $\epsilon > 0$ and any $n$. If $q-p \notin \bR^D_{\geq 0}$ or $r-q \notin \bR^D_{\geq 0}$ then the following inequality holds trivially (because the right-hand side is $-\infty$)
\begin{equation}\label{9}
d^{ \epsilon} ((np, n\xi),(nr, n\eta)) \geq d^{\epsilon} ((np, n\xi),(nq, n\nu)) + d^{\epsilon} ((nq, n \nu),(nr, n\eta)) 
\end{equation}
Now suppose $r-q, q-p \in \bR^D_{\geq 0}$. Given paths $\pi: \lfloor np \rfloor \rightarrow \lfloor nq \rfloor$, $\pi': \lfloor nq \rfloor \rightarrow \lfloor nr \rfloor$, we concatenate them to obtain a unique path $\pi \cdot \pi': \lfloor np \rfloor \rightarrow \lfloor nr \rfloor$ with unnormalized empirical measure $\mu_{\pi \cdot \pi'} = \mu_{\pi} + \mu_{\pi'}$. By the subadditivity of the Levy-Prokhorov metric (Lemma \ref{lemma4}),
\[\rho(\mu_{\pi \cdot \pi'}, n(\eta-\xi)) = \rho(\mu_{\pi} + \mu_{\pi'}, n(\eta-\nu) + n(\nu-\xi)) \leq \rho(\mu_{\pi}, n(\eta-\nu)) + \rho(\mu_{\pi'}, n(\nu-\xi))\]
so 
\begin{align*}
&\bigg(\sum_{\pi: \lfloor np \rfloor \rightarrow \lfloor nq \rfloor} e^{-\frac1{\epsilon} \rho(\mu_{\pi}, n(\eta - \nu) ) } \bigg) \bigg( \sum_{\pi': \lfloor nq \rfloor \rightarrow \lfloor nr \rfloor} e^{-\frac1{\epsilon} \rho(\mu_{\pi'}, n(\nu - \xi) ) }  \bigg) \\
&\leq \bigg(\sum_{\pi''': \lfloor np \rfloor \rightarrow \lfloor nr \rfloor} e^{-\frac1{\epsilon} \rho(\mu_{\pi'''}, n(\eta-\xi)) } \bigg)
\end{align*}
It follows that \eqref{9} holds. Dividing \eqref{9} by $n$ and taking the limit (in probability) as $n \rightarrow \infty$ we get
\[\widetilde{d}^{\epsilon} ((p, \xi),(r, \eta)) \geq \widetilde{d}^{\epsilon} ((p, \xi),(q, \nu)) + \widetilde{d}^{\epsilon} ((q, \nu),(r, \eta))\]
so $\widetilde{d}^{\epsilon}$ is a directed metric with negative sign. Taking the limit as $\epsilon \rightarrow 0^+$, we still obtain a directed metric with negative sign.
\end{proof}

We proceed to show that $\widetilde{d}$ gives rise to a directed norm with negative sign.

\begin{theorem}\label{thm16}
\begin{enumerate}[label=(\roman*)]
\item[]
\item Each $\widetilde{d}^{\epsilon}$ is translation-invariant and positive-homogeneous. So is $\widetilde{d}$. 
\item For $q \in \bR^D, \nu \in \mathcal{M}_+$ define the grid entropy with respect to $(q, \nu)$ to be
\[|| (q, \nu) || := \widetilde{d}((\vec{0}, 0), (q, \nu)) \]
Then this is a directed norm with negative sign on $\bR^D \times \mathcal{M}_+$. 
\end{enumerate}
\end{theorem}
\begin{remark}
From before, $||(q, \nu)||$ is $-\infty$ if $q \notin \bR^D_{\geq 0}$ or if $q = \vec{0}$ and $\nu \neq 0$, and it is 0 if $q = \vec{0}$ and $\nu = 0$.
\end{remark}

\begin{remark}
A directed metric with negative sign is clearly concave. Thus each $ \widetilde{d}^{\epsilon}$ as well as $||(\cdot, \cdot)||$ are concave functions on their respective domains.
\end{remark}

\begin{remark}
These properties of grid entropy also follow from \cite{rassoul2014quenched}.
\end{remark}

\begin{proof}
(i) Fix $\epsilon > 0$. We already showed that $\widetilde{d}^{\epsilon}$ is translation-invariant while proving Theorem 7.

By translation-invariance, it suffices to show that $\widetilde{d}^{\epsilon}((\vec{0},0), (q,\nu))$ is positive-homogeneous. Consider any $c = \frac{a}{b} \in \bQ_{> 0}$. Then
\[\widetilde{d}^{\epsilon}((\vec{0}, 0) , (cq, c\nu))=  \lim_{n \rightarrow \infty} \frac{d^{\epsilon} ((\vec{0}, 0), (cnq, cn\nu))}{n} \ \mbox{a.s.}\]
Looking at the subsequence consisting of multiples $n = mb$ of $b$, we get
\[\widetilde{d}^{\epsilon}((\vec{0}, 0), (cq, c\nu))=  \lim_{m \rightarrow \infty} \frac{d^{\epsilon} ((\vec{0}, 0), (amq, am\nu))}{mb} \ \mbox{a.s.}\]
But each $am \in \bN$ so
\[\widetilde{d}^{\epsilon}((\vec{0}, 0), (cq, c\nu))=\frac{a}b \lim_{n \rightarrow \infty} \frac{d^{\epsilon} ((\vec{0}, 0), (nq, n\nu))}{n} = c \widetilde{d}^{\epsilon}((\vec{0}, 0), (q,\nu)) \ \mbox{a.s.}\]
Thus $\widetilde{d}^{\epsilon}$ is positive-homogeneous for rational factors. 

Now consider any $c \in \bR_{> 0}$ and take sequences $a_k, b_k \in \bQ_{> 0}, a_k \uparrow c, b_k \downarrow c$. Consider any $n$, and let us look at 
\[X_n^{\epsilon, cq, c\nu} = d^{\epsilon}((\vec{0},0), (cq, c\nu))\]
By Lemma \ref{lemma12},
\[ X_n^{\epsilon, a_kq, a_k \nu } - \frac1{\epsilon} (n|c-a_k| \cdot (||q||_1 + ||\nu||_{TV}) ) \leq X_n^{\epsilon, cq, c\nu} \leq X_n^{\epsilon, b_kq, b_k \nu } + \frac1{\epsilon} (n|b_k-c| \cdot (||q||_1 + ||\nu||_{TV}) ) \]
Dividing by $n$ and taking a.s. limits, and applying homogeneity for positive rational factors we get
\begin{align*}
&a_k \widetilde{d}^{\epsilon}((\vec{0}, 0), (q, \nu)) - \frac{|c-a_k| \cdot (||q||_1 + ||\nu||_{TV})}{\epsilon} \\
& \leq  \widetilde{d}^{\epsilon}((\vec{0}, 0), (cq, c \nu)) \\
&\leq b_k \widetilde{d}^{\epsilon}((\vec{0}, 0), (q, \nu)) + \frac{|b_k-c| \cdot (||q||_1 + ||\nu||_{TV})}{\epsilon} 
\end{align*}
Taking $k \rightarrow \infty$ gives us 
\[\widetilde{d}^{\epsilon}((\vec{0}, 0), (cq, c \nu)) = c\widetilde{d}^{\epsilon}((\vec{0}, 0), (q, \nu))\]
so $\widetilde{d}^{\epsilon}$ is homogenous with respect to any positive real factor.

Taking the infimum over $\epsilon > 0$ we get that $\widetilde{d}$ is translation-invariant and positive\hyp{}homogenous.

(ii) Follows directly from (i).

\end{proof}

\subsection{Direction-free Grid Entropy}
We now wish to develop a grid entropy for the case where we no longer restrict ourselves to paths $\pi: \vec{0} \rightarrow \lfloor nq \rfloor$ for a given direction $q \in \bR^D$, and instead look at all length $\lfloor nt \rfloor$ paths from $\vec{0}$ for a given size parameter $t \geq 0$. Another way of putting this is that we look at paths from the origin to the line or "level" $x_1 +x_2+\ldots + x_D = \lfloor nt \rfloor$. Recall that the set of all such paths is denoted $\mathcal{P}_{\lfloor nt \rfloor}(\vec{0})$. 

If we try to simply repeat our previous argument, we run into a dead end because we are no longer in a superadditive setting. The solution is to observe that the distances $\widetilde{d}^{\epsilon}((\vec{0},0), (q, \nu))$ are maximized over $q \in \bR^D_{\geq 0}$ with $||q||_1 = t$ by $q = t(\frac1D, \ldots, \frac1D) := t \ell$. This intuitively makes sense, since this direction is the direction which has the most NE paths. 

\begin{lemma}\label{maximalLemma}
Fix $t \geq 0, \nu \in \mathcal{M}_t$. Then
\[\sup_{q \in \bR^D_{\geq 0}: ||q||_1 = t} \widetilde{d}^{\epsilon}((\vec{0}, 0), (q, \nu)) = \widetilde{d}^{\epsilon}((\vec{0}, 0), (t\ell, \nu)) \ \forall \epsilon > 0 \ \mbox{and}\ \sup_{q \in \bR^D_{\geq 0}: ||q||_1 = t} ||(q, \nu)|| = ||(t\ell, \nu)||\]
\end{lemma}

\begin{remark}
In Section \ref{section4} we show that $||q||_1 = ||\nu||_{TV}$ is a necessary condition for $||(q,\nu)||$ to be finite, so it makes sense that we only take the supremum over $q \in \bR^D_{\geq 0}$ with $||q||_1 = t$.
\end{remark}

\begin{proof}
This is an easy consequence of the symmetries of the grid and the concavity of $\widetilde{d}^{\epsilon}$ and direction-fixed grid entropy. We focus on the proof for $\widetilde{d}^{\epsilon}$; the argument for grid entropy goes the same way.

Fix $\epsilon > 0$. By positive-homogeneity and since the $t = 0$ is trivial, we may assume $t = 1$. Suppose there exists $q \in \bR^D_{\geq 0}$ s.t. $||q||_1 = 1$ and 
\begin{equation}\label{label10}
    \widetilde{d}^{\epsilon}((\vec{0}, 0), (q, \nu)) >  \widetilde{d}^{\epsilon}((\vec{0}, 0), (\ell, \nu))
\end{equation}

 Among such $q$ pick one that maximizes the number of coordinates which are equal $\frac{1}D$. Thus there are distinct $1 \leq i,j \leq D$ s.t. $q_i < \frac{1}D < q_j$, so we can write $\frac{1}D$ as a convex combination of $q_i, q_j$:
\begin{equation}\label{label11} \frac{1}D = wq_i + (1-w) q_j \ \mbox{for some} \ w \in (0,1)
\end{equation}
Let $\sigma_{ij}(q)$ be $q$ with $q_i, q_j$ swapped. By symmetry of the grid, 
\[\widetilde{d}^{\epsilon}((\vec{0}, 0), (q, \nu)) = \widetilde{d}^{\epsilon}((\vec{0}, 0), (\sigma_{ij}(q), \nu))\]
hence by concavity of $\widetilde{d}^{\epsilon}$,
\begin{align*}
\widetilde{d}^{\epsilon}((\vec{0}, 0), (q, \nu)) &= w \widetilde{d}^{\epsilon}((\vec{0}, 0), (q, \nu)) + (1-w) \widetilde{d}^{\epsilon}((\vec{0}, 0), (\sigma_{ij}(q), \nu)) \\
& \leq \widetilde{d}^{\epsilon}((\vec{0}, 0), (wq + (1-w)\sigma_{ij}(q), \nu))
\end{align*}
But $wq + (1-w)\sigma_{ij}(q)$ only changes the coordinates of $q$ in positions $i,j$, with $q_i$ becoming $\frac{1}D$ by \eqref{label11}. Thus we have found a $q$ satisfying \eqref{label10}
that has at least one more coordinate that is $\frac{1}D$ than our previous $q$, which we had assumed had the maximal number of such coordinates. Contradiction. Therefore $q = \ell$ as desired.
\end{proof}

We now use this useful fact along with the compactness of $\{q \in \bR^D_{\geq 0}: ||q||_1 = t\}$ to show that $||(t\ell, \nu)||$ is the desired direction-free grid entropy of length $t$.

\begin{theorem}\label{thmNoDir}
Fix $t \geq 0, \nu \in \mathcal{M}_t$. For any $\epsilon > 0$ we have
\begin{align*} \widetilde{d}^{\epsilon}((\vec{0},0), (t\ell, \nu)) &= \lim_{n \rightarrow \infty} \sup_{q \in \bR^D_{\geq 0}: ||q||_1=t} \frac1n \log \sum_{\pi \in \mathcal{P}(\vec{0}, \lfloor nq \rfloor)}  e^{-\frac{n}{\epsilon} \rho(\frac1n \mu_{\pi}, \nu)} \\
&=  \lim_{n \rightarrow \infty} \frac1n \log \sum_{\pi \in \mathcal{P}_{ \lfloor nt \rfloor}(\vec{0})}  e^{-\frac{n}{\epsilon} \rho(\frac1n \mu_{\pi}, \nu)}
\end{align*}
a.s.
\end{theorem}

\begin{proof}
The statement is trivial when $t=0$ so we may assume $t > 0$. We focus on the first equality. By Lemma \ref{maximalLemma} and the trivial fact that  $\sup \limsup \leq \limsup \sup$ in general, we immediately get
\[\widetilde{d}^{\epsilon}((\vec{0},0), (t\ell, \nu)) \leq \limsup_{n \rightarrow \infty} \sup_{q \in \bR^D_{\geq 0}: ||q||_1=t} \frac1n \log \sum_{\pi \in \mathcal{P}(\vec{0}, \lfloor nq \rfloor)}  e^{-\frac{n}{\epsilon} \rho(\frac1n \mu_{\pi}, \nu)} \ \mbox{a.s.}\]
Suppose equality does not hold on some event  of positive probability. Thus there exists $\delta > 0$ s.t.
\begin{equation} \label{EQ13}
\widetilde{d}^{\epsilon}((\vec{0},0), (t\ell, \nu)) + 7\delta  < \limsup_{n \rightarrow \infty} \sup_{q \in \bR^D_{\geq 0}: ||q||_1=t} \frac1n \log \sum_{\pi \in \mathcal{P}(\vec{0}, \lfloor nq \rfloor)}  e^{-\frac{n}{\epsilon} \rho(\frac1n \mu_{\pi}, \nu)}
\end{equation}
with positive probability. Let us intersect this event with the measure 1 event that \\ $\widetilde{d}^{\epsilon}((\vec{0}, 0), (p, \nu))$ exists and is equal to the limit in Theorem \ref{thm7} for every $p \in \bQ^D_{\geq 0}$. Denote this event by $\mathcal{E}$.

Unfortunately, the expression inside the supremum is only continuous when approaching $q$ from the SE, and not necessarily when approaching along $\{r \in \bR^D_{\geq 0}: ||r||_1=t\}$, so we need to work with some rather unpleasant approximations. 

Let $(n_i)$ be the (event-dependent) subsequence corresponding to the $\limsup$. For every  $i \in \bN$ pick some (event-dependent) $q^{n_i} \in \bR^D_{\geq 0}$ with $||q^{n_i}||_1 = t$ s.t. 
\begin{equation}\label{EQ14} \sup_{q \in \bR^D_{\geq 0}: ||q||_1=t} \frac1{n_i} \log \sum_{\pi \in \mathcal{P}(\vec{0}, \lfloor n_iq \rfloor)}  e^{-\frac{n_i}{\epsilon} \rho(\frac1{n_i} \mu_{\pi}, \nu)} \leq \delta + \frac1{n_i} \log \sum_{\pi \in \mathcal{P}(\vec{0}, \lfloor n_iq^{n_i} \rfloor)}  e^{-\frac{n_i}{\epsilon} \rho(\frac1{n_i} \mu_{\pi}, \nu)} 
\end{equation}

By compactness of $\{r \in \bR^D_{\geq 0}: ||r||_1=t\}$, there is a converging subsequence $q^{n_{i_j}} \rightarrow^{L^1} q'$ for some $q' \in \bR^D_{\geq 0}, ||q'||_1=t$. Pick some $q'' \in  \bQ^D_{\geq 0}$ s.t. $q''-q' \in \bR^D_{> 0}$ and
\begin{equation}\label{EQ15}
   \max \bigg(\frac1{\epsilon} ||q''-q'||_1, \bigg(\frac{||q''||_1}t-1\bigg) \widetilde{d}^{\epsilon}((\vec{0}, 0), (t \ell,\nu ) ) ,  \frac1{\epsilon} \bigg| \bigg|\bigg(1-\frac{t}{||q''||_1}\bigg) \nu \bigg| \bigg|_{TV} \bigg) < \delta
\end{equation}

The idea is that since $q^{n_{i_j}} \rightarrow^{L^1} q'$ and the coordinates of $q''-q'$ are \emph{strictly} positive, then there exists $J$ s.t. $q'' - q^{n_{i_j}} \in \bR^D_{\geq 0} \ \forall j \geq J$. For such $j$ we can apply Lemma \ref{lemma12} to get
\begin{equation}\label{EQ16} 
\begin{split}
& \frac1{n_{i_j}} \log \sum_{\pi \in \mathcal{P}(\vec{0}, \lfloor n_{i_j}q^{n_{i_j}} \rfloor)}  e^{-\frac{n_{i_j}}{\epsilon} \rho(\frac1{n_{i_j}} \mu_{\pi}, \nu)} \\
& \leq \frac1{\epsilon} \bigg(||q''-q^{n_{i_j}}||_1+ \frac{D}{n_{i_j}}\bigg)+ \frac1{n_{i_j}} \log \sum_{\pi \in \mathcal{P}(\vec{0}, \lfloor n_{i_j}q'' \rfloor)}  e^{-\frac{n_{i_j}}{\epsilon} \rho(\frac1{n_{i_j}} \mu_{\pi}, \nu)} 
\end{split}
\end{equation}
Recalling that we are in the event $\mathcal{E}$, pick some  $J' \geq J$ large enough so that for any $j \geq J',$
\begin{equation}\label{EQ17}
    \frac1{n_{i_j}} \log \sum_{\pi \in \mathcal{P}(\vec{0}, \lfloor n_{i_j}q'' \rfloor)}  e^{-\frac{n_{i_j}}{\epsilon} \rho(\frac1{n_{i_j}} \mu_{\pi}, \nu)}  \leq \delta +  \widetilde{d}^{\epsilon}((\vec{0}, 0), (q'', \nu))
\end{equation}
and $  \frac1{\epsilon} \bigg(||q'-q^{n_{i_j}}||_1+ \frac{D}{n_{i_j}}\bigg) < \delta \ \mbox{hence}$
\begin{equation}\label{EQ18}
    \frac1{\epsilon} \bigg(||q''-q^{n_{i_j}}||_1+ \frac{D}{n_{i_j}}\bigg) \leq \frac1{\epsilon} \bigg(||q'-q^{n_{i_j}}||_1+ ||q''-q'||_1+ \frac{D}{n_{i_j}}\bigg)<  2\delta 
\end{equation}
by \eqref{EQ15}. Putting everything together, we get 
\begin{align*}
&\widetilde{d}^{\epsilon}((\vec{0},0), (t\ell, \nu)) + 7\delta  \\
&\leq \widetilde{d}^{\epsilon}((\vec{0}, 0), (q'', \nu)) + 4 \delta \ \mbox{by \eqref{EQ13}-\eqref{EQ18}} \\
&= \frac{||q''||_1}t \widetilde{d}^{\epsilon}\bigg((\vec{0}, 0), \bigg(\frac{t}{||q''||_1} q'', \frac{t}{||q''||_1}\nu \bigg) \bigg) + 4 \delta \ \mbox{by positive-homogeneity} 
\end{align*}
By Lemma \ref{maximalLemma}, this is
\begin{align*}
&\leq \frac{||q''||_1}t \widetilde{d}^{\epsilon}\bigg((\vec{0}, 0), \bigg(t \ell, \frac{t}{||q''||_1}\nu \bigg) \bigg) + 4 \delta \\
&\leq \frac{||q''||_1}t \widetilde{d}^{\epsilon}((\vec{0}, 0), (t \ell,\nu ) ) + \frac1{\epsilon} \bigg| \bigg|\bigg(1-\frac{t}{||q''||_1}\bigg) \nu \bigg| \bigg|_{TV} + 4 \delta \ \mbox{by Lemma \ref{lemma12}}\\
&\leq \widetilde{d}^{\epsilon}((\vec{0},0), (t\ell, \nu)) + 6\delta  \ \mbox{by \eqref{EQ15}} 
\end{align*}
which is a contradiction. Note that even though our sequence $q^{n_{i_j}}$, the limit $q''$ and the $J'$ chosen are event-dependent, the upper bound above is not by virtue of Lemma \ref{maximalLemma}. That is what makes this argument work. 

We have thus shown
\begin{align*} \widetilde{d}^{\epsilon}((\vec{0},0), (t\ell, \nu)) = \lim_{n \rightarrow \infty} \sup_{q \in \bR^D_{\geq 0}: ||q||_1=t} \frac1n \log \sum_{\pi \in \mathcal{P}(\vec{0}, \lfloor nq \rfloor)}  e^{-\frac{n}{\epsilon} \rho(\frac1n \mu_{\pi}, \nu)} \ \mbox{a.s.}
\end{align*}

We now  proceed with the second equality. Observe that for any $n$, any length $\lfloor nt \rfloor$ NE path $\pi$ from  $\vec{0}$ ends at $\lfloor nq \rfloor$ for some $q \in \bR^D_{\geq 0}$ with $||q||_1 = t$. Furthermore, the number of possible different endpoints of such paths is precisely $\binom{\lfloor nt \rfloor + D-1 }{ D-1} = O(n^D)$, which is the number of $D$-tuples of non-negative integers summing to $\lfloor nt \rfloor$. Therefore
\begin{align*} 
& \log \sum_{ \pi \in \mathcal{P}_{\lfloor nt \rfloor}(\vec{0})}  e^{-\frac{n}{\epsilon} \rho(\frac1n \mu_{\pi}, \nu )} \\
&\in \bigg[\sup_{||q||_1 = t} \log \sum_{\pi \in \mathcal{P}(\vec{0}, \lfloor nq \rfloor)}  e^{-\frac{n}{\epsilon} \rho(\frac1n\mu_{\pi}, \nu )},    \log \bigg( O(n^D) \sup_{||q||_1 = t} \sum_{\pi \in \mathcal{P}(\vec{0}, \lfloor nq \rfloor)}  e^{-\frac{n}{\epsilon} \rho(\frac1n\mu_{\pi}, \nu )} \bigg ) \bigg] \\
&= \bigg[\sup_{||q||_1 = t} d^{\epsilon}((\vec{0}, n\xi), (nq, n\nu)), O(D\log n) + \sup_{||q||_1 = t} d^{n\epsilon}((\vec{0}, n\xi), (nq, n\nu))\bigg]
\end{align*}
When we divide by $n$ and take the limit, the $O(D \log n)$ term goes to 0 and we get that
\[\lim_{n \rightarrow \infty}\log \sum_{\pi \in \mathcal{P}_{\lfloor nt \rfloor}(\vec{0})}  e^{-\frac{n}{\epsilon} \rho(\frac1n \mu_{\pi}, \nu )}   = \lim_{n \rightarrow \infty} \sup_{q \in \bR^D_{\geq 0}, ||q||_1 = t} \frac{d^{n\epsilon}((\vec{0}, n\xi), (nq, n\nu))}n  \ \mbox{a.s.}\]
as desired.
\end{proof}

Of course, we already established that $\widetilde{d}^{\epsilon}((\vec{0}, 0), (t\ell, \nu))$ is a  directed metric with negative sign.
This means that if we now take the infimum over $\epsilon > 0$ we get that the length $t$ direction-free grid entropy is precisely $||(t\ell, \nu)||$. Since this is $-\infty$ unless $t= ||\nu||_{TV}$ (see Theorem \ref{thm19} in a later section), then we can simply let $t = ||\nu||_{TV}$ to begin with.

\begin{definition}
For $\nu \in \mathcal{M}_+$ let $t := ||\nu||_{TV}$ and define the direction-free grid entropy to be
\[ ||\nu|| := \inf_{\epsilon} \widetilde{d}^{\epsilon}((\vec{0}, 0), (t\ell, \nu)) = ||(t\ell, \nu)|| = \sup_{q \in \bR^D} ||(q, \nu)||\]
\end{definition}

\begin{remark}
The final equality holds by Lemma \ref{maximalLemma} combined with the fact that $||(q,\nu)|| = -\infty$ unless $||q||_1 = ||\nu||_{TV}$.
\end{remark}

\subsection{Equivalence of Approaches}
The next order of business is to  show that these definitions of direction-fixed/direction-free grid entropy are equivalent to our original, more intuitive, definitions \eqref{firstDef}. The novel results in this subsection show the benefit of the fresh approach to grid entropy presented in this paper.

The key idea is that the sum
\[\sum_{\pi \in \mathcal{P}( \vec{0}, \lfloor nq \rfloor)} e^{-\frac{n}{\epsilon} \rho(\frac1n\mu_{\pi}, \nu ) }\]
appearing in the definition of $d^{\epsilon}$ is approximately the number of paths $\vec{0} \rightarrow \lfloor nq \rfloor$ with very small $\rho(\frac1n\mu_{\pi}, \nu )$ and the error should disappear in the limit.

Let us recall how the $\min \limits_{\pi}^k$ are defined. Fix $q \in \bR^D, \nu \in \mathcal{M}_+$. For any $n, k \in \bN$ with $k \leq (\#\pi: \vec{0}  \rightarrow \lfloor nq \rfloor)$ define 
\[\min \limits_{\pi \in \mathcal{P}( \vec{0}, \lfloor nq \rfloor)}^k \rho \bigg(\frac1n \mu_{\pi}, \nu \bigg)\]
to be the order statistics of  $\rho(\frac1n \mu_{\pi}, \nu)$ as we range over all $\pi: \vec{0}  \rightarrow \lfloor nq \rfloor$, and let
\[\min \limits_{\pi: \vec{0} \rightarrow \lfloor nq \rfloor}^k \rho \bigg(\frac1n \mu_{\pi}, \nu \bigg) := +\infty \ \mbox{for} \ k > \# \pi: \vec{0} \rightarrow \lfloor nq \rfloor\]
In this way $\min \limits_{\pi \in \mathcal{P}( \vec{0}, \lfloor nq \rfloor)}^1 \rho(\frac1n \mu_{\pi}, \nu), \min \limits_{\pi \in \mathcal{P}( \vec{0}, \lfloor nq \rfloor)}^{\#\pi} \rho(\frac1n \mu_{\pi}, \nu)$ are just the minimum/maximum values of $\rho(\frac1n \mu_{\pi}, \nu)$ respectively. 

Similarly, in the direction-free case, for $\nu \in \mathcal{M}_+$ we let $t := ||\nu||_{TV}$ and for any $n, k \in \bN$ with $k \leq (\# \ \mbox{length $\lfloor nt \rfloor$ paths $\pi$ from $\vec{0}$})$ define
\[\min \limits_{\pi \in \mathcal{P}_{\lfloor nt \rfloor}( \vec{0})}^k \rho \bigg(\frac1n \mu_{\pi}, \nu \bigg)\]
to be the $k$-th smallest value of $\rho(\frac1n \mu_{\pi}, \nu)$ when we range over \textit{all} length $\lfloor nt \rfloor$ paths from $\vec{0}$.

We usually omit the $\rho(\frac1n \mu_{\pi}, \nu)$ in $\min \limits_{\pi}^j \rho(\frac1n \mu_{\pi}, \nu)$ for the sake of space.

We are now ready to prove that our definitions of grid entropy are equivalent, and moreover that the pertinent limits either converge to 0 a.s. or the corresponding $\liminf$'s are $>0$ a.s.. In terms of notation, we denote by $||(q, \nu)||, ||\nu||$ the direction-fixed/direction-free grid entropy as defined in Theorem \ref{thm16}/Theorem \ref{thmNoDir}.

\begin{theorem}\label{thmEquiv}
Fix $q \in \bR^D, \nu \in \mathcal{M}_+$ and let $t := ||\nu||_{TV}$.
\begin{enumerate}[label=(\roman*)]
\item The grid entropy being finite is  characterized by the  existence of paths with empirical measure arbitrarily close to the target. More specifically,
\[ \lim_{n \rightarrow \infty} \min_{\pi \in \mathcal{P}( \vec{0}, \lfloor nq \rfloor)}^1 \rho \bigg(\frac1{n} \mu_{\pi}, \nu \bigg),  \lim_{n \rightarrow \infty} \min_{\pi \in \mathcal{P}_{\lfloor nt \rfloor}( \vec{0})}^1 \rho \bigg(\frac1{n} \mu_{\pi}, \nu \bigg)   \ \mbox{exist a.s. and} \]
\begin{equation*}
||(q,\nu)|| \neq -\infty \Leftrightarrow \lim_{n \rightarrow \infty} \min_{\pi \in \mathcal{P}( \vec{0}, \lfloor nq \rfloor)}^1 \rho \bigg(\frac1{n} \mu_{\pi}, \nu \bigg) = 0 \ \mbox{a.s.}
\end{equation*}  
\begin{align*}
    ||\nu|| \neq -\infty \Leftrightarrow \lim_{n \rightarrow \infty} \min_{\pi \in \mathcal{P}_{\lfloor nt \rfloor}( \vec{0})}^1 \rho \bigg(\frac1{n} \mu_{\pi}, \nu \bigg) = 0  \ \mbox{a.s.}
\end{align*}
Furthermore, if $||(q,\nu)|| \neq -\infty$ then 
\[||(q,\nu)|| \in [0, H(q)] \ \mbox{where} \ H(q) = \sum_{i=1}^D -q_i \log \frac{q_i}{||q||_1}\]
and if $||\nu|| \neq -\infty$ then $||\nu|| = ||(t\ell, \nu)|| \in [0,  t\log D]$. In particular, 
$$|| \ ||q||_1 \Lambda \ || \geq ||(q, ||q||_1\Lambda)|| \geq 0$$
i.e. $\Lambda$ is always among the distributions observed.
\item We have
\[ \lim_{n \rightarrow \infty} \min_{\pi \in \mathcal{P}( \vec{0}, \lfloor nq \rfloor)}^{\lfloor e^{n \alpha } \rfloor} \rho \bigg(\frac1{n} \mu_{\pi}, \nu \bigg) = 0 \ \mbox{a.s.} \ \forall 0 \leq \alpha < ||(q,\nu)||\]
\[ \liminf_{n \rightarrow \infty} \min_{\pi \in \mathcal{P}( \vec{0}, \lfloor nq \rfloor)}^{\lfloor e^{n \alpha } \rfloor} \rho \bigg(\frac1{n} \mu_{\pi}, \nu \bigg) > 0 \ \mbox{a.s.} \ \forall ||(q,\nu)|| < \alpha \leq H(q)\]
In other words,
\begin{equation}\label{gridEntropy}
\begin{split}
    ||(q, \nu)|| &=\sup \bigg\{\alpha \geq 0: \lim_{n \rightarrow \infty} \min_{\pi \in \mathcal{P}( \vec{0}, \lfloor nq \rfloor)}^{\lfloor e^{n \alpha } \rfloor} \rho\bigg(\frac1{n} \mu_{\pi}, \nu \bigg) = 0  \ \mbox{a.s.} \bigg\} \\
    &= \sup \bigg\{\alpha \geq 0: \liminf_{n \rightarrow \infty} \min_{\pi \in \mathcal{P}( \vec{0}, \lfloor nq \rfloor)}^{\lfloor e^{n\alpha} \rfloor} \rho\bigg(\frac1n \mu_{\pi}, \nu \bigg) = 0  \ \mbox{a.s.} \bigg\} 
\end{split}
\end{equation}
where this is $-\infty$ if the set of $\alpha$'s is empty. Similarly,  
\[ \lim_{n \rightarrow \infty} \min_{\pi \in \mathcal{P}_{\lfloor nt \rfloor}( \vec{0})}^{\lfloor e^{n\alpha} \rfloor} \rho\bigg(\frac1n \mu_{\pi}, \nu \bigg) = 0 \ \mbox{a.s.} \ \forall 0 \leq \alpha < ||\nu||\]
\[ \liminf_{n \rightarrow \infty} \min_{\pi \in \mathcal{P}_{\lfloor nt \rfloor}( \vec{0})}^{\lfloor e^{n\alpha} \rfloor} \rho\bigg(\frac1n \mu_{\pi}, \nu \bigg) > 0 \ \mbox{a.s.} \ \forall ||\nu|| < \alpha \leq t\log D\]
and the direction-free grid entropy is given by
\begin{align*}
||\nu|| &=\sup \bigg\{\alpha \geq 0: \lim_{n \rightarrow \infty} \min_{\pi \in \mathcal{P}_{\lfloor nt \rfloor}( \vec{0})}^{\lfloor e^{n\alpha } \rfloor} \rho\bigg(\frac1{n} \mu_{\pi}, \nu \bigg) = 0  \ \mbox{a.s.} \bigg\} \\
&=\sup \bigg\{\alpha \geq 0: \liminf_{n \rightarrow \infty} \min_{\pi \in \mathcal{P}_{\lfloor nt \rfloor}( \vec{0})}^{\lfloor e^{n\alpha} \rfloor} \rho\bigg(\frac1n \mu_{\pi}, \nu \bigg) = 0  \ \mbox{a.s.} \bigg\} 
\end{align*}
\end{enumerate}
\end{theorem}

\begin{remark}
Since the $\min \limits_{\pi}^j$ increase in $j$, then the sets we are taking the supremum over in the definition in (ii) are either $\emptyset$ or of the form $[0, C)$ or $[0, C]$ where $C$ is the corresponding grid entropy.
\end{remark}

\begin{remark}
The  limits 
\[\lim_{n \rightarrow \infty} \min_{\pi \in \mathcal{P}( \vec{0}, \lfloor nq \rfloor)}^{\lfloor e^{n \alpha } \rfloor} \rho\bigg(\frac1{n_i} \mu_{\pi}, \nu \bigg)\]
need not  exist for all $\alpha$. What we are saying is that they do exist and are equal to 0 a.s. for $\alpha < ||(q,\nu)||$ and that the corresponding $\liminf$'s are $>0$ a.s. for $\alpha > ||(q,\nu)||$. This along with the existence of the limits in (i) are completely non-trivial facts which will be a direct consequence of the a.s. convergence of $\frac{d^{\epsilon}((0,0),(nq, n\nu))}n$.
\end{remark}

\begin{proof}
We focus on proving the direction-fixed grid entropy statements. The direction-free arguments are analogous.

(i) The strategy is to provide  upper and lower bounds on $d^{\epsilon}$ in terms of
\[\mbox{exp} \bigg[-\frac{n}{\epsilon}\min \limits_{\pi \in \mathcal{P}( \vec{0}, \lfloor nq \rfloor)}^1 \rho \bigg(\frac1{n} \mu_{\pi}, \nu \bigg) \bigg],\]
the leading term in the sum over $\pi$. Since $\frac{d^{\epsilon}((\vec{0},0), (nq, n\nu))}n$ converges a.s. then our analysis gives that $\widetilde{d}^{\epsilon}$ must be in the intersection of two intervals, one in terms of \\ $\limsup \min \limits_{\pi \in \mathcal{P}( \vec{0}, \lfloor nq \rfloor)}^1 \rho (\frac1{n} \mu_{\pi}, \nu)$ and the other in terms of the $\liminf$. The only way these intervals can have non-empty intersection for small enough $\epsilon > 0$ is if the $\limsup$'s equal the $\liminf$'s.

Fix $\epsilon > 0$. We consider the $\min \limits_{\pi \in \mathcal{P}( \vec{0}, \lfloor nq \rfloor)}^j$ over $\pi: \vec{0} \rightarrow \lfloor n q\rfloor$. Since $\min \limits_{\pi \in \mathcal{P}( \vec{0}, \lfloor nq \rfloor)}^1 \leq \min \limits_{\pi \in \mathcal{P}( \vec{0}, \lfloor nq \rfloor)}^j$ for all valid $j$, then
\begin{align*}
    &\frac{d^{\epsilon}((\vec{0},0), (nq, n\nu))}{n} \\
    &= \frac1{n} \log \sum_{\pi \in \mathcal{P}(\vec{0}, \lfloor nq \rfloor)} e^{-\frac{n}{\epsilon} \rho(\frac1{n} \mu_{\pi}, \nu)}\\
    & \in \bigg[\frac1{n} \log \bigg(e^{-\frac{n}{\epsilon} \min \limits_{\pi \in \mathcal{P}(\vec{0}, \lfloor nq \rfloor)}^1}\bigg),  \frac1{n} \log \bigg((\# \pi: \vec{0} \rightarrow \lfloor nq \rfloor) \  e^{-\frac{n}{\epsilon} \min \limits_{\pi \in \mathcal{P}(\vec{0}, \lfloor nq \rfloor)}^1} \bigg)\bigg]\\
    & = \bigg[-\frac1{\epsilon}  \min_{\pi \in \mathcal{P}(\vec{0}, \lfloor nq \rfloor)}^1 \rho \bigg(\frac1{n} \mu_{\pi}, \nu \bigg), \frac1{n} \log(\# \pi: \vec{0} \rightarrow \lfloor nq \rfloor) -\frac1{\epsilon}  \min \limits_{\pi \in \mathcal{P}(\vec{0}, \lfloor nq \rfloor)}^1 \rho \bigg(\frac1{n} \mu_{\pi}, \nu \bigg) \bigg]
\end{align*}
where each $\rho(\frac1{n} \mu_{\pi}, \nu) \geq 0$. Taking the a.s. $\limsup$ or the a.s. $\liminf$ and using Lemma \ref{lemma1} we get that a.s.,
\begin{equation} \label{equation11}
\begin{split}
\widetilde{d}^{\epsilon}((\vec{0}, 0), (q, \nu)) \in  & \bigg[-\frac1{\epsilon} \liminf_{n \rightarrow \infty} \min \limits_{\pi \in \mathcal{P}(\vec{0}, \lfloor nq \rfloor)}^1, H(q) -\frac1{\epsilon} \liminf_{n \rightarrow \infty} \min \limits_{\pi \in \mathcal{P}(\vec{0}, \lfloor nq \rfloor)}^1 \bigg]\\
& \bigcap  \bigg[-\frac1{\epsilon} \limsup_{n \rightarrow \infty} \min \limits_{\pi \in \mathcal{P}(\vec{0}, \lfloor nq \rfloor)}^1, H(q) -\frac1{\epsilon} \limsup_{n \rightarrow \infty} \min \limits_{\pi \in \mathcal{P}(\vec{0}, \lfloor nq \rfloor)}^1 \bigg]
\end{split}
\end{equation}
Since this holds for arbitrarily small $\epsilon > 0$ then we must have
\[\liminf_{n \rightarrow \infty} \min \limits_{\pi \in \mathcal{P}(\vec{0}, \lfloor nq \rfloor)}^1 \rho \bigg(\frac1{n} \mu_{\pi}, \nu \bigg)= \limsup_{n \rightarrow \infty} \min \limits_{\pi \in \mathcal{P}(\vec{0}, \lfloor nq \rfloor)}^1 \rho \bigg(\frac1{n} \mu_{\pi}, \nu \bigg) \ \mbox{a.s.}\]
Furthermore, taking $\epsilon \rightarrow 0^+$ in \eqref{equation11},  we see that 
\begin{align*}
||(q,\nu)|| =\inf_{\epsilon > 0} \widetilde{d}^{\epsilon}((\vec{0}, 0), (q, \nu)) \neq -\infty  \Leftrightarrow \lim_{n \rightarrow \infty} \min_{\pi: \vec{0} \rightarrow \lfloor nq \rfloor}^1 \rho \bigg(\frac1{n} \mu_{\pi}, \nu \bigg) = 0 \ \mbox{a.s.} \end{align*} 
Moreover, if indeed $||(q, \nu)|| \neq -\infty$ then \eqref{equation11} gives $ ||(q, \nu)|| \in [0, H(q)]$.

Finally, if we fix an infinite NE path $\pi$ passing through all $\lfloor nq \rfloor$ then by the Glivenko-Cantelli Theorem (Theorem \ref{thm1}), 
\[ \rho\bigg(\frac1n \mu_{\pi|_{\vec{0} \rightarrow \lfloor nq \rfloor}}, ||q||_1 \Lambda \bigg) = ||q||_1 \rho \bigg(\frac1{|\pi|}\mu_{\pi|_{\vec{0} \rightarrow \lfloor nq \rfloor}}, \Lambda \bigg) \ \rightarrow 0 \]
where $\mu_{\pi|_{\vec{0} \rightarrow \lfloor nq \rfloor}}$ denotes the restriction of the path to $\vec{0} \rightarrow \lfloor nq \rfloor$, so from above it follows that $||(q, ||q||_1\Lambda)|| \geq 0$.

(ii) If $||(q,\nu)|| = -\infty$ then by (i),
\[\lim_{n \rightarrow \infty} \min_{\pi \in \mathcal{P}( \vec{0}, \lfloor nq \rfloor)}^{1} \rho\bigg(\frac1{n_i} \mu_{\pi}, \nu \bigg) \ \mbox{exists a.s. and is $>0$}\]
so the statement holds trivially because the $ \min \limits_{\pi \in \mathcal{P}( \vec{0}, \lfloor nq \rfloor)}^{j}$ are nondecreasing in $j$. Thus it suffices to assume  $||(q,\nu)|| \geq 0$ and hence
\begin{equation}\label{eqn12}
\min \limits_{\pi \in \mathcal{P}(\vec{0}, \lfloor nq \rfloor)}^1 \rho\bigg(\frac1{n} \mu_{\pi}, \nu \bigg) \rightarrow 0 \ \mbox{a.s.}
\end{equation}

Fix $\epsilon > 0$ and consider any $n \in \bN$. By \eqref{eqn12}, the leading term of 
\[\sum_{\pi \in \mathcal{P}(\vec{0}, \lfloor nq \rfloor)} e^{-\frac{n}{\epsilon} \rho(\frac1{n} \mu_{\pi}, \nu)}\]
is $\approx 1$. So we must look at the secondary terms by writing
\begin{equation} \label{eqn13}
\frac{d^{\epsilon}((\vec{0},0), (nq, n\nu))}{n} = -\frac1{\epsilon} \min_{\pi \in \mathcal{P}(\vec{0}, \lfloor nq \rfloor)}^1 +
    \frac1{n} \log \bigg[ 1 + \sum_{j = 2}^{\# \pi} e^{-\frac{n}{\epsilon} [\min \limits_{\pi \in \mathcal{P}(\vec{0}, \lfloor nq \rfloor)}^j - \min \limits_{\pi \in \mathcal{P}(\vec{0}, \lfloor nq \rfloor)}^1 ] } \bigg]
\end{equation}
where the summands are nonincreasing in $j$.

The argument goes similar to the one in (i), in that we compute some upper and lower bounds for this expression. 

Fix an arbitrary $0 \leq C \leq H(q)$. We lower bound \eqref{eqn13} by truncating the sum at $j = \lfloor e^{Cn} \rfloor$ and lower bounding all the summands by the $j =  \lfloor e^{Cn} \rfloor$ summand: 
\begin{equation}\label{equation14}
\begin{split}
&\frac{d^{\epsilon}((\vec{0},0), (nq, n\nu))}{n} \geq\\
&  -\frac1{\epsilon} \min_{\pi \in \mathcal{P}(\vec{0}, \lfloor nq \rfloor)}^1  +
    \frac1{n} \log \bigg[ 1 + (\lfloor e^{Cn} \rfloor-1) \mbox{exp} \bigg[-\frac{n}{\epsilon} \bigg(\min \limits_{\pi \in \mathcal{P}(\vec{0}, \lfloor nq \rfloor)}^{\lfloor e^{Cn} \rfloor}  - \min \limits_{\pi \in \mathcal{P}(\vec{0}, \lfloor nq \rfloor)}^1 \bigg) \bigg] \bigg]
    \end{split}
\end{equation}

We upper bound \eqref{eqn13} by upper bounding summands for $j = 2$ to $\lfloor e^{Cn} \rfloor$ to 1 and upper bounding the rest of the summands to the $j = \lfloor e^{Cn} \rfloor$ summand:
\begin{equation}\label{EQ24}
\begin{split}
    & \frac{d^{\epsilon}((\vec{0},0), (nq, n\nu))}{n} \leq  \\ & -\frac1{\epsilon} \min_{\pi \in \mathcal{P}(\vec{0}, \lfloor nq \rfloor)}^1 +
    \frac1{n} \log \bigg[ 1 + \lfloor e^{Cn} \rfloor \cdot 1  + (\#\pi) \mbox{exp}\bigg[-\frac{n}{\epsilon} \bigg(\min \limits_{\pi \in \mathcal{P}(\vec{0}, \lfloor nq \rfloor)}^{\lfloor e^{Cn} \rfloor}  - \min \limits_{\pi \in \mathcal{P}(\vec{0}, \lfloor nq \rfloor)}^1 \bigg) \bigg] \bigg]
    \end{split}
\end{equation}
Consider the event-dependent sequence
\[a_{n}(C) := \min \limits_{\pi \in \mathcal{P}(\vec{0}, \lfloor nq \rfloor)}^{ \lfloor e^{Cn} \rfloor} \rho \bigg(\frac1{n} \mu_{\pi}, \nu \bigg) - \min \limits_{\pi \in \mathcal{P}(\vec{0}, \lfloor nq \rfloor)}^1 \rho\bigg(\frac1{n} \mu_{\pi}, \nu \bigg) \geq 0  \]
Taking the $\liminf/\limsup$ in \eqref{equation14},\eqref{EQ24} and using \eqref{eqn12} and Lemma \ref{lemma1} we get 
\begin{align*} 
\widetilde{d}^{\epsilon}((0,0),(q,\nu)) &\geq \log \limsup_{n \rightarrow \infty} \bigg( 1 + (\lfloor e^{Cn} \rfloor-1) e^{-\frac{a_n(C)}{\epsilon} n} \bigg)^{\frac1n} \\
&= C- \frac1{\epsilon} \liminf \limits_{n \rightarrow \infty} a_n(C),
\end{align*}
\begin{align*}
 \widetilde{d}^{\epsilon}((0,0),(q,\nu)) &\leq \log \liminf_{n \rightarrow \infty} \bigg( 1 + e^{Cn}+ e^{H(q)n-\frac{a_n(C)}{\epsilon} n} \bigg)^{\frac1n} \\
 &= \max(C, H(q) - \frac1{\epsilon} \limsup \limits_{n \rightarrow \infty} a_n(C))
 \end{align*}
a.s. for this arbitrary $0 \leq C \leq H(q)$.

In particular, for any $0 \leq C < ||(q,\nu)||$ we get
\[0 \leq ||(q,\nu)|| \leq \widetilde{d}^{\epsilon}((0,0),(q,\nu)) \leq  \max(C, H(q) - \frac1{\epsilon} \limsup \limits_{n \rightarrow \infty} a_n(C)) \ \mbox{a.s.} \ \forall \epsilon > 0\]
hence
\[ \lim_{n\rightarrow \infty} \min \limits_{\pi \in \mathcal{P}(\vec{0}, \lfloor nq \rfloor)}^{ \lfloor e^{Cn} \rfloor} \rho \bigg(\frac1{n} \mu_{\pi}, \nu \bigg) =  \lim_{n\rightarrow \infty} a_n(C) = 0 \ \mbox{a.s.} \]

On the other hand, for any $||(q,\nu)||< C \leq H(q)$,
\[||(q,\nu)|| = \inf_{\epsilon > 0}\widetilde{d}^{\epsilon}((0,0),(q,\nu)) \geq C- \inf_{\epsilon > 0} \frac1{\epsilon} \liminf \limits_{n \rightarrow \infty} a_n(C) \ \mbox{a.s.} \]
hence 
\[ \liminf_{n\rightarrow \infty} \min \limits_{\pi \in \mathcal{P}(\vec{0}, \lfloor nq \rfloor)}^{ \lfloor e^{Cn} \rfloor} \rho \bigg(\frac1{n} \mu_{\pi}, \nu \bigg) =  \lim_{n\rightarrow \infty} a_n(C) > 0 \ \mbox{a.s.} \]
Equation \eqref{gridEntropy} follows.
\end{proof}

Thus our approaches to grid entropy are equivalent.

Note that so far we have not made use of the coupling $\tau_e = \tau(U_e)$ or of the compactness of the space of measures on $[0,1]$. We could have developed grid entropy in the original environment, with unnormalized empirical measures $\sigma_{\pi}$ and edge weight distribution $\theta$ and target measures $\nu$ on $\bR$ in the same way. Furthermore, if $\theta$ has a continuous cdf then the value  of the grid entropy would be the same either way, because of the duality between the environments, established in Lemma \ref{lemma5Bates}; that is, for any $q \in \bR^D$ and $\nu \in \mathcal{M}_+$ we would have 
\[||(q, \nu)|| = ||(q, \tau_{*}(\nu))||, ||\nu|| = ||\tau_{*}(\nu)||\]
where the grid entropies are developed on the environments $([0,1], U_e \sim \Lambda), (\bR, \tau_e \sim \theta)$ respectively. Even in the general case where $F_{\theta}$ may not be continuous, \cite[Lemma 6.15]{bates} implies that
\[||(q, \nu)|| \leq ||(q, \tau_{*}(\nu))||, ||\nu|| \leq ||\tau_{*}(\nu)||\]
Again though, this is just a nice fact; for practical purposes we may simply define grid entropy on the space of measures on $\bR$ and our arguments thus far hold.

We now reap the benefits of this coupling, by describing the sets $\mathcal{R}^q, \mathcal{R}^{q, \alpha}, \mathcal{R}^t, \mathcal{R}^{t,\alpha}$ defined in Theorem \ref{thm5Bates} and Corollary \ref{corollary6Bates} in terms of direction-fixed/direction-free grid entropy.

\begin{corollary}\label{corollary18}
Fix $q \in \bR^D$ and $t \geq 0$. 
\begin{enumerate}[label=(\roman*)]
\item The deterministic set $\mathcal{R}^q$ of limits of empirical measures in direction $q$ whose existence is established by Theorem \ref{thm5Bates} (i) is precisely $\{\nu \in \mathcal{M}_+: ||(q, \nu)|| \geq 0\}$. Thus
\begin{align*}
||(q,\nu)|| = -\infty  &\Leftrightarrow \lim_{n \rightarrow \infty} \min_{\pi \in \mathcal{P}( \vec{0}, \lfloor nq \rfloor)}^1 \rho \bigg(\frac1{n} \mu_{\pi}, \nu \bigg) > 0 \ \mbox{a.s.} \\
\mbox{and} \ ||(q,\nu)|| \geq 0  &\Leftrightarrow P \bigg( \exists \pi: \vec{0} \rightarrow \lfloor nq \rfloor \ \mbox{with} \ \rho \bigg(\frac1{n} \mu_{\pi}, \nu \bigg) < \epsilon \ \mbox{i.o.}\bigg) = 1 \ \forall \epsilon > 0
\end{align*} 
Likewise, the deterministic set $\mathcal{R}^t$ of limits of length $t$ empirical measures from direction\hyp{}free paths, whose existence is established by Theorem \ref{thm5Bates}(ii) is precisely 
\[\{\nu \in \mathcal{M}_t: ||\nu|| \geq 0\} = \{\nu \in \mathcal{M}_t: ||(t\ell, \nu)|| \geq 0\} = \mathcal{R}^{t\ell}\]
Thus
\begin{align*}
||\nu|| = -\infty  &\Leftrightarrow \lim_{n \rightarrow \infty} \min_{\pi \in \mathcal{P}_{\lfloor nt \rfloor}( \vec{0})}^1 \rho \bigg(\frac1{n} \mu_{\pi}, \nu \bigg) > 0 \ \mbox{a.s.} \\
\mbox{and} \ ||\nu|| \geq 0  &\Leftrightarrow P \bigg( \exists \ \pi \ \mbox{s.t.} \ |\pi| = \lfloor nt \rfloor \ \mbox{with} \ \rho \bigg(\frac1{n} \mu_{\pi}, \nu \bigg) < \epsilon \ \mbox{i.o.}\bigg) = 1 \ \forall \epsilon > 0
\end{align*} 

\item Fix $0 < C \leq H(q)$. The set of measures with grid entropy in direction $q$ at least/most $C$ can be characterized in terms of the deterministic sets $\mathcal{R}^{q,\alpha}$  defined in Corollary \ref{corollary6Bates} (i):
\[\mathcal{R}^{q,C} = \bigcap_{0 \leq \alpha \leq C} \mathcal{R}^{q, \alpha}  \subseteq \{\nu \in \mathcal{M}_+: ||(q, \nu)|| \geq C\} \subseteq \bigcap_{0 \leq \alpha < C} \mathcal{R}^{q, \alpha}\] 
Thus
\begin{align*}
    &  ||(q,\nu)|| \geq C > 0      \Leftrightarrow  \\
     & \qquad P \bigg( \exists \lfloor e^{n\alpha} \rfloor \ \mbox{paths} \ \pi: \vec{0} \rightarrow \lfloor nq \rfloor \ \mbox{with} \ \rho \bigg(\frac1{n} \mu_{\pi}, \nu \bigg) < \epsilon \ \mbox{i.o.} \bigg) = 1
     \end{align*}
 $ \forall \epsilon > 0 \ \forall \alpha \in (0, C)$, and
\[||(q, \nu)|| < C  \Leftrightarrow  \liminf_{n \rightarrow \infty} \min_{\pi \in \mathcal{P}( \vec{0}, \lfloor nq \rfloor)}^{\lfloor e^{n C} \rfloor} \rho \bigg(\frac1{n} \mu_{\pi}, \nu \bigg) > 0 \ \mbox{a.s.} \]
Now fix $0 < C \leq H(t\ell) = t\log D$.  The set of measures with length $t$ direction-free grid entropy at least/most $C$ can be characterized in terms of the deterministic sets $\mathcal{R}^{t,\alpha}$  defined in Corollary \ref{corollary6Bates} (ii):
\[\mathcal{R}^{t\ell, C} = \mathcal{R}^{t,C} = \bigcap_{0 \leq \alpha \leq C} \mathcal{R}^{t, \alpha}  \subseteq \{\nu \in \mathcal{M}_t: || \nu|| \geq C\} \subseteq \bigcap_{0 \leq \alpha < C} \mathcal{R}^{t, \alpha}\] 
Thus
\begin{align*}
  &   ||\nu|| \geq C > 0   \Leftrightarrow  \\ 
  & \qquad P \bigg( \exists \lfloor e^{n\alpha} \rfloor \ \pi \in \mathcal{P}_{\lfloor nt \rfloor}( \vec{0}) \ \mbox{with} \ \rho \bigg(\frac1{n} \mu_{\pi}, \nu \bigg) < \epsilon \ \mbox{i.o.} \bigg) = 1 \ \forall \epsilon > 0 
     \end{align*}
   $\forall \alpha \in (0, C)$, and
\[ ||\nu|| < C  \Leftrightarrow  \liminf_{n \rightarrow \infty} \min_{\pi \in \mathcal{P}_{\lfloor nt \rfloor}( \vec{0})}^{\lfloor e^{n C} \rfloor} \rho \bigg(\frac1{n} \mu_{\pi}, \nu \bigg) > 0 \ \mbox{a.s.} \]
\end{enumerate}
\end{corollary}

\begin{proof}
This  follows immediately from Theorem \ref{thmEquiv}.
\end{proof}

Now that we have a firm grasp of what direction-fixed and direction-free grid entropies actually measure, we can move on examining  what information they can give us.

\section{Properties of Grid Entropy}\label{section4}
Since direction-free entropy $||\nu||$ is just $||(t\ell, \nu)||$ for $t:= ||\nu||_{TV}$ then it suffices to study the properties of direction-fixed grid entropy. We are particularly interested in what the grid entropy can tell us about $q$ and $\nu$ when $||(q, \nu)||$ is finite. 

We have already established that it is necessary for $q$ to be in $\bR^D_{\geq 0}$ if we want $||(q, \nu)||$  not to be $-\infty$, and that if $q =\vec{0}$ then $||(q, \nu)|| > -\infty$ if and only if $\nu = 0$. The question now is what sorts of measures $\nu$ are observed by the empirical measures along the direction $q \in \bR^D_{\geq 0} \setminus \{\vec{0}\}$.

By positive-homogeneity of the norm, it suffices to consider $q \in \bR^D_{\geq 0}$ with $||q||_1 = 1$ and study which $\nu \in \mathcal{M}_+$ give rise to finite $||(q, \nu)||$. In this section, it will be more convenient to work with the description of grid entropy given by Corollary \ref{corollary18} (ii).

\begin{theorem}\label{thm19}
Recall that we couple the edge weights with i.i.d. edge labels $U_e \sim \Lambda$ (the Lebesgue measure on $[0,1]$) and that 
\[H(q) := \lim \limits_{n \rightarrow \infty} \frac1n \log (\#\pi: \vec{0} \rightarrow \lfloor nq \rfloor)\]
 Consider any $q \in \bR^D_{\geq 0}$ with $||q||_1 = 1$ and any $\nu \in \mathcal{M}_+$. Suppose $||(q, \nu)|| \neq -\infty$. Then: 
 \begin{enumerate}[label=(\roman*)]
\item $\nu$ is a probability measure on $[0,1]$.
\item For any $G_{\delta}$ or $F_{\sigma}$ set $A \subseteq [0,1]$,  we have
\[  \Lambda(A)^{\nu(A)} \Lambda(A^C)^{\nu(A^C)} \geq  \nu(A)^{\nu(A)}\nu(A^C)^{\nu(A^C)} e^{||(q,\nu)||-H(q)} \]
\item From Section \ref{section2.6} the Kullback-Leibler divergence is defined as
\[D_{KL}(\nu||\Lambda) = \begin{cases} \int f \log f \ dx \ \mbox{with} \ f:= \frac{d\nu}{dx}, & \nu \ll \Lambda \\ +\infty, & \mbox{otherwise} \end{cases}\]
 Then
\[ ||(q, \nu)|| + D_{KL}(\nu || \Lambda) \leq H(q) \]
so in particular, $\nu$ is absolutely continuous with respect to the Lebesgue measure on $[0,1]$.
\end{enumerate}
\end{theorem}

\begin{remark}
If we take $q = \ell$ and recall that $H(\ell)=\log D$ we get the analogous statements for length 1 direction-free grid entropy.
\end{remark}

\begin{remark}
(i) is exactly what we expect because the empirical measures $\frac1n \mu_{\pi}$ for paths \\ $\pi: \vec{0} \rightarrow \lfloor nq \rfloor$ are non-negative Borel measures with total mass $\frac{||\lfloor nq \rfloor||}n \approx 1$ so in the limit we expect to only observe $\nu$ which are probability measures.
\end{remark}

\begin{remark}
Recalling Theorem \ref{thm5Bates}, (i) implies that $\mathcal{R}^q = \{\nu \in \mathcal{M}_+: ||(q,\nu)|| \geq 0\}$ (and hence $\mathcal{R}^1 = \mathcal{R}^{\ell}$) is weakly compact.
\end{remark}

\begin{remark}
Weaker versions of (iii) previously appeared in \cite{bates} and \cite{rassoul2014quenched}. To be specific, Bates shows that $D_{KL}(\nu||\Lambda) \leq H(q)$ for $\nu \in \mathcal{R}^q$ and it is immediate from the work of Rassoul-Agha and Sepp{\"a}l{\"a}inen that grid entropy $||(q,\nu)||$ is upper bounded by $H(q)$. Here (iii) combines both of these into an amazing inequality.
\end{remark}

\begin{remark}
If we do not make use of our coupling and define grid entropy directly on the space of measures on $\bR$,
(i)-(iii) hold analogously. This is trivial since we do not use weak compactness in the proof of this theorem. 
\end{remark}

\begin{proof}
(i) The intuition is that if $||\nu||_{TV} \neq 1$ then each $\rho(\frac1n \mu_{\pi}, \nu)$ is bounded below away from 0 hence it does not converge to 0.

Suppose $\nu \in \mathcal{M}_+$ is not a probability measure on $[0,1]$. By the reverse triangle inequality for $\rho$, for any $n$ and any path $\pi: \vec{0} \rightarrow \lfloor nq \rfloor$,
\[\rho\bigg(\frac1n \mu_{\pi}, \nu\bigg) \geq | \rho(\mu_{\pi}, 0) - \rho(n\nu, 0)| =  | \ ||\mu_{\pi}||_{TV} - n||\nu||_{TV} \ | =  | \ ||\lfloor nq \rfloor||_1 - n||\nu||_{TV} \ |\]
Since $\frac{||\lfloor nq \rfloor||_1}n \rightarrow ||q||_1 = 1$ and $||\nu||_{TV} \neq 1$ then there is $\delta > 0$ s.t. for large enough $n$, 
\[\rho\bigg(\frac1n \mu_{\pi}, \nu \bigg) \geq \frac1n | \ ||\lfloor nq \rfloor||_1 - n||\nu||_{TV} \ | > \delta\]
It follows that $\min \limits_{\pi}^1 \rho(\frac1n \mu_{\pi}, \nu) \not \rightarrow 0$ a.s. so $||(q, \nu)|| = -\infty$. Contradiction. Therefore $\nu$ must be a probability measure on $[0,1]$.

(ii)  Before proceeding with (ii) we prove a quick claim that will also help with the proof of (iii).

\textit{Claim:} Fix $\alpha \in [0, H(q)]$ and a nonempty set of measures $T \subseteq \mathcal{M}_+$. Then
\begin{align*}
     & \limsup_{n \rightarrow \infty} \frac1{n} \log P\bigg(\exists \lfloor e^{n\alpha} \rfloor \ \mbox{paths} \ \pi: \vec{0} \rightarrow \lfloor nq \rfloor \ \mbox{s.t.} \ \frac1n \mu_{\pi} \in T \bigg) \\
    & \leq H(q)  - \alpha + \limsup_{n \rightarrow \infty} \frac1n \log P \bigg( \frac1n \mu_{\pi} \in T \bigg) 
\end{align*}
In particular, if $\alpha \geq 0$ and
\[ \alpha > H(q) +\limsup_{n \rightarrow \infty} \frac1n \log P \bigg( \frac1n \mu_{\pi} \in T \bigg)  \]
then
\[P\bigg(\exists \lfloor e^{n\alpha} \rfloor \ \mbox{paths} \ \pi: \vec{0} \rightarrow \lfloor nq \rfloor \ \mbox{s.t.} \ \frac1n \mu_{\pi} \in T \ \mbox{i.o.} \bigg) = 0 \]

\textit{Proof.} Indeed, for any $n \in \bN$ we have by Markov's Inequality
\begin{align*}
    P\bigg(\sum_{\pi \in \mathcal{P}(0, \lfloor nq \rfloor)} \mathbf{1}_{\frac1n\mu_{\pi} \in T}  \geq \lfloor e^{n\alpha} \rfloor \bigg) & \leq \frac1{\lfloor e^{n\alpha} \rfloor} E\bigg[ \sum_{\pi \in \mathcal{P}(\vec{0}, \lfloor nq \rfloor)} \mathbf{1}_{\frac1n\mu_{\pi} \in T}  \bigg] \\
    & = \frac{|\mathcal{P}(\vec{0}, \lfloor nq \rfloor)|}{\lfloor e^{n\alpha} \rfloor} P \bigg( \frac1n \mu_{\pi} \in T \bigg) 
\end{align*}
Taking $\limsup \limits_{n \rightarrow \infty} \frac1n \log$ of both sides and using the definition of $H(q)$ we get  the desired inequality. 

Now suppose
\[ \alpha > \alpha -\delta >  H(q) +\limsup_{n \rightarrow \infty} \frac1n \log P \bigg( \frac1n \mu_{\pi} \in T \bigg)  \]
Then for large enough $n$,
\[P\bigg(\sum_{\pi \in \mathcal{P}(0, \lfloor nq \rfloor)} \mathbf{1}_{\frac1n\mu_{\pi} \in T} \geq \lfloor e^{n\alpha} \rfloor \bigg) \leq \mbox{exp}\bigg[n \bigg( H(q) - \alpha + \delta + \limsup_{n \rightarrow \infty} \frac1n \log P \bigg( \frac1n \mu_{\pi} \in T \bigg) \bigg) \bigg] \]
where the exponent is a strictly negative constant multiple of $n$. Hence, by the Borel\hypersetup{}Cantelli Lemma,
\[P\bigg(\exists \lfloor e^{n\alpha} \rfloor \ \mbox{paths} \ \pi: \vec{0} \rightarrow \lfloor nq \rfloor \ \mbox{s.t.} \ \frac1n \mu_{\pi} \in T \ \mbox{i.o.} \bigg) = 0 \]
$\square$ (Claim)

\bigskip

Now we begin the proof of (ii). The idea is to use the definition of the Levy-Prokhorov metric to unpack what $\rho (\frac1n \mu_{\pi}, \nu) < \epsilon$ means in terms of the values of the labels along $\pi$.

First, consider any closed $A \subseteq [0,1]$. Observe that $A = \overline{A} = \bigcap \limits_{\epsilon > 0} A^{\epsilon}$ hence by continuity from above, $\nu(A^{\epsilon}) \downarrow \nu(A)$ and $\Lambda(A^{\epsilon}) \downarrow \Lambda(A)$ as $\epsilon \rightarrow 0^+$.

We compute the probability that $\rho(\frac1n \mu_{\pi}, \nu) < \epsilon $ for a path $\pi: \vec{0} \rightarrow \lfloor nq \rfloor$ directly, using the definitions of $\rho$ and $\mu_{\pi}$, and then we will apply the Claim.

Consider $n \in \bN$. For any path $\pi: \vec{0} \rightarrow \lfloor nq \rfloor$, by definition of the Levy-Prokhorov metric and the  empirical measure $\mu_{\pi}$,
\[\rho \bigg(\frac1n \mu_{\pi}, \nu \bigg) < \epsilon \Rightarrow \nu(A) \leq \frac1n \mu_{\pi} (A^{\epsilon}) + \epsilon \ \mbox{and} \ \frac1n\mu_{\pi} (A^{\epsilon}) \leq \nu((A^{\epsilon})^{\epsilon})  + \epsilon\]
\begin{align*}
& \Rightarrow \mbox{$\pi$ has $\geq \lfloor n(\nu(A)-\epsilon) \rfloor$ edge labels in $A^{\epsilon}$}\\
& \qquad  \mbox{and $\leq \lfloor n(\nu(A^{2\epsilon})+\epsilon) \rfloor$  edge labels in $A^{\epsilon}$} \\
&\Rightarrow \mbox{$\pi$ has $\geq \lfloor n(\nu(A)-\epsilon) \rfloor$ edge labels in $A^{\epsilon}$}  \\
& \qquad  \mbox{and $\geq ||\lfloor nq \rfloor ||_1 - \lfloor n(\nu(A^{2\epsilon})-\epsilon) \rfloor$ edge labels in $(A^{\epsilon})^C$}
\end{align*}
If $\nu(A) = 0$ or $\nu(A) = 1$ then  we may completely omit the first/second half of the statement above as it is trivial. Otherwise, we take $\epsilon > 0$ small enough so that $\nu(A^{2\epsilon})+\epsilon < 1$ and $n$ large enough so that $n(\nu(A^{2\epsilon})+\epsilon) < ||\lfloor nq \rfloor||_1$. For the sake of convenience we only show the computation in the latter case; the $\nu(A) = 0$ or $\nu(A) = 1$ case is merely a simplified version of it.

Since $U_e$ are i.i.d. Unif$[0,1]$, then
\begin{equation}\label{label14}
\begin{split}
   P \bigg(\rho \bigg(\frac1n \mu_{\pi}, \nu\bigg) < \epsilon \bigg) &\leq P \bigg(\begin{array}{l}
    \mbox{$\pi$ has $\geq \lfloor n(\nu(A)-\epsilon) \rfloor$ edge labels in $A^{\epsilon}$,}  \\
    \mbox{$\geq ||\lfloor nq \rfloor ||_1 - \lfloor n(\nu(A^{2\epsilon})-\epsilon) \rfloor$ labels in $(A^{\epsilon})^C$} 
  \end{array} \bigg) \\
& \leq \binom{||\lfloor nq \rfloor ||_1} {\lfloor n(\nu(A)-\epsilon) \rfloor, ||\lfloor nq \rfloor ||_1 - \lfloor n(\nu(A^{2\epsilon})-\epsilon) \rfloor } \\
& \qquad \qquad  \cdot \Lambda(A^{\epsilon})^{n(\nu(A) - \epsilon)} \Lambda((A^{\epsilon})^C)^{||\lfloor nq \rfloor ||_1 - n(\nu(A^{2\epsilon})-\epsilon)} 
\end{split}
\end{equation}
by the union bound. We  get the asymptotics of this binomial coefficient from Lemma \ref{lemma1}:
\begin{equation}\label{label16}
\begin{split}
        &\binom{||\lfloor nq \rfloor ||_1} {\lfloor n(\nu(A)-\epsilon) \rfloor, ||\lfloor nq \rfloor ||_1 - \lfloor n(\nu(A^{2\epsilon})-\epsilon) \rfloor } \\
      &= \bigg(\frac{1}{ (\nu(A)-\epsilon)^{\nu(A)-\epsilon}  (1-\nu(A^{2\epsilon})+\epsilon)^{1-\nu(A^{2\epsilon})+\epsilon} (\nu(A^{2\epsilon})- \nu(A))^{\nu(A^{2\epsilon})- \nu(A)}}  + o(1) \bigg)^n 
\end{split}
\end{equation}
Therefore
\begin{align*}
&\limsup_{n \rightarrow \infty} \frac1n \log P \bigg(\rho \bigg(\frac1n \mu_{\pi}, \nu\bigg) < \epsilon \bigg) \leq \\
        & \log  \frac{\Lambda(A^{\epsilon})^{\nu(A) - \epsilon} \Lambda((A^{\epsilon})^C)^{1 - \nu(A^{2\epsilon})+\epsilon} }{ (\nu(A)-\epsilon)^{\nu(A)-\epsilon}  (1-\nu(A^{2\epsilon})+\epsilon)^{1-\nu(A^{2\epsilon})+\epsilon} (\nu(A^{2\epsilon})- \nu(A))^{\nu(A^{2\epsilon})- \nu(A)}}
\end{align*}
Combining this with  the Claim with $T = B_{\epsilon}(\nu)$, we get that 
\begin{align*}
       & P\bigg(\exists \lfloor e^{n\alpha} \rfloor \ \mbox{paths} \ \pi: \vec{0} \rightarrow \lfloor nq \rfloor \ \mbox{with} \ \rho \bigg(\frac1n \mu_{\pi}, \nu \bigg) < \epsilon \ \mbox{i.o.} \bigg) = 0 \\ 
       &\forall \alpha >  H(q)  + \log \frac{\Lambda(A^{\epsilon})^{\nu(A) - \epsilon} \Lambda((A^{\epsilon})^C)^{1 - \nu(A^{2\epsilon})+\epsilon} }{ (\nu(A)-\epsilon)^{\nu(A)-\epsilon}  (1-\nu(A^{2\epsilon})+\epsilon)^{1-\nu(A^{2\epsilon})+\epsilon} (\nu(A^{2\epsilon})- \nu(A))^{\nu(A^{2\epsilon})- \nu(A)}} 
\end{align*}
By Corollary \ref{corollary18} (ii), it follows that 
\begin{align*} & ||(q,\nu)|| \leq \\
& H(q)  + \log \frac{\Lambda(A^{\epsilon})^{\nu(A) - \epsilon} (1-\Lambda(A^{\epsilon}))^{1 - \nu(A^{2\epsilon})+\epsilon} }{ (\nu(A)-\epsilon)^{\nu(A)-\epsilon}  (1-\nu(A^{2\epsilon})+\epsilon)^{1-\nu(A^{2\epsilon})+\epsilon} (\nu(A^{2\epsilon})- \nu(A))^{\nu(A^{2\epsilon})- \nu(A)}}  
\end{align*}
This holds for arbitrary $ \epsilon > 0$. Taking $\epsilon \rightarrow 0^+$ and using $\nu(A^{\epsilon}) \downarrow \nu(A)$ and $\Lambda(A^{\epsilon}) \downarrow \Lambda(A)$, 
\[ 1 \leq e^{H(q)-||(q,\nu)||} \frac{ \Lambda(A)^{\nu(A)} \Lambda(A^C)^{\nu(A^C) }}{ \nu(A)^{\nu(A)}  \nu(A^C)^{\nu(A^C)}} \]
Therefore
\begin{equation} \label{14}  \Lambda(A)^{\nu(A)} \Lambda(A^C)^{\nu(A^C)} \geq  \nu(A)^{\nu(A)}\nu(A^C)^{\nu(A^C)} e^{||(q,\nu)|| -H(q)}
\end{equation}
This equation is symmetric in $A, A^C$ hence it holds for all $A$ open or closed. By continuity from below/above it holds for all $G_{\delta}$ and $F_{\sigma}$ subsets of $[0,1]$.

(iii) The idea is to  again use the Claim, this time to compute $\limsup \limits_{n \rightarrow \infty} \frac1n \log P (\rho (\frac1n \mu_{\pi}, \nu) < \epsilon )$  using Sanov's Theorem. 

 Fix $\epsilon > 0$. By Sanov's Theorem (Theorem \ref{thmSanov}), since $\overline{B_{2\epsilon}}(\nu)$ is weakly closed then
\[\limsup_{n \rightarrow \infty} \frac1n \log P \bigg(\frac1{|\pi|} \mu_{\pi} \in \overline{B_{2\epsilon}}(\nu) \bigg) \leq -\inf_{\xi \in \overline{B_{2\epsilon}}(\nu)} D_{KL} (\xi ||\Lambda)\]
where the right hand side might a priori be $-\infty$. 

Since $\rho(\frac1{|\pi|}\mu_{\pi}, \frac1n \mu_{\pi}) \leq |1- \frac{||\lfloor nq \rfloor||_1}n| \rightarrow 0$ as $n \rightarrow \infty$, it follows that
\begin{equation*}
\limsup_{n \rightarrow \infty} \frac1n \log P \bigg(\frac1{n} \mu_{\pi} \in B_{\epsilon}(\nu) \bigg) \leq \limsup_{n \rightarrow \infty} \frac1n \log P \bigg(\frac1{|\pi|} \mu_{\pi} \in \overline{B_{2\epsilon}}(\nu) \bigg) \leq -\inf_{\xi \in \overline{B_{2\epsilon}}(\nu)} D_{KL} (\xi ||\Lambda)
\end{equation*}
By the Claim from (ii) with $T = B_{\epsilon}(\nu)$, 
\begin{equation*}
        P\bigg(\exists \lfloor e^{n\alpha} \rfloor \ \mbox{paths} \ \pi: \vec{0} \rightarrow \lfloor nq \rfloor \ \mbox{with} \ \rho \bigg(\frac1n \mu_{\pi}, \nu \bigg) < \epsilon \ \mbox{i.o.} \bigg) = 0 \ \forall \alpha >  H(q)  -\inf_{\xi \in \overline{B_{2\epsilon}}(\nu)} D_{KL} (\xi ||\Lambda)
\end{equation*}
Hence, by Corollary \ref{corollary18} (ii),
\[ ||(q,\nu)|| \leq H(q)  -\inf_{\xi \in \overline{B_{2\epsilon}}(\nu)} D_{KL} (\xi ||\Lambda) \]

Take $\epsilon = \frac1k \downarrow 0$ and let $\xi_k = \argmin \limits_{\xi \in \overline{B_{\frac{2}k}}(\nu)} D_{KL}(\xi ||\Lambda)$ (the infimum over the weakly closed $ \overline{B_{\frac{2}k}}(\nu)$ is achieved since relative entropy is lower semicontinuous). Then $\xi_k \Rightarrow \nu$ hence 
\[ D_{KL}(\nu || \Lambda) + ||(q,\nu)|| \leq \liminf_{k \rightarrow \infty} D_{KL}(\xi_k || \Lambda) + ||(q,\nu)|| \leq H(q)\]
as desired. 

Finally, recall that $\nu \in \mathcal{R}^q$ hence $||(q,\nu)|| \geq 0$ so $D_{KL}(\nu || \Lambda)$ is finite and $\nu \ll \Lambda$.
\end{proof}

It is not clear in general whether any probability measure $\nu$ satisfying $D_{KL}(\nu || \Lambda) \leq H(q)$ results in finite grid entropies $||(q, \nu)||$ for $q \in \bR^D_{\geq 0}, ||q||_1 = 1$.  

In one specific case, however, this does hold trivially. When $q$ is a permutation of \\ $(1,0,\ldots, 0)$, so that there is exactly one NE path $\vec{0} \rightarrow nq$, we get
\[||(q, \nu)|| > -\infty \Leftrightarrow D_{KL} (\nu || \Lambda) = 0 \Leftrightarrow \nu = \Lambda\]
with the value of $||(q, \Lambda)||$ being 0. Note that this completely covers the $D=1$ case. 

\bigskip

Also, in general we cannot explicitly compute grid entropy. But when the target measure $\nu$ is $\Lambda$, we can easily show that grid entropy is maximal. This makes sense, since the Glivenko-Cantelli Theorem says we expect that most empirical measures observed are very close to the edge label distribution.

\begin{corollary}
Let $q \in \bR^D_{\geq 0} \setminus \{\vec{0}\}$. Then $||(q, q \Lambda)|| = H(q)$.
\end{corollary}

\begin{proof}
Since $H(q)$ and $||(q,\nu)||$ are positive-homogeneous (in the sense that $H(cq) = cH(q)$ and $||(cq, c\nu)|| = c||(q,\nu)|| \ \forall c > 0$), we may without loss of generality assume $||q||_1 = 1$.

Suppose $||(q, \Lambda)|| < H(q)$. Consider any $\epsilon > 0$. By Sanov's Theorem  (Theorem \ref{thmSanov}), 
 \begin{equation}\label{eq10} \limsup_{n \rightarrow \infty} \frac1n \log P\bigg(\frac1{n} \mu_{\pi} \in B_{\epsilon}(\Lambda)^C \bigg) \leq -\inf_{\xi \in B_{\epsilon}(\Lambda)^C} D_{KL}(\xi || \Lambda) 
 \end{equation}
where $B_{\epsilon}(\Lambda)^C$ is weakly closed and relative entropy is lower semicontinuous hence the infimum is achieved by some $\xi \in B_{\epsilon}(\Lambda)^C$. But $\xi \neq \Lambda$ implies $-D_{KL}(\xi||\Lambda) < 0$. Fix any 
\[ \max(||(q,\Lambda)||, H(q) - D_{KL}(\xi||\Lambda)) < \alpha < H(q) \]

By the Claim from the proof of (ii) in the previous theorem with $T = B_{\epsilon}(\Lambda)^C$,  we have
\[P\bigg(\exists \lceil e^{n\alpha} \rceil \ \mbox{paths} \ \pi: \vec{0} \rightarrow \lfloor nq \rfloor \ \mbox{s.t.} \ \rho\bigg(\frac1n\mu_{\pi}, \nu\bigg) \geq \epsilon \ \mbox{i.o.}\bigg) = 0 \]
hence
\[P\bigg(\exists  (\#\pi: \vec{0} \rightarrow \lfloor nq \rfloor) - \lceil e^{n\alpha} \rceil \ \mbox{paths} \ \pi: \vec{0} \rightarrow \lfloor nq \rfloor \ \mbox{s.t.} \ \rho\bigg(\frac1n\mu_{\pi}, \nu\bigg) < \epsilon \ \mbox{eventually}\bigg) = 1 \]

But $(\#\pi: \vec{0} \rightarrow \lfloor nq \rfloor) - \lceil e^{n\alpha} \rceil \geq \lfloor e^{n\alpha} \rfloor$ for large enough $n$ since $\alpha < H(q)$, so
\[P\bigg(\exists \lfloor e^{n\alpha} \rfloor \ \mbox{paths} \ \pi: \vec{0} \rightarrow \lfloor nq \rfloor \ \mbox{s.t.} \ \rho\bigg(\frac1n\mu_{\pi}, \nu\bigg) < \epsilon \ \mbox{eventually} \bigg) = 1\]
for this arbitrary $\epsilon > 0$, hence $\alpha \geq ||(q, \Lambda)||$. Contradiction.
\end{proof}

Continuity is another property of interest. It turns out that grid entropy is upper semicontinuous in its argument(s); this follows immediately from \eqref{translationEqn} and the lower semicontinuity of relative entropy. However, we have not yet proved \eqref{translationEqn} so we prove a slightly weaker statement using the tools developed in this paper. We also note that grid entropy is lower semicontinuous (hence fully continuous) in some cases.

 Note that here we no longer restrict ourselves to the case when $||q||_1 = 1$.

\begin{theorem}\label{thm20}
Let $(q, \nu), (q^k, \nu^k) \in \bR^D \times \mathcal{M}_+$ s.t. $q^k \rightarrow^{L^1} q, \nu^k \Rightarrow \nu$.
\begin{enumerate}[label=(\roman*)]
\item If either $q = \vec{0}$ or  for large enough $k$ we have  $q - q^k \in \bR^D_{\geq 0} $ or for large enough $k$ we have $(q^k)_i = 0$ for all $i$ s.t. $q_i = 0$  then
\[ \limsup_{k \rightarrow \infty} ||(q^k, \nu^k)|| \leq ||(q, \nu)|| \]
\item If $||(q^k-q, \nu^k-\nu)|| \neq - \infty$, for large enough $k$ then
\[||(q, \nu)||  \leq \liminf_{k \rightarrow \infty} ||(q^k, \nu^k)||  \]
\item If there exist constants $ C_k > 0$ s.t. $C_k \rightarrow 0$ and s.t. for large enough $k$,
\[ \bigg|\bigg|\bigg(q-\frac{q-q^k}{C_k}, \nu-\frac{\nu-\nu^k}{C_k} \bigg)\bigg|\bigg| \neq - \infty\]
then
\[ ||(q, \nu)|| \leq \liminf_{k \rightarrow \infty} ||(q^k, \nu^k)||\]
\end{enumerate}
\end{theorem}

\begin{remark}
In the case when $q^k = q$ for large enough $k$, (i) shows that direction-fixed grid entropy $||(q,\nu)||$ is always upper semicontinuous in $\nu$. Taking $q^k = ||\nu^k||_{TV} \ell \rightarrow ||\nu||_{TV} \ell$ we get that
\[ \limsup_{k \rightarrow \infty} ||\nu^k|| = \limsup_{k \rightarrow \infty} ||(||\nu^k||_{TV} \ell, \nu^k)|| \leq ||(||\nu||_{TV} \ell, \nu) || = ||\nu|| \]
so direction-free grid entropy is upper semicontinuous.
\end{remark}

\begin{remark}
In the assumptions in (ii) and (iii) it is implicit that the parameters lie in \\ $\bR^D_{\geq 0} \times \mathcal{M}_+$ because otherwise the grid entropy would be $-\infty$.
\end{remark}

\begin{remark}
The conditions in (i), (ii), (iii) are not mutually exclusive so for some sequences we may combine the statements to obtain 
\[ ||(q, \nu)|| = \lim_{k \rightarrow \infty} ||(q^k, \nu^k)||\]
\end{remark}

\begin{proof}
For all three parts, we may assume the conditions hold for all $k \in \bN$.\\
(i) \textbf{Case 1}: $q =\vec{0}$\\
First suppose $\nu = 0$ so that $||(q, \nu)|| = 0$. By Theorem \ref{thmEquiv}, 
\[||(q^k, \nu^k)|| \in \{-\infty\} \cup [0, H(q^k)] \ \forall k\]
hence
\begin{align*}
\limsup_{k \rightarrow \infty} ||(q^k, \nu^k)|| &\leq \limsup_{k \rightarrow \infty} H(q^k) \\
&= \limsup_{k \rightarrow \infty} ||q^k||_1 \log ||q^k||_1 - \sum_{i=1}^D (q^k)_i \log (q^k)_i \\
&=0 
\end{align*}
since $||q^k-q||_1 = ||q^k||_1 \rightarrow 0$.

On the other hand, if $\nu \neq 0$ then for large enough $k$ we have $||q^k||_1 \neq ||\nu^k||_{TV}$ (because otherwise $||\nu^{k}||_{TV} \rightarrow 0$ and hence $\nu$, the weak limit of the $\nu^k$, would have to be 0). But then
\[\limsup_{k \rightarrow \infty} ||(q^k, \nu^k)|| = -\infty = ||(q, \nu)|| \]
by Theorem \ref{thm19} (i).

\textbf{Case 2}: $q- q^k \in \bR^D_{\geq 0} \ \forall k$\\
Recall that the infimum of a family of upper semicontinuous functions is upper semicontinuous. Thus it suffices to fix $\epsilon > 0$ and show that
\begin{equation} \label{equation21} \limsup_{k \rightarrow \infty} \widetilde{d}^{\epsilon}((\vec{0},0), (q^k, \nu^k)) \leq \widetilde{d}^{\epsilon}((\vec{0},0), (q, \nu)) \ \mbox{a.s.}
\end{equation}
But for every $k, n \in \bN$ we have by Lemma \ref{lemma12} and the fact that $q- q^k \in \bR^D_{\geq 0}$,
\[ \frac{d^{\epsilon}((\vec{0},0), (nq^k, n\nu^k))}n - \frac1{\epsilon} \bigg(\frac1n ||\lfloor nq \rfloor - \lfloor nq^k \rfloor||_1 + \rho(\nu,\nu^k)\bigg) \leq \frac{d^{\epsilon}((\vec{0},0), (nq^k, n\nu^k))}n  \]
For a fixed $k$, taking $n \rightarrow \infty$ we get
\[ \widetilde{d}^{\epsilon}((\vec{0},0), (q^k, \nu^k))  - \frac1{\epsilon} (||q-q^k||_1 + \rho(\nu,\nu^k)) \leq \widetilde{d}^{\epsilon}((\vec{0},0), (q, \nu)) \ \mbox{a.s.} \]
\eqref{equation21} follows immediately.

\textbf{Case 3}: $(q^k)_i = 0$ whenever $q_i = 0$ for all $k \in \bN$\\
We may assume $q \neq \vec{0}$ (because this scenario was covered by Case 1). In particular, we may assume without loss of generality that $q^k \neq \vec{0} \ \forall k$.

To show 
\[ \limsup_{k \rightarrow \infty} ||(q^k, \nu^k)|| \leq  ||(q, \nu)||\]
it suffices to prove this when $||(q^k, \nu^k)|| \geq 0 \ \forall k$. In particular, this implies that \\ $(q^k)_i \geq 0 \ \forall k \in \bN, 1 \leq i \leq D$ and hence $q_i \geq 0 \ \forall 1 \leq i \leq D$.

Now, the trick is to rescale the $q^k$ by a factor $B_k \geq 1$ converging to 1 in $k$ s.t. we land back in Case 2. More precisely, define
\[ B_k := \max \bigg(\max_{1 \leq i \leq D \ \mbox{s.t.} \ (q^k)_i \neq 0} \frac{q_i}{(q^k)_i}, \max_{1 \leq i \leq D \ \mbox{s.t.} \ q_i \neq 0} \frac{(q^k)_i}{q_i}\bigg) \]
Note that $B_k \geq 0$ because $q_i, (q^k)_i \geq 0$. Since $q^k \neq \vec{0}$ then there is some $i$ s.t. $(q^k)_i > 0$ so by the assumption of Case 3, $q_i > 0$. Thus $B_k \geq 1$.

Furthermore, $||q-q^k||_1 \rightarrow 0$ implies $B_k \rightarrow 1$. In particular, $\frac1{B_k} q_k \rightarrow^{L^1} q, \frac1{B_k} \nu^k \Rightarrow \nu$. By the choice of $B_k$ we have for all $1 \leq i \leq D$, either $q_i = 0$ hence $B_k q_i = 0 = (q^k)_i$ or $q_i \neq 0$ so $B_k q_i \geq (q^k)_i$. Thus
\[B_kq - q^k \in \bR^D_{\geq 0} \Rightarrow q - \frac1{B_k} q^k \in \bR^D_{\geq 0} \]
By Case 2 and positive-homogeneity of grid entropy,
\[ \limsup_{k \rightarrow \infty} ||( q^k, \nu^k)|| = \limsup_{k \rightarrow \infty} \bigg|\bigg|\bigg(\frac1{B_k} q^k, \frac1{B_k}\nu^k\bigg)\bigg|\bigg| \leq  ||(q,  \nu)|| \]
as desired.

(ii) By the triangle inequality for a directed norm with negative sign,
\[||(q^k, \nu^k) || \geq ||(q, \nu)|| + ||(q^k-q, \nu^k-\nu)|| \geq ||(q, \nu)|| \ \forall k\]
since $||(q^k-q, \nu^k-\nu)|| \neq -\infty$. Therefore
\[ \liminf_{k \rightarrow \infty} ||(q^k, \nu^k)|| \geq ||(q,\nu)||\]

(iii) Now suppose
\[  \bigg|\bigg|\bigg(q-\frac{q-q^k}{C_k}, \nu-\frac{\nu-\nu^k}{C_k} \bigg)\bigg|\bigg|  \geq 0 \ \mbox{and} \  C_k \geq 0 \ \forall k, C_k \rightarrow 0\]
By the reverse triangle inequality and positive-homogeneity,
\begin{align*}
    (1-C_k)||(q, \nu) || &= ||(q^k - (C_kq - (q-q^k)), \nu^k - (C_k \nu - (\nu-\nu^k)))||  \\
    & \leq ||(q^k, \nu^k)|| - C_k \bigg|\bigg|\bigg(q-\frac{q-q^k}{C_k}, \nu-\frac{\nu-\nu^k}{C_k} \bigg)\bigg|\bigg|\\
    & \leq ||(q^k, \nu^k)|| 
\end{align*}
But $C_k \rightarrow 0$ hence
\[ ||(q, \nu)|| \leq \liminf_{k \rightarrow \infty} ||(q^k, \nu^k)||\]

\end{proof}

In Lemma \ref{maximalLemma} we made a simple observation that is worth mentioning explicitly:  $||(q, \nu)||$ is invariant under permutations of the coordinates of $q$ because of the symmetry of the grid lattice $\bZ^D$.

A last property we are interested in is the convexity of $\mathcal{R}^q, \mathcal{R}^t$.

\begin{lemma}\label{lemma22}
Let $q \in \bR^D_{\geq 0}, t \geq 0$. The sets $\mathcal{R}^q = \{(q, \nu) \in \bR^D_{\geq 0} \times \mathcal{M}_+: ||(q, \nu)|| \geq 0\}$, $\mathcal{R}^t = \{\nu \in \mathcal{M}_t: ||\nu|| \geq 0\}$ are convex.
\end{lemma}
\begin{proof}
Grid entropy is concave in  by positive-homogeneity and by the reverse triangle-inequality it satisfies.
\end{proof}

\section{Application: Directed Polymers}\label{section5}
We now apply our results to give variational formulas for the point-to-point/point-to-level Gibbs Free Energy  in terms of direction-fixed/direction-free grid entropy and  last passage time.  As mentioned, these variational formulas have previously appeared in \cite{rassoul2014quenched}, \cite{georgiou}. Our aim is to arrive at these from the framework presented in this paper with the hope of providing fresh insights.

\subsection{Setup and Known Results}
Recall that the measurable function $\tau: [0,1] \rightarrow \bR$ determines the edge weights \\ $\tau_e = \tau(U_e)$. We will be dealing with the Gibbs Free Energy so to make our life easy we restrict ourselves in Section \ref{section5} to $\tau$ bounded and measurable. This could be extended to measurable but possibly unbounded $\tau$ by truncating $\tau$ at constants $C > 0$ and taking a supremum over $C > 0$.

Define the last passage time between points $p, q \in \bR^D$ to be
\[ T(p, q) = \sup_{\pi: \lfloor p \rfloor \rightarrow \lfloor q \rfloor} T(\pi)  = \sup_{\pi: \lfloor p \rfloor \rightarrow \lfloor q \rfloor} \sum_{e \in \pi} \tau(U_e)  \]
where this is $-\infty$ if $q-p \notin \bR^D_{\geq 0}$. A standard result (see \cite{martin} for example) is the existence of a last passage time constant, or in other words a deterministic limiting value for the last passage time in a fixed direction $q \in \bR^D_{\geq 0}$. 

\begin{theorem}
 Let $q \in \bR^D_{\geq 0}$. Then there is a $\tau$-dependent time constant $\lambda_q(\tau) \in \bR$  s.t.
\[ \frac{T(\vec{0}, nq)}n \rightarrow \lambda_q(\tau)\]
$\lambda_q$ is homogeneous and concave in $q$.
\end{theorem}

\begin{remark}
As was the case with grid entropy, the symmetries of the grid and concavity of passage time imply that the time constant is maximized along the direction $\ell = (\frac1D,\ldots, \frac1D)$:
\[\sup \limits_{q \in \bR^D, ||q||_1 =1} \lambda_q = \lambda_{\ell}\]
\end{remark}

\begin{remark}
Thus far the time constant is only known to be explicitly computable when the underlying distribution $\theta$ is exponential.
\end{remark}

It is useful to view last passage time as a linear functional. For any NE path $\pi: \lfloor np \rfloor \rightarrow \lfloor nq \rfloor$, note that
\[T(\pi) = \sum_{e \in \pi} \tau(U_e) =\int_0^{1} \tau(u) d(\mu_{\pi}) :=  \langle \tau, \mu_{\pi} \rangle \ \mbox{so} \ \frac{T(\pi)}n = \bigg \langle \tau, \frac1n \mu_{\pi} \bigg \rangle  \]
where $\langle f, \nu \rangle$ denotes the linear functional that is integration of a measurable function \\ $f: [0,1] \rightarrow \bR$ against a measure $\nu \in \mathcal{M}_+$. In this language, weak convergence $\nu_k \Rightarrow \nu$ is equivalent to
\[ \lim_{k \rightarrow \infty} \langle f, \nu_k \rangle = \langle f, \nu \rangle \ \forall \ \mbox{bounded, continuous functions $f$}\]
Of course, $\tau$ is likely not continuous. However, recalling  \cite[Lemma 6.15]{bates} , we see that there is an a.s. event on which $\frac1n \mu_{\pi_n}$ converging weakly to $\nu$ implies $(\tau)_{\ast}(\frac1n \mu_{\pi_n})$ converges weakly to $(\tau)_{\ast}(\nu)$ hence
\[ \bigg\langle \tau, \frac1n \mu_{\pi_n} \bigg\rangle = \bigg\langle Id, (\tau)_{\ast}\bigg(\frac1n \mu_{\pi_n}\bigg) \bigg\rangle \rightarrow \bigg \langle Id, (\tau)_{\ast}(\nu) \bigg\rangle =\langle \tau, \nu\rangle   \]
where $Id(x) =x$ is the identity function on [0,1].

By using the compactness of $\mathcal{R}^q$ and Theorem \ref{thm5Bates}, Bates establishes a variational formula for the direction-fixed limit shape. See \cite[Theorem 2.1]{bates} for details. We restate this theorem here in our LPP setting, which can be proved by tweaking Bates' original argument. We give our own proof in the next section.

\begin{theorem}\label{thm27}
Let $q \in \bR^D_{\geq 0}$. Recall that Theorem  \ref{thm5Bates} gives a deterministic weakly closed set $\mathcal{R}^q$ which a.s. consists of all weak limits of subsequences of $\mu_{\pi}$ for paths $\pi$ of increasing length. Then for any measurable $\tau: [0,1] \rightarrow \bR$, the corresponding limit shape is given by
\[ \lambda_q(\tau) = \sup_{\nu \in \mathcal{R}^q} \langle \tau, \nu \rangle \]
The direction-free analogue holds as well.
\end{theorem}

\begin{remark} \label{thm27REM}
Since $\mathcal{R}^q$ is weakly closed and nonempty (because $||q||_1 \Lambda \in \mathcal{R}^q$ from before), then the  set of maximizers is nonempty:
\[ \mathcal{R}^q_{\tau} := \{\nu \in \mathcal{R}^q: \langle \tau, \nu \rangle = \lambda_q(\tau)\} \neq \emptyset\]
As Bates remarks, when $\mathcal{R}^q_{\tau}$ is a singleton, empirical measures along geodesics converge weakly to this unique maximizer (or analogously unique minimizer in FPP), answering in the affirmative Hoffman's question. It is further pointed out that $\mathcal{R}^q_{\tau}$ is a singleton for a dense family of functions $\tau$. See \cite{bates} for more details.
\end{remark}

\begin{remark}
Noting that the limit shape is the zero temperature analogue of $\beta$-Gibbs Free Energy, these direction-fixed/direction-free variational formulas for the limit shape are also given as (7.6), (7.5) in \cite{georgiou}. 
\end{remark}

Finally, let us recall the directed polymer model, still restricting ourselves to the LPP on $\bR^D$ setting. 

\begin{definition}
Fix a direction $q \in \bR^D_{\geq 0}$, an inverse temperature $\beta > 0$ and a bounded measurable $\tau:[0,1] \rightarrow \bR$. The  point-to-point $\beta$-partition function in direction $q$ is defined to be
\[ Z_{n, q}^{\beta} = \sum_{\pi \in \mathcal{P}( \vec{0}, \lfloor nq \rfloor)} e^{\beta T(\pi)} \]
The corresponding point-to-point $\beta$-polymer measure on the set of paths $\pi: \vec{0} \rightarrow \lfloor nq \rfloor$ is
\[ \rho_{n,q}^{\beta}(d\pi) = \frac1{ Z_{n, q}^{\beta}} e^{\beta T(\pi)}\]
and the corresponding  point-to-point $\beta$- Gibbs Free Energy  is defined to be
\[ G_{q}^{\beta}  = \lim_{n \rightarrow \infty} \frac1n \log Z_{n,q}^{\beta}\]
Suppressing the direction $q$ we get point-to-level analogues (to the "level" $x_1+\ldots+x_D = n$):
\[ Z_{n}^{\beta} = \sum_{\pi \in \mathcal{P}_n(\vec{0})} e^{\beta T(\pi)},\ \rho_{n}^{\beta}(d\pi) = \frac1{ Z_{n}^{\beta}} e^{\beta T(\pi)},\ G^{\beta}  = \lim_{n \rightarrow \infty} \frac1n \log Z_{n}^{\beta}\]
\end{definition}

\subsection{Variational Formulas for PTP/PTL Gibbs Free Energies}
We  derive a formula for the point-to-point Gibbs Free Energy as the supremum over measures $\nu$ of the sum of the integral $\beta \langle \tau, \nu \rangle$ and grid entropy $||(q, \nu)||$. Since direction-free grid entropy is just grid entropy in the $\ell$ direction, then we arrive at a similar formula for the point-to-level Gibbs Free Energy.

\begin{theorem}\label{theorem28}
Fix an inverse temperature $\beta > 0$ and a bounded measurable \\ $\tau: [0,1] \rightarrow \bR$.
% s.t.
%\[E[e^{\beta \tau(U)}] < \infty  \ \mbox{for} \ U \sim Unif[0,1]\]
\begin{enumerate}[label=(\roman*)]
\item  Fix $q \in \bR^D_{\geq 0}$. The point-to-point Gibbs Free Energy is given by
\[ G_{q}^{\beta}(\tau) = \sup_{\nu \in \mathcal{M}_+} [ \beta \langle \tau, \nu \rangle + ||(q, \nu)||] \ \mbox{a.s.} \]
Moreover, this supremum is achieved by some $\nu \in \mathcal{R}^q$.
\item The point-to-level Gibbs Free Energy is given by 
\[ G^{\beta}(\tau) = \sup_{\nu \in \mathcal{M}_1} [ \beta \langle \tau, \nu \rangle + || \nu||] = G_{\ell}^{\beta} = \sup_{q \in \bR^D_{\geq 0}, ||q||_1=1} G^{\beta}_{q} \ \mbox{a.s.} \]
\end{enumerate}
\end{theorem}

\begin{remark}
These entropy variational formulas appear as (7.12), (7.10) in \cite{georgiou} respectively, with a different normalization. The point-to-point formula originally appears as (5.4) in \cite{rassoul2014quenched}.
\end{remark}

\begin{remark}
If $\tau$ is non-negative then the variational formulas are trivially positive\hyp{}homogeneous.
\end{remark}

\begin{remark}
Again, we may extend these formulas to a supremum over all measurable $\tau$ by truncating $\tau$ in $\langle \tau, \nu \rangle$ at some $C > 0$ and then taking a supremum over $C>0$ of the variational formula. 
\end{remark}

We require a short lemma in order to prove our variational formula. The main technical difficulty is that each direction-$q$ grid entropy is an almost sure limit:
\[ ||(q,\nu)|| = \inf_{\epsilon > 0} \lim_{n \rightarrow \infty} \frac{d^{\epsilon}((\vec{0},0), (nq, n\nu))}n \ \mbox{a.s.}\]
but we are now dealing with a whole space of measures, which is uncountable. Thus we need to assume these limits exist for a countable dense family of $\nu$'s and approximate  the sums of $e^{\beta T(\pi)}$ over  paths with empirical measure in $\epsilon$-balls.

\begin{lemma}\label{lemma32}
Fix $q \in \bR^D_{\geq 0}, \chi \in \mathcal{M}_+, \gamma > 0$ and  $m \in \bN$. For any $n \in \bN$,
\begin{align*}
& \log \bigg(\sum_{\pi: \vec{0} \rightarrow \lfloor nq \rfloor \ \mbox{s.t.} \ \frac1{n}\mu_{\pi} \in B_{\frac1m}(\chi)} e^{\beta T(\pi)} \bigg) \\
&\leq \beta n \bigg( \sup_{\xi \in B_{\frac1m}(\chi)} \langle \tau, \xi \rangle \bigg) + n \frac{1/m}{\gamma} + d^{\gamma}((\vec{0}, 0),(nq, n \chi))
\end{align*}
\end{lemma}

\begin{proof}
Let us look closer at this sum. Observe that we can upper bound the indicator of the set we are summing over by a smooth function:
\[\mathds{1} \bigg\{\frac1{n}\mu_{\pi} \in B_{\frac1m}(\chi) \bigg\} \leq \mbox{exp}\bigg(n\frac{1/m}{\gamma}\bigg) \mbox{exp}\bigg(-\frac{n}{\gamma} \rho\bigg(\frac1{n} \mu_{\pi}, \chi \bigg)\bigg)\]
Thus
\begin{align*}
    &\log \bigg(\sum_{\pi \ \mbox{s.t.} \ \frac1{n}\mu_{\pi} \in B_{\frac1m}(\chi)} e^{\beta T(\pi)} \bigg)  \\
    &\leq \beta \bigg( \sup_{\pi \ \mbox{s.t.} \ \frac1{n}\mu_{\pi} \in B_{\frac1m}(\chi)} T(\pi) \bigg) + \log \bigg( \sum_{\pi \ \mbox{s.t.} \ \frac1{n}\mu_{\pi} \in B_{\frac1m}(\chi)} 1 \bigg) \\
    & \leq \beta n \bigg( \sup_{\pi \ \mbox{s.t.} \ \frac1{n}\mu_{\pi} \in B_{\frac1m}(\chi)} \langle \tau, \frac1{n} \mu_{\pi} \rangle \bigg) + \log  \bigg[ \mbox{exp}\bigg(n\frac{1/m}{\gamma}\bigg) \sum_{\pi \in \mathcal{P}( \vec{0}, \lfloor nq \rfloor)} \mbox{exp}\bigg(-\frac{n}{\gamma} \rho\bigg(\frac1{n} \mu_{\pi}, \chi \bigg)\bigg) \bigg] \\
    &\leq \beta n \bigg( \sup_{\xi \in B_{\frac1m}(\chi)} \langle \tau, \xi \rangle \bigg) + n \frac{1/m}{\gamma} + d^{\gamma}((\vec{0}, 0),(nq, n \chi))
\end{align*}
\end{proof}

\begin{proof}[Proof of Theorem \ref{theorem28}]
We focus on showing (i). The proof of the direction-free case (ii) is analogous. 

(i) Note that if $q = \vec{0}$ then $Z_{n,q}^{\beta} = 1$ hence the variational formula holds trivially. Thus we may assume $||q||_1 > 0$.

Recall that the Levy-Prokhorov metric metrizes weak convergence of measures and that $\{\xi \in \mathcal{M}_+: \frac12 ||q|_1 \leq ||\xi||_{TV} \leq \frac32 ||q||_1\} := S$ is weakly compact. It follows that $\beta \langle \tau, \cdot \rangle$ is uniformly continuous on $S$. More precisely,
\begin{equation} \label{equation18}
\forall \epsilon > 0 \ \exists \delta > 0 \ \mbox{s.t.} \ \forall  \xi, \xi' \in S, \ \rho(\xi, \xi') < \delta \Rightarrow |\beta\langle \tau, \xi \rangle - \beta \langle \tau, \xi' \rangle| < \epsilon
\end{equation}

 We split the argument into two claims.

\textbf{Claim 1}:
\[\lim_{n \rightarrow \infty} \frac1n \log Z_{n,q}^{\beta} \geq \sup_{\nu \in \mathcal{M}_+} [ \beta \langle \tau, \nu \rangle + ||(q, \nu)||] \ \mbox{a.s.} \]

\emph{Proof of Claim 1}\\
Note that $||(q,\nu)|| = -\infty$ unless $\nu \in \mathcal{R}^q$. Also $\mathcal{R}^q$ is weakly closed and $||(q,\nu)||$ is upper semicontinuous in $\nu$ so there exists some $\nu \in \mathcal{R}^q $ that achieves the supremum
\[\sup_{\nu \in \mathcal{M}_+} [ \beta \langle \tau, \nu \rangle + ||(q, \nu)||] \]
We prove the claim when $||(q,\nu)|| > 0$; the $||(q,\nu)||=0$ case is very similar.

For Claim 1, we restrict ourselves to the measure 1 event  
\begin{equation} \label{event1}
\bigg\{\lim_{n \rightarrow \infty} \min_{\pi}^{\lfloor e^{n(||(q, \nu)|| - \epsilon)} \rfloor} \rho \bigg(\frac1{n} \mu_{\pi}, \nu \bigg) = 0 \ \forall \epsilon \in \bQ_+ \ \mbox{with} \ \epsilon < ||(q,\nu)|| \bigg \}
\end{equation}
(this has  measure 1 by Theorem \ref{thmEquiv}). We wish to show that in this event,
\[ \liminf_{n \rightarrow \infty} \frac1n \log Z_{n,q}^{\beta} \geq  \beta \langle \tau, \nu \rangle + ||(q, \nu)|| - 2\epsilon \ \mbox{for $\epsilon > 0$ arbitrarily small}  \]

Fix $0 < \epsilon < ||(q, \nu)||$ in $\bQ_+$ and let $\delta > 0$ satisfy \eqref{equation18}.  We are in the event \eqref{event1} so
\[\lim_{n \rightarrow \infty} \min_{\pi}^{\lfloor e^{n(||(q, \nu)|| - \epsilon)} \rfloor} \rho \bigg(\frac1{n} \mu_{\pi}, \nu \bigg) = 0 \]
Denote by $\pi_n$ the (event-depedent) paths  corresponding to these order statistics. For large enough $n$ we have $\rho (\frac1{n} \mu_{\pi_n}, \nu ) < \delta$ hence there are $\lfloor e^{n(||(q, \nu)|| - \epsilon)} \rfloor$ paths $\pi: \vec{0} \rightarrow \lfloor nq \rfloor$ satisfying 
\begin{equation}\label{equation19}
\rho \bigg(\frac1{n} \mu_{\pi}, \nu \bigg) \leq \rho \bigg(\frac1{n} \mu_{\pi_n}, \nu \bigg)  < \delta \ \mbox{so} \ \bigg|\beta \langle \tau, \frac1n\mu_{\pi} \rangle - \beta \langle \tau, \nu \rangle \bigg|< \epsilon 
\end{equation}
by \eqref{equation18}. But then 
\[ Z_{n,q}^{\beta} \geq \sum_{\mbox{the $\lfloor e^{n(||(q, \nu)|| - \epsilon)} \rfloor$ paths $\pi$ from above}} e^{\beta n \langle \tau, \frac1n \mu_{\pi} \rangle} \geq \lfloor e^{n(||(q, \nu)|| - \epsilon)} \rfloor e^{ n (\beta \langle \tau, \nu \rangle - \epsilon)} \]
by \eqref{equation19}. It follows that
\[ \liminf_{n \rightarrow \infty} \frac1n \log Z_{n,q}^{\beta} \geq  \beta\langle \tau, \nu \rangle - \epsilon + ||(q, \nu)|| - \epsilon\]
Since $\epsilon > 0$ was arbitrarily small, then taking $\epsilon \rightarrow 0^+$ we get
\begin{equation}\label{eqnGT}
\liminf_{n \rightarrow \infty} \frac1n \log Z_{n,q}^{\beta} \geq \sup_{\nu \in \mathcal{R}^q} [\beta \langle \tau, \nu \rangle  + ||(q, \nu)||] \ \mbox{a.s.} \
\end{equation}
as desired. $\square$[Claim 1]

For the second half, we actually show a slightly stronger claim, which we use in a later proof.

\textbf{Claim 2}: Let $W \subseteq \mathcal{M}_+$ be a weakly open, possibly empty set s.t. $\mathcal{R}^q \setminus W \neq \emptyset$. Define
\[ Y_{n,q}^{\beta} := \sum_{\pi: \vec{0} \rightarrow \lfloor nq \rfloor \ \mbox{s.t.} \ \frac1n \mu_{\pi} \notin W} e^{\beta T(\pi)} = \sum_{\pi: \vec{0} \rightarrow \lfloor nq \rfloor \ \mbox{s.t.} \ \frac1n \mu_{\pi} \notin W} e^{\beta n \langle \tau, \frac1n \mu_{\pi} \rangle}\]
Then
\begin{equation} \label{claim}
 \limsup_{n \rightarrow \infty} \frac1n \log Y_{n,q}^{\beta} \leq \sup_{\nu \in \mathcal{M}_+ \setminus W} [ \beta \langle \tau, \nu \rangle + ||(q, \nu)||] \ \mbox{a.s.} 
\end{equation}
Of course, in our case $W = \emptyset$ and $Y_{n,q}^{\beta} = Z_{n,q}^{\beta}$. 

\emph{Proof of Claim 2}\\
We must be extra careful to only specify  countably many measure zero events we exclude. We do this by considering a countable dense family of target measures.

$S \setminus W$ is weakly compact. So for every $m \in \bN$ there is some $M = M(m) \in \bN$ and a finite $\frac1m$-net $B_{\frac1m}(\nu_1^m),\ldots, B_{\frac1m}(\nu_M^m)$ covering $S \setminus W$ with the property that $\nu_i^m \notin W$ (but the balls themselves may intersect $W$). From now on we restrict ourselves to the event 
\begin{equation}\label{event2}
\bigg\{\lim_{n \rightarrow \infty} \frac{d^{\frac1{\sqrt{m}}}((\vec{0}, 0), (nq, n \nu_i^m))}n = \widetilde{d}^{\frac1{\sqrt{m}}}((\vec{0},0), (q, \nu^m_i))  \ \forall m \in \bN, 1 \leq i \leq M(m) \bigg\}
\end{equation}
which is a countable intersection of measure 1 events by Theorem  \ref{thm9} hence has measure 1. It is important to note that for any arbitrary measure $\xi \in \mathcal{M}_+$ we still have 
\[||(q,\xi)|| = \inf_{\epsilon > 0} \widetilde{d}^{\epsilon}((\vec{0},0), (q,\xi))\]
because this is a non-probabilistic statement about constants; we do not however assume that these directed metrics $\widetilde{d}^{\epsilon}((\vec{0},0), (q,\xi))$ are limits of the $\frac{d^{\epsilon}((\vec{0},0),(nq,n\xi))}n$. 

Now we wish to fix an arbitrary $\epsilon > 0$ and prove that in the measure 1 event \eqref{event2},
\[ L:= \limsup_{n \rightarrow \infty} \frac1n \log Y_{n,q}^{\beta} \leq \sup_{\nu \in \mathcal{M}_+ \setminus W} [\beta \langle \tau, \nu \rangle  + ||(q, \nu)||] + 5\epsilon\ \]
It suffices to consider a convergent subsequence $\frac1{n_k} \log Y_{n_k,q}^{\beta}  \rightarrow L$ and show $L \leq $ the right hand side.

Fix any $m \in \bN$. For large $k$,
\[\frac{||\lfloor n_kq \rfloor||_1}{n_k}  \in \bigg[\frac12 ||q||_1, \frac32 ||q||_1 \bigg]\]
so any path $\pi: \vec{0} \rightarrow \lfloor n_kq \rfloor$ we are summing over in $Y_{n_k,q}^{\beta}$ has $\frac1{n_k} \mu_{\pi} \in S \setminus W$ hence $\frac1{n_k} \mu_{\pi} $  is in one of the $B_{\frac1m}(\nu_i^m)$. Therefore
\[ Y_{n_k,q}^{\beta} \leq \sum_{i=1}^{M(m)} \sum_{\pi \ \mbox{s.t.} \ \frac1{n_k}\mu_{\pi} \in B_{\frac1m}(\nu_i^m)} e^{\beta T(\pi)} \leq M(m) \max_{i=1}^{M(m)} \bigg( \sum_{\pi \ \mbox{s.t.} \ \frac1{n_k}\mu_{\pi} \in B_{\frac1m}(\nu_i^m)} e^{\beta T(\pi)} \bigg) \]
Note that $M(m)$ is a constant with respect to $k$, whereas the max above is both event-dependent and $k$-dependent. Thus there is some  event-dependent $1 \leq I(m) \leq M(m)$ and some subsequence $n_{k_j}$ s.t. the above maximum is achieved by the \emph{same} $i=I(m)$ for every $n_k = n_{k_j}$. (Finite pigeonholes, infinite pigeons if you will.) Therefore
\begin{equation}\label{equation20} L = \limsup_{j \rightarrow \infty} \frac1{n_{k_j}} \log Y_{n_{k_j}, q}^{\beta} \leq \limsup_{j \rightarrow \infty} \frac1{n_{k_j}} \log \bigg( \sum_{\pi \ \mbox{s.t.} \ \frac1{n_{k_j}}\mu_{\pi} \in B_{\frac1m}(\nu_{I(m)}^m)} e^{\beta T(\pi)} \bigg) 
\end{equation}
This holds for \emph{any} $m$. Our strategy will be to choose a very large, event-dependent $m$ that gives a nice event-\emph{independent} upper bound.

Let us upper bound the expression in the limit on the right-hand side of \eqref{equation20}. Lemma \ref{lemma32} with $\chi = \nu_{I(m)}^m, \gamma = \frac1{\sqrt{m}}$ establishes that  for any $j \in \bN$,
\begin{equation} \label{Equation28}
\begin{split}
&\log \bigg(\sum_{\pi: \vec{0} \rightarrow \lfloor n_{k_j}q \rfloor \ \mbox{s.t.} \ \frac1{n_{k_j}}\mu_{\pi} \in B_{\frac1m}(\nu_{I(m)}^m)} e^{\beta T(\pi)} \bigg) \\
&\leq \beta n_{k_j} \bigg( \sup_{\xi \in B_{\frac1m}(\nu_{I(m)}^m)} \langle \tau, \xi \rangle \bigg) + n_{k_j} \frac{1/m}{1/\sqrt{m}} + d^{\frac1{\sqrt{m}}}((\vec{0}, 0),(n_{k_j}q, n_{k_j} \nu_{I(m)}^m)))
\end{split}
\end{equation}
This holds for any arbitrary $m \in \bN$. The next step is to specify a sufficiently large $m$.

By compactness of $S \setminus W$, the measures $\nu_{I(m)}^m$ have a weakly convergent subsequence \\ $\nu_{I(m_l)}^{m_l} \Rightarrow \nu' \in S \setminus W$.

Fix some $l_0 \in \bN$ large enough so that 
\begin{equation}\label{bounds}
\frac1{m_{l_0}} + \frac{\epsilon}{\sqrt{m_{l_0}}} < \delta, \ \frac1{\sqrt{m_{l_0}}} < \epsilon, \ \widetilde{d}^{\ \frac1{\sqrt{m_{l_0}}}}((\vec{0},0), (q, \nu'))< ||(q,\nu')||+\epsilon 
\end{equation}
 Consider some event-dependent $l \geq l_0$ large enough s.t.
\begin{equation} \label{bound2}
 \rho(\nu_{I(m_l)}^{m_l}, \nu') < \frac{\epsilon}{\sqrt{m_{l_0}}}
\end{equation}

Since 
\[ \xi \in \overline{B_{\frac1{m_l}}(\nu_{I(m_l)}^{m_l})} \Rightarrow \rho(\xi, \nu') \leq \rho(\xi,\nu_{I(m_l)}^{m_l}) + \rho(\nu_{I(m_l)}^{m_l}, \nu') \leq \frac1{m_{l}} + \frac{\epsilon}{\sqrt{m_{l_0}}}  \leq  \frac1{m_{l_0}} + \frac{\epsilon}{\sqrt{m_{l_0}}}  < \delta \]
and $\delta$ satisfies \eqref{equation18}  then 
\begin{equation}\label{EQ31}
    \sup_{\xi \in B_{\frac1{m_l}}(\nu_{I(m_l)}^{m_l})} \beta \langle \tau, \xi \rangle \leq \sup_{\xi \in \overline{B_{\frac1{m_l}}(\nu_{I(m_l)}^{m_l})}} \beta \langle \tau, \xi \rangle < \beta \langle \tau, \nu' \rangle + \epsilon
\end{equation} 
Next, we have 
\begin{equation}\label{EQ32}
\frac1{\sqrt{m_l}} \leq \frac1{\sqrt{m_{l_0}}} < \epsilon
\end{equation}
Finally, we are in the event \eqref{event2} so for large enough $j$ we have 
\begin{equation} \label{EQ33}
\begin{split}
   & d^{\frac1{\sqrt{m_l}}}((\vec{0}, 0), (n_{k_j} q, n_{k_j} \nu^{m_l}_{I(m_l)})) \\
    & < n_{k_j} \widetilde{d}^{\ \frac1{\sqrt{m_l}}} ((\vec{0}, 0), (q, \nu^{m_l}_{I(m_l)})) + n_{k_j} \epsilon  \\
    & \leq n_{k_j} \widetilde{d}^{\ \frac1{\sqrt{m_{l_0}}}} ((\vec{0}, 0), (q, \nu^{m_l}_{I(m_l)})) + n_{k_j} \epsilon  \\
    & \leq n_{k_j} \widetilde{d}^{\ \frac1{\sqrt{m_{l_0}}}} ((\vec{0}, 0), (q, \nu')) + \frac{n_{k_j}}{1/\sqrt{m_{l_0}}} \rho( \nu^{m_l}_{I(m_l)}, \nu') + n_{k_j} \epsilon
   \end{split}
\end{equation}
by Lemma \ref{lemma12}. By \eqref{bounds},\eqref{bound2} this last expression is
    \[ < n_{k_j} ||(q,\nu')||  + 3n_{k_j} \epsilon  \]
Combining the inequalities \eqref{EQ31},\eqref{EQ32},\eqref{EQ33} with \eqref{Equation28} we get that for large enough $j$,
\begin{equation}\label{32}
\log \bigg(\sum_{\pi: \vec{0} \rightarrow \lfloor n_{k_j}q \rfloor \ \mbox{s.t.} \ \frac1{n_{k_j}}\mu_{\pi} \in B_{\frac1{m_l}}(\nu_{I(m_l)}^{m_l})} e^{\beta T(\pi)} \bigg) \leq n_{k_j} \bigg( \beta \langle \tau, \nu' \rangle  + ||(q,\nu')||  + 5\epsilon \bigg)
\end{equation}

Substituting this in \eqref{equation20}, we get
\[ L \leq  \beta \langle \tau, \nu' \rangle  + ||(q,\nu')||  + 5\epsilon \leq \sup_{\nu \in \mathcal{M}_+ \setminus W}  \beta \langle \tau, \nu \rangle  + ||(q,\nu)||  + 5\epsilon\]
as desired. Note that even though $m_l$ and $\nu'$ in \eqref{32} are event-dependent, the upper bound we get for $L$ is \emph{not} event-dependent, makes our argument work. \ $\square$[Claim 2]

This finishes the proof of the variational formula in the fixed-direction case.

(ii) An analogous argument gives
\[G^{\beta} = \sup_{\nu \in \mathcal{M}_1} [\beta \langle \tau, \nu \rangle + ||\nu||] \]
But direction-free grid entropy is just grid entropy in the direction $\ell$. Furthermore, \\ $||(\ell, \nu)|| = -\infty$ if $\nu \notin \mathcal{M}_1$. Thus
\begin{align*}
    G^{\beta} &= \sup_{\nu \in \mathcal{M}_1} [\beta \langle \tau, \nu \rangle + ||\nu||] \\
    &= \sup_{\nu \in \mathcal{M}_+} [\beta \langle \tau, \nu \rangle + ||(\ell,\nu)||]\\
    &= G^{\beta}_{\ell}\\
    &= \sup_{q \in \bR^D_{\geq 0}, ||q||_1=1}\sup_{\nu \in \mathcal{M}_+} [\beta \langle \tau, \nu \rangle + ||(q,\nu)||] \ \mbox{by Lemma \ref{maximalLemma}}\\
    &= \sup_{q \in \bR^D_{\geq 0}, ||q||_1=1} G^{\beta}_{q}
\end{align*}
\end{proof}

As an immediate corollary, we get another proof of Bates' Variational Formula (Theorem \ref{thm27}), as well as its direction-free analogue. We stress that our proof as written works only under the assumption that $\tau$ is bounded, but can easily be extended to the general case of measurable $\tau$.

\textit{Proof of Theorem \ref{thm27}.} Fix $q \in \bR^D_{\geq 0}$ and bounded measurable $\tau$. It is standard (e.g. see \cite{georgiou}) that  for every $n$,
\begin{equation}\label{label25} 
\lim_{\beta \rightarrow \infty} \beta^{-1}  \log Z_{n,q}^{\beta} = Z_{n,q} = \sup_{\beta > 0} \beta^{-1}  \log Z_{n,q}^{\beta} = T(\vec{0}, nq) := \log Z_{n,q}^{\infty} \ \mbox{a.s.}
\end{equation}
which we recall  is the last passage time between $\vec{0}$ and $\lfloor nq \rfloor$. That is, the last passage time is the zero temperature analogue of $\log Z_{n,q}^{\beta}$. Thus the  the zero temperature point-to-point Gibbs Free Energy  is the last passage time constant $\lambda_q$:
\[ G^{\infty}_q := \lim_{n \rightarrow \infty} n^{-1} \log Z_{n,q}^{\infty} = \lim_{n \rightarrow \infty}  n^{-1} T(\vec{0}, nq) = \lambda_q \ \mbox{a.s.}\]

On the other hand, another standard result (see \cite{georgiou}) says that the zero temperature point-to-point Gibbs Free Energy is a scaled limit of the positive temperature free energies:
\begin{equation}\label{label27}
    \lim_{\beta \rightarrow \infty} \beta^{-1} G^{\beta}_q = \sup_{\beta > 0} \beta^{-1} G^{\beta}_q = G^{\infty}_q = \lambda_q \ \mbox{a.s}
\end{equation}
Fix $N \in \bN$. Dividing our variational formula by $\beta$ and taking the supremum over $\beta \geq N$  we get by \eqref{label27}
\begin{align*}
    \lambda_q &=     \sup_{\beta \geq N} \beta^{-1} G^{\beta}_q \\
    &= \sup_{\beta \geq N} \sup_{\nu \in \mathcal{R}^q} \bigg(\langle \tau, \nu \rangle + \beta^{-1} ||(q,\nu)|| \bigg) \\
    &= \sup_{\nu \in \mathcal{R}^q} \sup_{\beta \geq N}  \bigg(\langle \tau, \nu \rangle + \beta^{-1} ||(q,\nu)|| \bigg) \\
    &= \sup_{\nu \in \mathcal{R}^q} \bigg(\langle \tau, \nu \rangle + N^{-1} ||(q,\nu)|| \bigg) \\
\end{align*}
Taking $N \rightarrow \infty$ and noting that $\nu \in \mathcal{R}^q$ implies $|(q, \nu)|| \in [0, H(q)]$ is bounded, we get Bates' variational formula:
\[ \lambda_q = \sup_{\nu \in \mathcal{R}^q} \langle \tau, \nu \rangle \]
a.s.. \ $\square$

In other words, Bates' variational formula is nothing more than the zero temperature analogue of the entropy variational formula for Gibbs Free Energy.

The other corollary of note is the answer to Hoffman's question in the directed polymer model when our variational formula has a unique maximizer. As Bates observes, this happens for a dense family of measurable functions $\tau$ (recall Remark \ref{thm27REM}).

\begin{corollary}
Fix an inverse temperature $\beta > 0$ and a bounded measurable \\ $\tau: [0,1] \rightarrow \bR$. 
\begin{enumerate}[label=(\roman*)]
\item Fix $q \in \bR^D_{\geq 0}$ and suppose $\beta \langle \tau, \nu \rangle + ||(q,\nu)||$ has a unique maximizer $\nu \in \mathcal{R}^q$. For every $n$ pick a path $\pi_n: \vec{0} \rightarrow \lfloor nq \rfloor$ independently and at random according to the probabilities prescribed the corresponding point-to-point $\beta$-polymer measure $\rho_{n,q}^{\beta}$. Then the empirical measures $\frac1n \mu_{\pi_n}$ converge weakly to $\nu$ a.s.
\item Suppose $\beta \langle \tau, \nu \rangle + ||\nu||$ has a unique maximizer $\nu \in \mathcal{M}_1$. For every $n$ pick a length $n$ path $\pi_n$ from $\vec{0}$ independently and at random according to the probabilities prescribed the corresponding point-to-level $\beta$-polymer measure $\rho_{n}^{\beta}$. Then the empirical measures $\frac1n \mu_{\pi_n}$ converge weakly to $\nu$ a.s.
\end{enumerate}
\end{corollary}

\begin{remark}
In fact, it follows from our argument below combined with the compactness of $\mathcal{M}_1$ that even in the general case when the variational formula might not have a unique maximizer, all of the accumulation points of the empirical measures along random paths chosen according to the $\beta$-polymer measure are maximizers of the variational formula. Thus, one interesting open question for the future is whether we can more precisely describe these maximizers. 
This information could help distinguish between weak and strong disorder regimes; see \cite{rassoul2017variational} for a recent paper on the role of minimizers of a "cocycle" variational formula for the Gibbs Free Energy in determining disorder regimes.
\end{remark}

\begin{remark}
This partial answer to Hoffman's question in the positive temperature case does also follow from the development of grid entropy by Rassoul-Agha and Sepp{\"a}l{\"a}inen \cite{rassoul2014quenched} as the rate function of the LDP of empirical measures. However, it seems it has not yet been formulated explicitly as we have done here.
\end{remark}

\begin{proof}
We focus on proving (i). The direction-free argument (ii) is analogous.

(i) If $q = \vec{0}$ then the paths $\pi_n$ are the trivial path $\vec{0} \rightarrow \vec{0}$  and $\mathcal{R}^{\vec{0}} = \{0\}$ so $\nu = 0$, hence the result holds trivially. Thus we may assume $||q||_1 > 0$.

Since the empirical measures $\frac1n \mu_{\pi^n}$ eventually lie in the weakly compact set \\ $S = \{\xi \in \mathcal{M}_+: \frac12 ||q||_1 \leq ||\xi||_{TV} \leq \frac32 ||q||_1\}$ then by compactness the existence of an accumulation point is guaranteed so it suffices to show that the sequence $\frac1n \mu_{\pi_n}$ cannot have any accumulation point other than $\nu$.

If $\mathcal{R}^q$ is a singleton, i.e. $\mathcal{R}^q = \{\nu\} = \{||q||_1 \Lambda\}$, then by the characterization of $\mathcal{R}^q $ as the set of accumulation points of empirical measures (see Theorem \ref{thm5Bates}), $\nu$ is the only possible accumulation point of $\frac1n \mu_{\pi_n}$ so we would be done. We can thus assume $\mathcal{R}^q$ is not a singleton. Thus there exists some $\eta_0 > 0$ s.t. $B_{\eta_0}(\nu)^C \cap \mathcal{R}^q \neq \emptyset$.

Suppose we are in the measure 1 event 
\begin{equation} \label{Event3}
\begin{split}
    &\Bigg\{\begin{array}{l}
    \limsup \limits_{n \rightarrow \infty} \frac1n \log \ \sum \limits_{\substack{\pi: \vec{0} \rightarrow \lfloor nq \rfloor \\  \frac1n \mu_{\pi} \notin B_{\eta}(\nu)}} e^{\beta T(\pi)} \leq \sup \limits_{\xi \in \mathcal{M}_+ \setminus B_{\eta}(\nu)} [\beta \langle \tau, \xi \rangle + ||(q,\xi)|| ] \ \forall \eta \in \bQ_+ \cap (0, \eta_0) 
  \end{array} \Bigg\}  \\
&\cap \bigg\{\lim_{n \rightarrow \infty} \frac1n \log Z_{n,q}^{\beta} = \beta \langle \tau, \nu \rangle + ||(q, \nu)|| \bigg\}
\end{split}
\end{equation}
This is the intersection of countably many measure 1 events by Claim 2 in the proof of Theorem \ref{theorem28}. Claim 2 applies because  
\[B_{\eta}(\nu)^C \cap \mathcal{R}^q \supseteq B_{\eta_0}(\nu)^C \cap \mathcal{R}^q \neq \emptyset \ \forall 0 < \eta < \eta_0\]

We need to prove that in the measure 1 event \eqref{Event3}, the probability that the empirical measures $\frac1n \mu_{\pi^n}$ of the paths chosen independently at random according to the polymer measures do not have an accumulation point other than $\nu$ is 0. 

Fix any $0 < \eta < \eta_0$ in $\bQ_+$. We wish to show that the probability that  $\frac1n \mu_{\pi_n}$ is not in $B_{\eta}(\nu)^C$ infinitely often is 0. 

Let $\xi_{\eta} \in \mathcal{M}_+ \setminus B_{\eta}(\nu)$ achieve the supremum
\[\sup_{\xi \in \mathcal{M}_+ \setminus B_{\eta}(\nu)} [\beta \langle \tau, \xi \rangle + ||(q,\xi)|| ] \]
By the maximality of $\nu$, there is some $\epsilon > 0$ s.t.
\begin{equation}\label{eqn38}
\beta \langle \tau, \xi_{\eta} \rangle + ||(q,\xi_{\eta})|| < \beta \langle \tau, \nu \rangle + ||(q,\nu)|| - 3\epsilon  
\end{equation}

Consider any $n \in \bN$. Recall the definition of the point-to-point $\beta$-polymer measure $\rho_{n,q}^{\beta}$ on the set of paths $\pi: \vec{0} \rightarrow \lfloor nq \rfloor$:
 \[ \rho_{n,q}^{\beta} (d\pi) = \frac1{Z_{n,q}^{\beta}} e^{\beta T(\pi)} \]
Thus the probability that $\frac1n \mu_{\pi_n} \notin B_{\eta}(\nu)$ (with respect to the $\beta$-polymer measure $\rho_{n,q}^{\beta}$) is 
\[ \frac1{Z_{n,q}^{\beta}} \sum_{\pi: \vec{0} \rightarrow \lfloor nq \rfloor \ \mbox{s.t.} \ \frac1n \mu_{\pi} \notin B_{\eta}(\nu)} e^{\beta T(\pi)}   \]
We are in the event \eqref{Event3}, so for large enough $n$, this is
\[ \leq \mbox{exp} \bigg(-n\beta \langle \tau, \nu \rangle -n||(q,\nu)|| + n\epsilon \bigg) \mbox{exp} \bigg(n\beta \langle \tau, \xi_{\eta} \rangle + n||(q,\xi_{\eta})|| + n\epsilon \bigg) \leq e^{-n\epsilon}\]
by \eqref{eqn38}. $\sum \limits_{n =1}^{\infty} e^{-n\epsilon} < \infty$ so by the Borel-Cantelli Lemma,
\begin{align*}
&P\bigg(\frac1n \mu_{\pi_n} \ \mbox{has an accumulation point in $B_{2\eta}(\nu)^C$}\bigg) \\
&\leq P\bigg(\frac1{n} \mu_{\pi_{n}}\in B_{\eta}(\nu)^C \ \mbox{for infinitely many $n$}\bigg) = 0
\end{align*}

By continuity from below, 
\begin{align*} 
&P\bigg(\frac1n \mu_{\pi_n} \ \mbox{has an acc. point other than $\nu$} \bigg) \\
&= \lim_{\eta \rightarrow 0^+}P \bigg(\frac1n \mu_{\pi_n} \ \mbox{has an acc. point in $B_{2\eta}(\nu)^C$} \bigg) =0 
\end{align*}
Therefore the empirical measures $\frac1n \mu_{\pi_n} $ converge weakly to $\nu$ a.s..
\end{proof}

\subsection{Grid Entropy as  The Negative Convex Conjugate of  Gibbs Free Energy}
Another way of viewing the variational formulas in Theorem \ref{theorem28} is that the point-to-point/point-to-level $\beta$-Gibbs Free Energies as  functions of the bounded measurable function $\tau$ are the convex conjugates of the functions $\nu \mapsto -||(q, \nu)||, \nu \mapsto -||\nu||$  on $\mathcal{M}_+, \mathcal{M}_1$ respectively.

Let us briefly recall what that means (see \cite{zalinescu} for more details).

\begin{definition}
Let $X$ be a locally convex Hausdorff space and let $X^{\ast}$ be its dual with respect to an inner product $\langle \cdot, \cdot \rangle$. A convex function $f: X \rightarrow [-\infty, \infty]$ is said to be proper if $f > -\infty$ and $f \not \equiv \infty$. For any proper convex function $f: X \rightarrow [-\infty, \infty]$, its convex conjugate is the function $f^{\ast}: X^{\ast} \rightarrow [-\infty, \infty] $ given by
\[ f^{\ast}(x^{\ast}) = \sup_{x \in X} [\langle x, x^{\ast} \rangle - f(x)]\]
\end{definition}

In our case, we have $X = \mathcal{M}_+$ (finite Borel measures on [0,1]) and $X^{\ast}$ is the set of bounded measurable functions $\tau: [0,1] \rightarrow \bR$ with inner product given by integration.

Also, recall that direction-fixed/direction-free grid entropy is concave  hence the maps $\nu \mapsto -||(q, \nu)||, \nu \mapsto -||\nu||$ are convex and trivially proper. Furthermore, by our continuity theorem (Thm \ref{thm20}), $-||(q, \cdot)||$ and $-||\nu||$ are lower semicontinuous in $\nu$.

With this setup in mind, it is evident that Theorem \ref{theorem28} establishes that the point-to-point $\beta$-Gibbs Free Energy is precisely the convex conjugate of $-||(q, \cdot)||$ and the point-to-level $\beta$-Gibbs Free Energy is the convex conjugate of $-||\nu||$.

The convex conjugate has various important properties. We focus on those that are the most relevant and interesting in regards to grid entropy and Gibbs Free Energy. See \cite[Sect 2.3]{zalinescu} for the full list of properties.

\begin{corollary}\label{biconj}
Fix $q \in \bR^D_{\geq 0}$, an inverse temperature $\beta > 0$. Then viewing Gibbs Free Energies as functions of bounded measurable $\tau$ we have the following.
\begin{enumerate}[label=(\roman*)]
\item Gibbs Free Energies $G_{q}^{\beta}(\tau), G^{\beta}(\tau)$ are convex and lower semicontinuous in $\tau$ 
\item (Order-preservation) If $q' \in \bR^D_{\geq 0}$ with $||(q, \cdot)|| \leq ||(q', \cdot)||$ then \\ $G_q^{\beta}(\cdot) \leq G_{q'}^{\beta}(\cdot)$
\item (Biconjugate Duality) The convex conjugate of $\beta$-Gibbs Free Energy (as a function of $\tau$) is minus grid entropy. That is,
\begin{equation*} (G_q^{\beta})^{\ast} = -||(q, \cdot)|| \ \mbox{and} \ (G^{\beta})^{\ast} = -|| \cdot|| 
\end{equation*}
In other words, grid entropy is equal to its biconjugate.
\item (Fenchel's Duality Theorem) Let $g: \mathcal{M}_+ \rightarrow [-\infty, \infty], h: \mathcal{M}_1 \rightarrow [-\infty, \infty]$ be  proper convex functions s.t.  $\exists \nu \in \mathcal{R}^q, \nu' \in \mathcal{R}^1$ with $g(\nu) < \infty, h(\nu') < \infty$ and one of $g, ||(q, \cdot)||$ is continuous at $\nu$ and one of $h, ||\cdot||$ is continuous at $\nu'$. Then
\[ \inf_{\nu \in \mathcal{M}_+ } [g(\nu) - ||(q, \nu)||] = \sup_{\tau} [G_q^{\beta}(\tau) - g^{\ast}(\tau) ]\]
\[ \inf_{\nu \in \mathcal{M}_1 } [h(\nu) - ||\nu||] = \sup_{\tau} [G^{\beta}(\tau) - h^{\ast}(\tau) ]\]
where $g^{\beta \ast}, h^{\beta \ast}$ denote the $\beta$-convex conjugates of $g,h$ respectively:
\[ g^{\beta \ast}(\tau) = \sup_{\nu \in \mathcal{M}_+} [\beta \langle \tau, \nu \rangle - g(\nu) ], \ h^{\beta \ast}(\tau) = \sup_{\nu \in \mathcal{M}_+} [\beta \langle \tau, \nu \rangle - h(\nu) ] \]
\end{enumerate}
\end{corollary}

\begin{remark}
(iii) gives yet another definition of direction-fixed/direction-free grid entropy. This one is perhaps the most remarkable of the three, as it connects it to the seemingly unrelated quantity that is Gibbs Free Energy. This more than justifies the importance and canonical nature of grid entropy.
\end{remark}

\begin{proof}
(iii) is an immediate consequence of the Fenchel-Moreau Theorem. The rest follow from the properties of convex conjugates of proper convex, lower semicontinuous functions. See \cite[Sect 2.3]{zalinescu} for more details on (i), (ii) and see \cite[Sect 7.15]{van} for more details on (iv).
\end{proof}

The key takeaway is this: grid entropy and  Gibbs Free Energy are intricately connected via convex duality.

\subsection{Connection to Previous Literature on Grid Entropy} \label{translationSection}
We end our discussion of these variational formulas by explicitly deriving the formulas \eqref{translationEqn} which link our notion of grid entropy to that presented previously in \cite{rassoul2014quenched}. For the sake of conciseness we focus on the direction-fixed (a.k.a. point-to-point) case.

Let us first recalibrate our setting to align with the framework of Rassoul-Agha and \\ Sepp{\"a}l{\"a}inen. We view each edge label as a value $\omega_{x,z}$ associated to the "anchor" vertex $x \in \bZ^D$ and the unit $NE$ direction $z$ of the step (with respect to $x$). Thus the environment becomes $\Omega := [0,1]^{\bZ^D \times  \mathcal{G}}$ where $\mathcal{G} = \{e_1, \ldots, e_D\} $ is the set of $D$ unit NE steps. For each $y \in \bZ^D$, we define a shift $T_y: \Omega \rightarrow \Omega$ in the natural way, by shifting the anchor: $(T_y \omega)_{x,z} = \omega_{x+y,z}$. We let $f: \Omega \times \mathcal{G} \rightarrow \bR$ be a potential.

We initially work on the space $\Omega \times \mathcal{G}$ of pairs of the environment and one unit step; that is we consider the empirical measures 
\[ \frac1n \chi_{\pi_n}  = \frac1n \sum_{i=0}^{n-1} \delta_{(T_{X_i}\omega, X_{i+1}-X_i)}\]
of the polymer paths $\pi_n: (X_0 \rightarrow X_1 \rightarrow \cdots \rightarrow X_{n-1})$. It is established in \cite[Theorem 3.1]{rassoul2013quenched} that the rate function $I_1(\mu, q)$ of these empirical measures in the direction-$q$ case is precisely the lower semicontinuous regularization of
\begin{equation} \label{i1} 
H_1(\mu) - \beta \langle f, \mu \rangle + \Lambda_q^{\beta}(f) 
\end{equation}
where $H_1$ is the infimum of certain relative entropies (see (5.2) of \cite{rassoul2013quenched}), and  
\[ \Lambda_q^{\beta}(f) := \lim_{n \rightarrow \infty} \frac1n \log E \bigg[\mbox{exp} \bigg(\beta \sum \limits_{i=0}^{n-1} f(T_{X_i}\omega, X_{i+1}-X_i)\bigg) \textbf{1}_{X_{n-1} = \lfloor n q \rfloor}\bigg] \]
is the (normalized) $\beta$-Gibbs Free Energy in direction $q$ with potential $f$ defined in (2.3) of the same paper. Here $\mu$ is a probability measure on $\Omega \times \mathcal{G}$ s.t. the 1st step $X_1-X_0 \in \mathcal{G}$ has $\mu$-mean $E^{\mu}[X_1-X_0]=q$. This last condition is imposed because these paths go in the direction $q$ hence any limit points of their empirical measures must average to $q$ in the step coordinate; any $\mu$ not satisfying this condition will trivially have an infinite rate function.

We consider the "edge label potential"  $\phi: \Omega \times \mathcal{G} \rightarrow [0,1]$ given by $\phi(\omega, z) = \omega_{0, z}$ which maps an (environment, unit direction) pair to the edge label of the corresponding unit direction anchored at the origin. If $\pi_n = (X_0 \rightarrow \cdots \rightarrow X_{n-1})$ is a polymer path then our empirical measures from this paper are just the $\phi$-pushforwards of $\frac1n \chi_{\pi_n}$:
\[\frac1n \mu_{\pi_n} = \frac1n \sum_{i=0}^{n-1} \delta_{\omega_{X_i,X_{i+1}-X_i}} = (\phi)_{\ast}\bigg(\frac1n \chi_{\pi_n} \bigg)\]
We obtain the rate function $I_1'(\nu,q)$ of the empirical measures $\frac1n \mu_{\pi_n}$ in the direction-$q$ case by contracting the higher level LDP rate function $I_1(\mu,q)$:
\begin{equation}\label{contraction} 
 I_1'(\nu, q) = \inf_{\mu: (\phi)_{\ast}\mu = \nu} I_1(\mu,q)
 \end{equation}
for probability measures $\nu \in \mathcal{M}_1$, where the infimum is taken over probability measures $\mu$ on $\Omega \times \mathcal{G}$  whose $\phi$-pushforward is $\nu$.

We now make some convenient observations. First, for any measurable and bounded \\ $\tau: [0,1] \rightarrow \bR$, our (unnormalized) $\beta$-point-to-point Gibbs Free Energy $G_q^{\beta}(\tau)$ is precisely $\Lambda_q^{\beta}(\tau \circ \phi) + \log D$. Second, for any $\mu$ s.t. $\nu = (\phi)_{\ast}\mu$ we have 
\[\langle \phi, \mu \rangle = \int \phi(u) d\mu = \int u d\nu = \langle \iota, \nu \rangle \]
where $\iota: [0,1] \rightarrow [0,1], \iota(u) = u$ is the identity function. Thus the latter two terms in \eqref{i1} for $I_1(\mu,q)$ are ignored by the infimum in \eqref{contraction} so we get that $I_1'(\nu,q)$ is the lower semicontinuous regularization of
\[ H_1'(\nu,q) - \beta \langle \iota, \nu \rangle + G_q^{\beta} (\iota) - \log D  \]
where
\begin{equation} \label{e0}
H_1'(\nu,q) = \inf \big\{H_1(\mu): \mu \ \mbox{is a prob. meas. on} \ \Omega \times \mathcal{G} \ \mbox{s.t.} \ (\phi)_{\ast}\mu = \nu, E^{\mu}[X_1-X_0]=q \big\}
\end{equation}

 (5.9) from \cite{rassoul2014quenched} gives the following variational formula for the point-to-point Gibbs Free Energy:
 \begin{equation} \label{e1}
 \Lambda_q^{\beta}(f)  = \sup_{\mu: E^{\mu}[X_1-X_0]=q} [\beta \langle f, \mu \rangle - H_1(\mu)]
 \end{equation}
for bounded measurable potentials $f: \Omega \times \mathcal{G} \rightarrow \bR$ where the supremum is over probability measures $\mu$ on $\Omega \times \mathcal{G}$ with mean step $q$.

%On the other hand, (5.6) and (5.7) from the same paper yield the following variational formula for the entropy $H_1(\mu)$:
%\begin{equation} \label{e2}
%H_1(\mu) = \sup_f [ \beta \langle f, \mu \rangle - \Lambda_q^{\beta}(f) ]
%\end{equation}
%where the supremum is over bounded potentials $f: \Omega \times \mathcal{G} \rightarrow \bR$. 
Taking $f = \tau \circ \phi$ for bounded measurable $\tau: [0,1] \rightarrow \bR$ and splitting the supremum over $\mu$ into an outer supremum over $\nu \in \mathcal{M}_1$ and an inner supremum over  $\mu$ whose $\phi$-pushforward is $\nu$, we get
 \begin{equation} \label{e3}
 \begin{split}
 \Lambda_q^{\beta}(\tau \circ \phi)  &= \sup_{\mu: E^{\mu}[X_1-X_0]=q} [\beta \langle \tau \circ \phi, \mu \rangle - H_1(\mu)]\\
 & = \sup_{\nu \in \mathcal{M}_1} \sup_{\mu: (\phi)_{\ast}\mu = \nu, E^{\mu}[X_1-X_0]=q} [\beta \langle \tau, \nu \rangle - H_1(\mu)]\\
 &= \sup_{\nu \in \mathcal{M}_1} [\beta \langle \tau, \nu \rangle - H_1'(\nu, q)] 
 \end{split}
 \end{equation}
 where the last equality holds by \eqref{e0}. Comparing this to our variational formula for the point-to-point Gibbs Free Energy from Theorem \ref{theorem28} (i), recalling that Gibbs Free Energy is biconjugate from Corollary \ref{biconj}, and recalling that $G_q^{\beta}(\tau)=\Lambda_q^{\beta}(\tau \circ \phi) + \log D$, we get
\[ H_1'(\nu,q) = \sup_{\tau: [0,1] \rightarrow \bR \ \mbox{bdd, meas.}} [\beta \langle \tau,\nu \rangle - G_q^{\beta}(\tau) + \log D ] = -||(q,\nu)|| +\log D \]
\bigskip

The direction-free (a.k.a. point-to-level) argument is nearly identical, except we drop the $q$ indices and the condition that the first step $X_1-X_0$ has $\mu$-mean $E^{\mu}[X_1-X_0] = q$. In the end we get
\[ H_1'(\nu) = -||\nu|| + \log D\]
where 
\begin{equation*}
H_1'(\nu) = \inf_{\mu: (\phi)_{\ast}\mu = \nu} H_1(\mu) 
\end{equation*}

As a final remark, we mention that this entropy $H_1$ originally appears in Varadhan's paper \cite{varadhan2003large} on large deviations for empirical measures of Markov chains.

\section{Extensions to Other Models and Open Questions}
Looking back at our development of grid entropy in $\bR^D$, we never required some specific property of the lattice model other than that the number of  paths $\pi: \vec{0} \rightarrow \lfloor nq \rfloor$ we consider is $O(e^{H(q)n})$ for some  model-dependent constant 
\[H(q) = \limsup_{n \rightarrow \infty} \frac1n \log (\# \pi: \vec{0} \rightarrow \lfloor nq \rfloor)\]
 that all such $\pi$ have the same length $O(n)$, and that any pair of paths $\vec{0} \rightarrow \lfloor mq\rfloor$ and $\lfloor mq \rfloor \rightarrow \lfloor (m+n)q \rfloor $ can be concatenated (thus giving us the superadditivity required to be able to apply the Subadditive Ergodic Theorem). 
 
 Therefore  the notion of grid entropy and all our results, including the variational formula for the Gibbs Free Energy, apply in any edge model with these properties. For example,  we might consider random walks on $\bZ_{+} \times \bZ^D$ where the first coordinate, "time," is always increasing along a path.

In addition, if we loosen some of these conditions, then a part of our work in this paper still holds. For example,  instead of NE paths on $\bZ^D$  consider self-avoiding walks (SAWs) on $\bZ^D$, as done in \cite{bates}. Then we cannot always concatenate a pair of SAWs $\vec{0} \rightarrow \lfloor mq\rfloor$ and $\lfloor mq \rfloor \rightarrow \lfloor (m+n)q \rfloor $ so our arguments that rely on superadditivity no longer apply. However, we may still define grid entropy by
\[ ||(q, \nu)|| := \sup \bigg\{\alpha \geq 0: \lim_{n \rightarrow \infty} \min_{\pi: \vec{0} \rightarrow \lfloor nq \rfloor}^{\lfloor e^{\alpha n} \rfloor}  \rho\bigg(\frac1{n} \mu_{\pi}, \nu \bigg) = 0  \ \mbox{a.s.} \bigg\}\]
and it will be true that only  target measures $\nu \ll \Lambda$ with total mass $||q||_1$ are observed and that we have an upper bound on the sum of relative entropy and grid entropy:
\[ D_{KL}(\nu || \Lambda) +||(q,\nu)|| \leq \chi_D \ \mbox{when $||q||_1=1$} \]
where $\chi_D$ is the connectivity constant for $\bZ^D$ satisfying 
\[\#(\mbox{length $n$ SAWs from $\vec{0}$})= e^{n \chi_D +o(n)}\]
 (see \cite[Section 6.2]{lawler} for details). It is not clear, however, whether any of the other properties (including this being a directed norm with negative sign) still hold.

Note however that with a small modification we can make this work. All our results hold for SAWs on $\bZ^{D+1}$ where the first coordinate, "time," is necessarily non-decreasing because we obtain a superadditivity akin to the one for NE paths.

\bigskip

We end by mentioning two open questions that could be the scope of future work. Can we describe the set of maximizers of our variational formula (Theorem \ref{theorem28})? As discussed before, this would provide insights into the limit points of empirical measures and also into the problem of distinguishing weak and strong order regimes. And secondly, is grid entropy strictly concave? From our results in this chapter, answering this would immediately imply that Gibbs Free Energy is strictly convex, which is currently a major open research problem in this field. 

\section{Acknowledgments}
I would like to thank my friends and family for their continual support during my graduate studies. Special thanks goes to my advisor B\'{a}lint Vir\'{a}g, for his patience, guidance, and optimism in the face of setbacks. I also wish to thank Duncan Dauvergne for very helpful discussions and peer review on a preliminary draft. In addition, I am grateful to my thesis external reviewer Firas Rassoul-Agha for his comments and for drawing some connections to previous literature. Finally, this work would not have been possible without the funding from my NSERC Canadian Graduate Scholarship-Doctoral.

\bibliography{biblio}
  
\end{document}